\documentclass[11pt]{amsart}
\usepackage{amsmath}
\usepackage{amssymb}
\usepackage{amsthm}
\usepackage[all]{xy}
\usepackage{graphicx}
\usepackage[margin=1in]{geometry}
\usepackage{color}
\usepackage{hyperref}
\usepackage{verbatim}
\usepackage{pb-diagram,pb-xy}

\newcommand{\bpsi}{\boldsymbol\psi}
\newcommand{\bPsi}{\boldsymbol\Psi}
\newtheorem*{thmA}{Theorem A}
\newtheorem*{thmA1}{Theorem A*}
\newtheorem*{thmB}{Corollary B}
\newtheorem*{thmB1}{Corollary B*}
\newtheorem*{thmC}{Corollary C}
\newtheorem*{thmD}{Corollary D}
\newtheorem*{thmE}{Theorem E}
\numberwithin{equation}{section}

\newcommand{\R}{\ensuremath{\mathbb{R}}}
\newcommand{\N}{\ensuremath{\mathbb{N}}}

\newcommand{\Z}{\ensuremath{\mathbb{Z}}}
\newcommand{\Q}{\ensuremath{\mathbb{Q}}}

\newtheorem{theorem}{Theorem}[section]
\newtheorem{corollary}{Corollary}[theorem]
\newtheorem{proposition}[theorem]{Proposition}
\newtheorem{lemma}[theorem]{Lemma}
\newtheorem{claim}[theorem]{Claim}

\theoremstyle{definition} % makes body in upright roman (not italics)
\newtheorem{defn}[theorem]{Definition}
\newtheorem{remark}[theorem]{Remark}

\newtheorem{Construction}{Construction}[section]
\newtheorem{example}{Example}[section]

\newcommand{\MT}{{\mathsf M}{\mathsf T}}

\newcommand{\mb}[1]{\mathbf{#1}}

\newcommand{\Cob}{{\mathbf{Cob}}}

\DeclareMathOperator*{\Cofibre}{Cofibre}

\DeclareMathOperator*{\colim}{colim\ }

\DeclareMathOperator*{\Int}{Int}

\DeclareMathOperator*{\Diff}{Diff}

\DeclareMathOperator*{\BDiff}{BDiff}

\DeclareMathOperator*{\Maps}{Maps}

\DeclareMathOperator*{\Image}{Im}

\DeclareMathOperator*{\proj}{pr}

\DeclareMathOperator*{\Id}{\text{Id}}
\DeclareMathOperator*{\ind}{\text{ind}}
\DeclareMathOperator*{\Bun}{\mathrm{Bun}}

\DeclareMathOperator*{\Ob}{\text{Ob}}
\DeclareMathOperator*{\Mor}{\mathrm{Mor}}

\DeclareMathOperator*{\Cl}{\text{Cl}}

\DeclareMathOperator*{\Ker}{Ker}

\author{Boris Botvinnik}
\author{Nathan Perlmutter}

\address{Department of Mathematics, University of Oregon, Eugene, OR,
  97403, USA}
\email{botvinn@uoregon.edu}

\address{Stanford University Department of Mathematics, Building 380, Stanford, California,  94305, USA}

\email{nperlmut@stanford.edu}

\title[Stable Moduli Spaces of High Dimensional Handlebodies]{Stable
  Moduli Spaces of High Dimensional Handlebodies}

\begin{document}
\maketitle

\begin{abstract}
We study the moduli space of handlebodies diffeomorphic to
$(D^{n+1}\times S^n)^{\natural g}$, i.e. the classifying space
$\BDiff((D^{n+1}\times S^n)^{\natural g}, D^{2n})$ of the group of
diffeomorphisms that restrict to the identity near an embedded disk
$D^{2n} \subset \partial (D^{n+1}\times S^n)^{\natural g}$. We prove
that there is a natural map
$$\colim_{g\to\infty}\BDiff((D^{n+1}\times
S^n)^{\natural g}, D^{2n}) \;
\longrightarrow \; Q_{0}BO(2n+1)\langle n \rangle_{+}$$ 
which induces an isomorphism in integral homology when $n\geq 4$. Above, 
$BO(2n+1)\langle n \rangle$ denotes the $n$-connective cover of
$BO(2n+1)$. 
\end{abstract}

\setcounter{tocdepth}{1}
\tableofcontents
\vspace*{-5mm}
\section{Introduction}
\subsection{Motivation}
This work is motivated by two recent results, one by Galatius and
Randal-Williams \textcolor{black}{on the moduli} spaces of
high-dimensional manifolds, \cite{GRW 14}, and the second one by
Hatcher on \textcolor{black}{the moduli spaces} of three-dimensional
handlebodies, \cite{H 12}.  Allen Hatcher proves the following result:
there is an isomorphism $ H_{i}(\BDiff((D^2\times S^1)^{\natural g},
D^{2})) \cong H_{i}(Q_0BSO(3)_+) $ when $i<< g$, where $D^{2}$ is an
embedded disk in the boundary of $(D^2\times S^1)^{\natural g}$.  The
main result of this paper is a direct analogue of Hatcher's result for
higher-dimensional handlebodies $(D^{n+1}\times S^n)^{\natural g}$.
On the other hand, our technique is closely related to (and motivated
by) the work of Galatius and Randal-Williams from \cite{GRW 14}.  Let
$\theta^{n}_{2n}: BO(2n)\langle n \rangle \longrightarrow BO(2n)$
denote the $n$-connective cover.  Let $\MT\theta^{n}_{2n}$ denote the
Thom spectrum associated to the virtual vector bundle
$-(\theta^{n}_{2n})^{*}\gamma^{2n}$ over $BO(2n)\langle n \rangle$.
The main result from \cite{GRW 14} proves the isomorphism
$H_i(B\Diff((S^n\times S^n)^{\# g}, D^{2n})) \cong
H_{i}(\Omega^{\infty}_0 \MT\theta^n_{2n}),$ for $n \geq 3$, when $i<<
g$ and where $D^{2n}$ is an embedded disk in $(S^n\times S^n)^{\# g}$.
This result is proven by the techniques of cobordism categories and
parametrized surgery first introduced by Madsen and Weiss in their
proof of the Mumford Conjecture, \cite{MW 07}.
\subsection{Main Theorem}
We begin with some definitions and terminology.  Let $M$ be a smooth
manifold of dimension $m$, with non-empty boundary.  Let $\Diff(M)$
denote the group of diffeomorphisms of $M$,
topologized in the $C^{\infty}$ topology.  Fix an embedding of the
disk $i: D^{m-1} \hookrightarrow \partial M$.  Let $\Diff(M, D^{m-1})
\subset \Diff(M)$ be the subgroup consisting of all diffeomorphisms
that agree with the identity when restricted to some neighborhood of
$i(D^{m-1})$ in $M$.  Our goal is to determine the (co)homology of the
\textit{classifying space} $\BDiff(M, D^{m-1})$ for some particular
manifolds $M$.  If $X$ is a paracompact space, there is a well-known
one-to-one correspondence between the set of homotopy classes $[X,
  \BDiff(M, D^{m-1})]$ and the set of principal $\Diff(M,
D^{m-1})$-fibre bundles over $X$.  Consequently, the cohomology
$H^{*}(\BDiff(M, D^{m-1}))$ contains all characteristic classes for
principal $\Diff(M, D^{m-1})$-fibre bundles. 

Let $n, g \in \N$ and let $V^{2n+1}_{g}$ denote the $g$-fold \textit{boundary connected sum}
$(D^{n+1}\times S^{n})^{\natural g}$.  For each $g$ there is a natural
homomorphism
\begin{equation} \label{equation: stabilization map}
\xymatrix{
\Diff(V^{2n+1}_{g}, D^{2n}) \longrightarrow \Diff(V^{2n+1}_{g+1}, D^{2n})
}
\end{equation}
defined by extending a diffeomorphism identically over an attached
boundary connect-summand of $V^{2n+1}_{1} = D^{n+1}\times S^{n}$.  Our
main theorem identifies the (co)homology of the limiting space, $
\displaystyle{\colim_{g\to\infty}} \textstyle{\BDiff}(V^{2n+1}_{g},
D^{2n}).  $

To state our main theorem let us first fix some further notation.  For
any space $X$ and $k \in \N$, let $X\langle k \rangle \longrightarrow
X$ denote the \textit{$k$-connective cover}, i.e. \ the space determined uniquely up to homotopy by the condition:
$\pi_{i}(X\langle k \rangle) = 0$ for $i \leq k$ and $\pi_{i}(X\langle k \rangle) \cong \pi_{i}(X)$ for $i > k$.
Let $QX_{+}$ denote the
\textit{infinite loop-space} of the suspension spectrum associated to
$X_{+} = X\sqcup\text{pt.}$, i.e.\ $QX_{+} :=
\Omega^{\infty}\Sigma^{\infty}X_{+}$.  Finally, let $Q_{0}X_{+}
\subset QX_{+}$ denote the path component of the constant loop.  Below
we consider the space $Q_{0}BO(2n+1)\langle n \rangle_{+}$.
\begin{thmA} \label{theorem: main homology equivalence}
Let $2n+1 \geq 9$.  Then there is a map 
$$\colim_{g\to\infty}\BDiff(V^{2n+1}_{g}, D^{2n}) \; \longrightarrow
\; Q_{0}BO(2n+1)\langle n \rangle_{+}$$ 
that induces an isomorphism on
integral homology $H_{*}(\underline{\hspace{.3cm}}; \; \Z)$.
\end{thmA}
\begin{remark}
The map in Theorem A is defined for all $n \in \N$ but our proof that
it induces an isomorphism on homology requires $n \geq 4$.
However, the authors believe that our methods
could be improved to yield the analogous theorem in
the cases where $n = 3$ and $n = 2$.  
Indeed, there
is a single technical step in our proof that requires $n \geq 4$, see
Proposition \ref{proposition: inductive disjunction}. 
If this can be upgraded, then
Theorem A can be upgraded as well.
\end{remark}
\begin{remark}
  There are interesting analogs of Theorem A in the cases $n = 1$
  and $n = 0$ that are worth discussing.  
  In the case that $n = 1$, we
  have $BO(3)\langle 1 \rangle = BSO(3)$.  
  Thus, by Hatcher's result
  (see \cite{H 12}) the above theorem holds for $n=1$.  
  Theorem A is also true when $n = 0$.  
  In this case we have $BO(1)\langle 0 \rangle = BO(1)$,
    and the manifold $V^{1}_{g} = (D^{1}\times S^{0})^{\natural g}$ is
    \textcolor{black}{the disjoint union} of $(g+1)$-many copies of the one dimensional
    disk $D^{1}$.  Using the homotopy equivalence $\BDiff(D^{1})
    \simeq BO(1)$, it follows that there is a weak homotopy
    equivalence, $ \BDiff(V^{1}_{g}, D^{0}) \simeq
    E\Sigma_{g}\times_{\Sigma_{g}}(BO(1)^{\times g}). $  
    By this weak homotopy equivalence, the version
    of Theorem A
with $n = 0$ follows from the \textit{Barrat-Priddy-Quillen-Segal
  Theorem} which proves that the \textit{scanning map}, $
\displaystyle{\colim_{g\to\infty}}E\Sigma_{g}\times_{\Sigma_{g}}(BO(1)^{\times
  g}) \; \longrightarrow \; Q_{0}BO(1)_{+} $ is a homological
equivalence.  We remark that these versions of Theorem A
for $n = 0$ and $n = 1$ are proven using
different methods than the ones that we employ.
\end{remark}
In \cite{P 15} the second author of this paper proves that the map on
homology induced by (\ref{equation: stabilization map}),
$$\xymatrix{
H_{k}(\BDiff(V^{2n+1}_{g}, D^{2n}); \; \Z) \longrightarrow H_{k}(\BDiff(V^{2n+1}_{g+1}, D^{2n}); \; \Z),
}$$
is an isomorphism when $k \leq \tfrac{1}{2}(g - 4)$ and $n \geq 4$. 
By combining Theorem A
with this homological stability result, we obtain the immediate
corollary:
\begin{thmB}
Let $n \geq 4$.
There is an isomorphism 
$$\xymatrix{
H_{k}(\BDiff(V^{2n+1}_{g}, D^{2n}); \; \Z) \; \cong \; H_{k}(Q_{0}BO(2n+1)\langle n \rangle_{+}; \; \Z)
}$$
when $k \leq \tfrac{1}{2}(g - 4)$. 
\end{thmB}
For any space $X$, the rational cohomology of the infinite loop-space
$Q_{0}X_{+}$ is easy to describe: it is the free commutative algebra
associated to the rational vector space that underlies the cohomology
ring $H^{*}(X; \Q)$.  Using the well known isomorphism
$H^{*}(BO(2n+1)\langle n \rangle; \Q) \cong \Q[p_{[\frac{n+1}{4}]},
  \dots, p_{n}]$, we obtain the following corollary.
\begin{thmC}
Let $n \geq 4$.
Let $\mathcal{B} \subset H^{*}(BO(2n+1); \Q)$ denote
the set of monomials in the Pontryagin classes $p_{n}, \dots,
p_{[\frac{n+1}{4}]}$.  There is an isomorphism
$
\displaystyle
\textcolor{black}{\lim_{\leftarrow}H^{*}(\textstyle{\BDiff}(V^{2n+1}_{g}, D^{2n}); \; \Q) \; \cong \; \Q[c \; ; c \in \mathcal{B}].}
$
\end{thmC} 
\subsection{The boundary map}
We now describe a further application of Theorem A.
For $g, n \in \N$, let $W^{2n}_{g}$ denote the $g$-fold connected sum
$(S^{n}\times S^{n})^{\# g}$.  Notice that $\partial V^{2n+1}_{g} =
W^{2n}_{g}$.  There is a continuous homomorphism
\begin{equation} \label{equation: boundary map}
\xymatrix{
\Diff(V^{2n+1}_{g}, D^{2n}) \longrightarrow \Diff(W^{2n}_{g}, D^{2n}), \quad f \mapsto f|_{W^{2n}_{g}},
}
\end{equation} 
given by restricting a self-diffeomorphism of $V^{2n+1}_{g}$ to the
boundary.  For non-negative integers $k$ and $d$, let $\theta^{k}_{d}:
BO(d)\langle k \rangle \longrightarrow BO(d)$ denote the
\textit{$k$-connected cover}.  Let $\gamma^{d} \longrightarrow BO(d)$
denote the \textit{canonical $d$-dimensional vector bundle}.  Let
$\MT\theta^{k}_{d}$ denote the Thom spectrum associated to the virtual
vector bundle $-(\theta^{k}_{d})^{*}\gamma^{d}$ over $BO(d)\langle k
\rangle$ (when $k = n$ and $d = 2n$ this is the same spectrum from
\cite[Theorem 1.2]{GRW 14}).  From \cite[Proposition 3.1]{GMTW 08} it
follows that there is a fibre-sequence
\begin{equation} \label{equation: fibre sequence}
\xymatrix{
\Omega^{\infty}_{0}\MT\theta^{n}_{2n+1} \ar[r] & Q_{0}BO(2n+1)\langle n\rangle_{+} \ar[r] & \Omega^{\infty}_{0}\MT\theta^{n}_{2n}.
}
\end{equation}
There is a commutative diagram 
\begin{equation} \label{equation: commutative boundary diagram}
\xymatrix{
\displaystyle{\colim_{g\to\infty}\BDiff(V^{2n+1}_{g}, D^{2n})}  \ar[rr] \ar[d] &&   Q_{0}BO(2n+1)\langle n\rangle_{+} \ar[d] \\
\displaystyle{\colim_{g\to\infty}\BDiff(W^{2n}_{g}, D^{2n})} \ar[rr] && \Omega^{\infty}_{0}\MT\theta^{n}_{2n}.
}
\end{equation}
The map \textcolor{black}{$Q_{0}BO(2n+1)\langle n\rangle_{+}
  \longrightarrow \Omega^{\infty}\MT\theta^{n}_{2n}$} was studied in
\cite{GT 11}.  There it was proven that with rational coefficients,
the induced map on cohomology $
H^{*}(\Omega^{\infty}_{0}\MT\theta^{n}_{2n}; \Q) \longrightarrow
H^{*}(Q_{0}BO(2n+1)\langle n\rangle_{+}; \Q) $ is surjective (this
follows as a consequence of \cite[Proposition
  2.1]{GT 11}).  Combining this result from \cite{GT
  11} with the commutative diagram (\ref{equation: commutative
  boundary diagram}) we obtain the following result.
 \begin{thmD}
 Let $n \geq 4$. 
 The homomorphism 
 $$\xymatrix{ \displaystyle{\textcolor{black}{\lim_{\leftarrow}}
     \ }H^{*}(\BDiff(W^{2n}_{g}, D^{2n}); \; \Q) \; \longrightarrow \;
   \displaystyle{\textcolor{black}{\lim_{\leftarrow}}\ }H^{*}(\BDiff(V^{2n+1}_{g},
   D^{2n}); \; \Q) }$$
 induced by {\rm (\ref{equation: boundary map})} in the limit as $g \to \infty$
 is surjective.  Furthermore, the map
 $$
 \xymatrix{
H^{k}(\BDiff(W^{2n}_{g}, D^{2n}); \; \Q) \; \longrightarrow \; H^{k}(\BDiff(V^{2n+1}_{g}, D^{2n}); \; \Q)}
 $$
 is surjective when for $k \leq \frac{1}{2}(g - 4)$.
  \end{thmD}

 \subsection{Tangential structures} \label{subsection: tangential structures}
Tangential structures on manifolds will play an important role in our constructions throughout used throughout the paper. 
Our methods will actually allow us to prove a generalization of Theorem A. 
Recall that 
a \textit{tangential structure} is a map $\theta: B \longrightarrow BO(d)$. 
A $\theta$-structure on a $d$-dimensional manifold $W$ is a bundle map $TW \longrightarrow \theta^{*}\gamma^{d}$ (i.e.\ a fibrewise linear isomorphism).
More generally, a $\theta$-structure on an $m$-dimensional manifold $M$ (with $m \leq d$) is a bundle map $TM\oplus\epsilon^{d-m} \longrightarrow \theta^{*}\gamma^{d}$. 

For what follows, fix a tangential structure 
$\theta: B \longrightarrow BO(d)$. 
Fix a $\theta$-structure on the $(d-1)$-dimensional disk, $\ell_{D}: TD^{d-1}\oplus\epsilon^{1} \longrightarrow \theta^{*}\gamma^{d}$.
\begin{defn}
Let $M$ be a smooth $d$-dimensional manifold with non-empty boundary, equipped with an embedding $D^{d-1} \hookrightarrow \partial M$.
Let $\Bun(TM, \theta^{*}\gamma^{d}; \ell_{D})$ denote the space of $\theta$-structures $\ell: TM \longrightarrow \theta^{*}\gamma^{d}$ that agree with $\ell_{D}$ when restricted to the embedded disk $D^{d-1} \hookrightarrow \partial M$.
The map 
$$
\Bun(TM, \theta^{*}\gamma^{d}; \ell_{D})\times\Diff(M, D^{2n}) \; \longrightarrow \; \Bun(TM, \theta^{*}\gamma^{d}; \ell_{D}), \quad (\ell, f) \; \mapsto \; \ell\circ Df,
$$
defines a continuous group action, where in the above formula $Df$ denotes the differential of the diffeomorphism $f$.
With this group action,
the space $\BDiff_{\theta}(M, D^{2n})$ is defined to be the \textit{homotopy quotient}, $\Bun(TM, \theta^{*}\gamma^{d}; \ell_{D})//\Diff(M, D^{2n})$. 
\end{defn}
Let $\theta$ be as above and let $d = 2n+1$. 
Consider the spaces $\BDiff_{\theta}(V^{2n+1}_{g}, D^{2n})$. 
\begin{comment}
Fix once an for all a trivialization of the tangent bundle, $\varphi: TV^{2n+1}_{1} \stackrel{\cong} \longrightarrow V^{2n+1}_{1}\times\R^{2n+1}$. 
Let $\ell_{0} \in \Bun(TV^{2n+1}_{1}, \theta^{*}\gamma^{2n+1}; \ell_{D})$ be the $\theta$-structure induced by the framing $\varphi$. 
We may assume that the restriction of $\ell_{0}$ to $D^{2n} \subset \partial V^{2n+1}_{1}$ agrees with $\ell_{D}$. 
For each $g \in \Z_{\geq 0}$, this choice of $\theta$-structure $\ell_{0}$ induces a map 
$$
\Bun(TV^{2n+1}_{g}, \theta^{*}\gamma^{2n+1}; \ell_{D}) \; \longrightarrow \; \Bun(TV^{2n+1}_{g+1}, \theta^{*}\gamma^{2n+1}; \ell_{D}), \quad \ell \mapsto \ell\cup_{D^{2n}}\ell_{0}, 
$$
$\ell\cup_{D^{2n}}\ell_{0}$ is the $\theta$-structure obtained by gluing together $\ell$ and $\ell_{0}$ on the boundary connected sum $V^{2n+1}_{g}\natural V^{2n+1}_{1} = V^{2n+1}_{g+1}$, which is formed by gluing the two summand manifolds along the standard disk. 
\end{comment}
In Section \ref{subsection: proof of theorem A} we construct a direct system
\begin{equation} \label{equation: direct system theta}
\xymatrix{
\cdots \ar[r] & \BDiff_{\theta}(V^{2n+1}_{g-1}, D^{2n}) \ar[r] & \BDiff_{\theta}(V^{2n+1}_{g}, D^{2n}) \ar[r] & \BDiff_{\theta}(V^{2n+1}_{g+1}, D^{2n}) \ar[r] & \cdots
}
\end{equation}
Let $\displaystyle{\colim_{g\to\infty}}\textstyle{\BDiff_{\theta}}\BDiff(V^{2n+1}_{g}, D^{2n})$ denote its colimit. 
We have the following theorem which generalizes Theorem A.
\begin{thmA1}
Let $n \geq 4$ and let $\theta: B \longrightarrow BO(2n+1)$ be such that $B$ is $n$-connected. 
Then there is a map 
$$\displaystyle{\colim_{g\to\infty}}\textstyle{\BDiff_{\theta}}(V^{2n+1}_{g}, D^{2n}) \; \longrightarrow \; Q_{0}B_{+}$$
that induces an isomorphism on
integral homology $H_{*}(\underline{\hspace{.3cm}}; \; \Z)$.
\end{thmA1}

In \cite{P 15} it is shown that the direct system (\ref{equation: direct system theta}) satisfies homological stability. 
As a result of this homological stability theorem we obtain the following corollary. 
\begin{thmB1}
Let $n \geq 4$
and let $\theta: B \longrightarrow BO(2n+1)$ be such that $B$ is $n$-connected. 
There is an isomorphism 
$\xymatrix{
H_{k}(\BDiff(V^{2n+1}_{g}, D^{2n}); \; \Z) \; \cong \; H_{k}(Q_{0}B_{+}; \; \Z)
}$
when $k \leq \tfrac{1}{2}(g - 4)$. 
\end{thmB1}

\begin{remark}
Theorem A is obtained from Theorem A* by setting the tangential structure $\theta$ equal to the $n$-connected cover, $\theta^{n}_{2n+1}: BO(2n+1)\langle n \rangle \longrightarrow BO(2n+1)$. 
Indeed, for all $g \in Z_{\geq 0}$ the space $\Bun(TV^{2n+1}_{g}, (\theta^{n}_{2n+1})^{*}\gamma^{2n+1}; \ell_{D})$ is weakly contractible (see Proposition \ref{proposition: theta-n structures}) and so it follows that there is a weak homotopy equivalence $\BDiff_{\theta^{n}_{2n+1}}(V^{2n+1}_{g}, D^{2n}) \simeq \BDiff(V^{2n+1}_{g}, D^{2n})$.
From this weak homotopy equivalence, the homology equivalence of Theorem A* yields Theorem A. 
\end{remark}

\subsection{Cobordism categories and outline of the proof} \label{section: cobordism categories intro}
The theorems stated in the previous sections are proven using cobordism
category methods similar to those employed by Galatius and
Randal-Williams from \cite{GRW 14}, \cite{GRW 09}, and \cite{GMTW 08}.
Below we give a description of the cobordism category techniques that
we employ.  
To begin, fix a tangential structure $\theta: B \longrightarrow BO(d)$.  
 We will need to consider relative cobordisms between manifolds with boundary equipped with a $\theta$-structure. 
\begin{defn} \label{defn: relative cobordism}
Let $(M, \partial M)$ and $(N, \partial N)$ be compact manifold pairs of dimension $d-1$. 
Let $W$ be a compact, $d$-dimensional manifold 
whose boundary is equipped with a decomposition 
$$\partial W = \partial_{0}W\cup\partial_{1}W$$
that satisfies the following conditions: 
\begin{enumerate}
\item[\textcolor{black}{(i)}] $\partial_{0}W, \partial_{1}W \subset \partial W$ are compact submanifolds of dimension $(d-1)$;
\item[\textcolor{black}{(ii)}] $\partial(\partial_{0}W) = \partial_{0}W\cap \partial_{1}W = \partial(\partial_{1}W)$ (we will denote $\partial_{0,1}W := \partial_{0}W\cap \partial_{1}W)$;
\item[\textcolor{black}{(iii)}] $\partial_{0}W = M\sqcup N$ and $\partial_{0,1}W = \partial M \sqcup \partial N$.
\end{enumerate}
The pair $(W, \partial_{1}W)$ is then said to be a \textit{relative cobordism} between the manifold pairs $(M, \partial M)$ and $(N, \partial N)$.
Now let the manifolds $M$ and $N$ be equipped with $\theta$-structures $\ell_{M}$ and $\ell_{N}$ respectively. 
Suppose that $W$ is equipped with a $\theta$-structure $\ell$ that restricts to $\ell_{M}\sqcup\ell_{N}$ on $\partial_{0}W = M\sqcup N$. 
Then the pair $((W, \partial_{1}W), \ell)$ is said to be a (relative) \textit{$\theta$-cobordism} between the pairs $((M, \partial M), \ell_{M})$ and $((N, \partial N), \ell_{N})$. 
\end{defn}
We will need to consider a topological category with objects given by $\theta$-manifold pairs and morphisms given by relative $\theta$-cobordisms.
We must fix some notation.
Let $\R^{\infty}_{+}$ denote the product $[0, \infty)\times\R^{\infty-1}$ and let $\partial\R^{\infty}_{+}$ denote the boundary $\{0\}\times\R^{\infty-1}$. 
In Section \ref{section: cobordism categories} we define a topological
category $\Cob^{\partial}_{\theta}$ whose object space is defined to
be the space of all pairs $(M, \ell)$ where $M \subset
\R^{\infty}_{+}$ is a $(d-1)$-dimensional compact submanifold with
$M\cap\partial\R^{\infty}_{+} = \partial M$, and $\ell$ is a
$\theta$-structure on $M$.  
A morphism between two objects $(M, \ell_{M})$ and $(N, \ell_{N})$ is defined to be a triple $(t, (W, \ell))$ with $t \in \R_{\geq 0}$, $W \subset [0, t]\times\R^{\infty}_{+}$ is a $d$-dimensional compact submanifold (with corners) with
$$W\cap(\{0\}\times\R^{\infty}_{+}) = \{0\}\times M \quad \text{and} \quad W\cap(\{t\}\times\R^{\infty}_{+}) = \{t\}\times N,$$ 
and $\ell$ is a $\theta$-structure on $W$ that restricts to $\ell_{M}$ and $\ell_{N}$ on $M$ and $N$ respectively. 
Furthermore, the intersection $W\cap([0, t]\times\partial\R^{\infty}_{+})$ is required to be a $(d-1)$-dimensional compact manifold such that there is a decomposition,
$
\partial W = [W\cap([0, t]\times\partial\R^{\infty}_{+})]\cup(\{0\}\times M)\cup(\{t\}\times N).
$ 
By setting 
$$\partial_{0}W = (\{0\}\times M)\sqcup(\{t\}\times N) \quad \text{and} \quad \partial_{1}W = W\cap([0, t]\times\partial\R^{\infty}_{+}),$$
it follows that $((W, \partial_{1}W), \ell)$ is a relative $\theta$-cobordism between $((M, \partial M), \ell_{M})$ and $((N, \partial N), \ell_{N})$.  
In \cite{G 12}, Genauer identified the weak homotopy type of the
classifying space
$B\Cob^{\partial}_{\theta}$. 
Namely, he proves the there is a weak
homotopy equivalence
\begin{equation} \label{equation: homotopy type of cob cat}
B\Cob^{\partial}_{\theta} \simeq \Omega^{\infty-1}\Sigma^{\infty}B_{+},
\end{equation}
 where recall that $B$ is the total-space of the fibration $\theta: B
 \longrightarrow BO(d)$.  
 
We will need to filter the cobordism category
  $\Cob^{\partial}_{\theta}$ by a sequence of certain ``smaller''
  subcategories whose objects and morphisms satisfy certain
  connectivity conditions.   
\begin{defn} \label{defn: cobordism cat sequence preliminary}
Fix a $(2n-1)$-dimensional disk $D \subset \partial\R^{\infty}_{+}$.
Fix a $\theta$-structure $\ell_{D}$ on $D$.  
Let $\ell_{\R\times D}$ denote the $\theta$-structure on $\R\times D$ induced by $\ell_{D}$
in the standard way.  
We define a sequence of subcategories of
$\Cob^{\partial}_{\theta}$ as follows:
\begin{enumerate} \itemsep.2cm
\item[(a)] The topological subcategory $\Cob^{\partial, D}_{\theta}
  \subset \Cob^{\partial}_{\theta}$ has as its space of objects those
  $(M, \ell)$ such that $\partial M \subset \partial\R^{\infty}_{+}$
  contains the disk $D$ and $\ell$ agrees with $\ell_{D}$ when
  restricted to $D$.  It has as its morphisms those $(t, (W, \ell))$
  such that the face $\partial_{1}W \subset [0,
    t]\times\partial\R^{\infty}_{+}$ contains $[0, t]\times D$, and $\ell$
  agrees with $\ell_{\R\times D}$ when restricted to $[0, t]\times D$.

\item[(b)] The topological subcategory $\Cob^{\partial,
  \mb{b}}_{\theta} \subset \Cob^{\partial, D}_{\theta}$ has as its
  space of objects those $(M, \ell)$ such that $\partial M$ is
  $(n-2)$-connected.  
  For $(M, \ell), (N, \ell) \in \Ob\Cob^{\partial, \mb{b}}_{\theta}$, the morphism space
  $\Cob^{\partial, \mb{b}}_{\theta}((M, \ell), (N, \ell))$ consists
  of those $(t, (W, \ell))$ such that the pair $(\partial_{1}W,
  \partial N)$ is $(n-1)$-connected.
  
\item[(c)] The topological subcategory $\Cob^{\partial,
  \mb{c}}_{\theta} \subset \Cob^{\partial, \mb{b}}_{\theta}$
  has the same space of objects.  
  Then for objects
  $(M, \ell), (N, \ell) \in \Ob\Cob^{\partial, \mb{c}}_{\theta}$, the
  morphism space $\Cob^{\partial, \mb{c}}_{\theta}((M, \ell), (N,
  \ell))$ consists of those $(t, (W, \ell))$ such that the triad $(W;
  \partial_{1}W, N)$ is $n$-connected, i.e.\ the homomorphism induced
  by inclusion $\pi_{i}(N, \partial N) \longrightarrow \pi_{i}(W,
  \partial_{1}W)$, is an isomorphism for $i < n$ and an epimorphism for
  $i = n$.

\item[(d)] The topological subcategory $\Cob^{\partial,
  \mathcal{H}}_{\theta} \subset \Cob^{\partial, \mb{c}}_{\theta}$ is defined to be the full-subcategory on those objects $(M,
  \ell)$ such that the submanifold $M \subset \R^{\infty}_{+}$ is
  diffeomorphic to the $2n$-dimensional disk $D^{2n}$.

\item[(e)] Fix once and for all an object $(P, \ell_{P}) \in
  \Ob\Cob^{\partial, \mathcal{H}}_{\theta}$.  We define
  $\mathcal{M}_{\theta} \subset \Cob^{\partial, \mathcal{H}}_{\theta}$ to be the endomorphism monoid on the object $(P, \ell_{P})$.
\end{enumerate}
\end{defn}

Consider the sequence of subcategories
\begin{equation} \label{equation: sequence of subcategories}
\mathcal{M}_{\theta} \; \hookrightarrow \; \Cob^{\partial,
  \mathcal{H}}_{\theta} \; \hookrightarrow \; \Cob^{\partial,
  \mb{c}}_{\theta} \; \hookrightarrow \; \Cob^{\partial,
  \mb{b}}_{\theta} \; \hookrightarrow \; \Cob^{\partial, D}_{\theta} \; \hookrightarrow\; \Cob^{\partial}_{\theta}.
\end{equation}
The majority of the technical work of this paper is geared toward
proving the following theorem.
\begin{thmE} \label{theorem: monoid reduction}
Let $d = 2n+1 \geq 9$.  Suppose that the tangential structure $\theta:
B \longrightarrow BO(2n+1)$ is such that the space $B$ is at least
$(n-1)$-connected.  Then the sequence of subcategories from
{\rm (\ref{equation: sequence of subcategories})}
induces a sequence of weak homotopy equivalences
$$
B\mathcal{M}_{\theta} \; \simeq \; B\Cob^{\partial, \mathcal{H}}_{\theta} \; \simeq \; B\Cob^{\partial, \mb{c}}_{\theta} \; \simeq \; B\Cob^{\partial, \mb{b}}_{\theta}  \; \simeq \; B\Cob^{\partial, D}_{\theta} \; \simeq \; B\Cob^{\partial}_{\theta}.
$$
Thus, there is a weak homotopy equivalence $\Omega B\mathcal{M}_{\theta} \simeq \Omega^{\infty}\Sigma^{\infty}B_{+}.$
\end{thmE}

The above theorem is proven in stages using the parametrized surgery techniques developed by Madsen and Weiss in \cite{MW 07} and Galatius and Randal-Williams in \cite{GRW 14}.
The second to last stage, the weak homotopy equivalence $B\Cob^{\partial, \mathcal{H}}_{\theta} \simeq B\Cob^{\partial, \mb{c}}_{\theta}$, is by the far the most technical part of this paper. 
Proving this equivalence takes up all of Sections \ref{section: surgery on objects} and \ref{section: contractibility of the spaces of surgery data}.
Furthermore, the weak homotopy equivalence $B\Cob^{\partial, \mathcal{H}}_{\theta} \simeq B\Cob^{\partial, \mb{c}}_{\theta}$ can be viewed as a manifestation of the \textit{$\pi-\pi$-theorem} (see \cite[p. 39]{W 70}) for cobordism categories.

In the final section, we show how to deduce Theorem A from Theorem E. 
For this final step we specialize to the case where $\theta: B \longrightarrow BO(2n+1)$ is such that $B$ is $n$-connected. 
For any morphism $(t, (W, \ell)) \in
\Mor\Cob^{\partial, \mathcal{H}}_{\theta}$, Definition \ref{defn: cobordism cat sequence preliminary} implies 
that the full boundary $\partial W = \partial_{0}W\cup\partial_{1}W$
is $(n-1)$-connected and that the pair $(W, \partial W)$ is
$n$-connected.  
After smoothing corners, the manifold
$W$ is diffeomorphic to an \textit{index-$n$ handlebody} of dimension
$2n+1$ (see Definition \ref{defn: handlebody} and also \cite{W 67}).
Since the space $B$ is $n$-connected and $W$ is equipped with a $\theta$-structure, it follows that $W$ is parallelizable.
It then follows as a consequence of Wall's classification theorem (\cite{W 67}) that $W$ is diffeomorphic to $V^{2n+1}_{g}$ for some $g \in \Z_{\geq 0}$. 
By this observation together with how $\Cob^{\partial}_{\theta}$ is topologized (see Section \ref{subsection: cobordism categories}), we obtain the weak homotopy equivalence 
 \begin{equation} \label{equation: monoid with only V-g}
 \mathcal{M}_{\theta} \simeq \coprod^{\infty}_{g= 0}\textstyle{\BDiff_{\theta}}(V^{2n+1}_{g}, D^{2n}).
 \end{equation}
   Composition in this monoid can be described as being given by
boundary connected sum of handlebodies and thus it follows that the topological monoid
$\mathcal{M}_{\theta}$ is homotopy commutative.  Our next step then is
to apply the group completion theorem of McDuff and Segal.  
The weak equivalence (\ref{equation: monoid with only V-g}) lets us identify
the localized Pontryagin ring
$H_{*}(\mathcal{M}_{\theta})[\pi_{0}\mathcal{M}_{\theta}^{-1}]$ with
the integral homology of the space 
  $\Z\times\displaystyle{\colim_{g\to\infty}}\textstyle{\BDiff_{\theta}}(V^{2n+1}_{g}, D^{2n})$.  
  The group completion theorem \cite{MS 75} then yields
the isomorphism $
H_{*}(\mathcal{M}_{\theta})[\pi_{0}\mathcal{M}_{\theta}^{-1}] \; \cong
\; H_{*}(\Omega B\mathcal{M}_{\theta}). $ Combining this isomorphism
with Theorem E then implies Theorem A*.

\begin{remark}
Theorem A* requires that the tangential structure $\theta: B \longrightarrow BO(2n+1)$ be such that $B$ is $n$-connected, while Theorem E only requires $B$ to be $(n-1)$-connected. 
The place where $n$-connectivity is needed for Theorem A* is the weak homotopy equivalence of (\ref{equation: monoid with only V-g}).
If $B$ is only $(n-1)$-connected, then instead of (\ref{equation: monoid with only V-g}) we obtain the weak homotopy equivalence 
$$
\mathcal{M}_{\theta} \simeq \coprod_{V}\textstyle{\BDiff_{\theta}}(V, D^{2n}),
$$
where $V$ ranges over the diffeomorphism classes of all $(2n+1)$-dimensional, \textit{index $n$-handlebodies}, which is a strictly larger class of manifolds than $\{V^{2n+1}_{g} \; | \; g \in \Z_{\geq 0} \}$, see Proposition \ref{proposition: structure of the monoid}.
Forming the group completion of the monoid $\mathcal{M}_{\theta}$ may require the inversion of more generators than just $V^{2n+1}_{1}$, and the direct limit in the statement of Theorem A* will have to take a different form. 
For the sake of concreteness, we do not explicitly pursue these examples in detail. 
\end{remark}

\subsection{Organization}
The paper is organized as follows.  In Section \ref{section: Cobordism
  categories an spaces of manifolds} we we cover some preliminary
constructions regarding the spaces of submanifolds of Euclidean space.
Section \ref{section: cobordism categories} is devoted to constructing the cobordism categories
described above and covering some basic results regarding them.
Section \ref{section: relative surgery} is devoted to developing some basic results regarding surgery and cobordism on manifolds with boundary and corners. 
Sections \ref{section: surgery on objects}, \ref{section: contractibility of the spaces of surgery data} and the appendix are devoted to proving the weak homotopy equivalences, 
  $B\Cob^{\partial, \mathcal{H}}_{\theta} \simeq B\Cob^{\partial, \mb{c}}_{\theta} \simeq B\Cob^{\partial, \mb{b}}_{\theta} \simeq B\Cob^{\partial, D}_{\theta}.$
  The weak homotopy equivalence $B\Cob^{\partial, \mathcal{H}}_{\theta} \simeq B\Cob^{\partial, \mb{c}}_{\theta}$ is by far the most significant step and it is divided between
Sections \ref{section: surgery on objects} and \ref{section:
  contractibility of the spaces of surgery data}.  
  The proofs of the other two weak homotopy equivalences are very close to proofs from \cite{GRW 14} and so we merely provide outlines of their proofs in the appendix 
  (they are in the appendix so as not to distract from the main technical thread of the paper which is proving $B\Cob^{\partial, \mathcal{H}}_{\theta} \simeq B\Cob^{\partial, \mb{c}}_{\theta}$).
In Section \ref{section: group completion} we prove the final weak homotopy equivalence $B\mathcal{M}_{\theta} \simeq
B\Cob^{\partial, \mathcal{H}}_{\theta}$ and then we show how to
deduce Theorem A* from $B\mathcal{M}_{\theta} \simeq
\Omega^{\infty-1}\Sigma^{\infty}B_{+}$.  
\subsection{Acknowledgments}
The authors would like to thank both Soren Galatius and Oscar
Randal-Williams for very helpful conversations regarding this project
while attending PIMS Symposium on Geometry and Topology of Manifolds
during the Summer of 2015 in Vancouver, Canada. We also would like to
thank Ib Madsen for his wisdom and encouragement.  Nathan Perlmutter
was supported by an NSF Postdoctoral Fellowship, DMS-1502657.

\section{Spaces of Manifolds} \label{section: Cobordism categories an spaces of manifolds}
\subsection{Spaces of manifolds}
Fix a tangential structure $\theta: B \longrightarrow BO(d).$
We begin by recalling a basic definition from \cite{GRW 09}. 
\begin{defn} \label{defn: space of manifolds (absolute)}
For an open subset $U \subset \R^{n}$, we denote by $\bPsi_{\theta}(U)$ the space of pairs $(M, \ell)$ where $M \subset U$ is a smooth $d$-dimensional submanifold (with empty boundary) that is closed as a topological subspace, and $\ell$ is a $\theta$-structure on $M$.
More generally, if $l \leq d$ is a nonnegative integer, we define $\bPsi_{\theta_{l}}(U)$ to be the space of pairs $(M, \ell)$ where $M \subset \R^{n}$ is an $l$-dimensional submanifold without boundary and closed as a topological subspace, while $\ell$ is a $\theta$-structure on $M$. 
\end{defn}

\begin{remark}
In \cite{GRW 09} the spaces $\bPsi_{\theta_{l}}(U)$ were topologized
in such a way so that the assignment $U \mapsto \bPsi_{\theta_{l}}(U)$
defines a continuous sheaf \textcolor{black}{on $\R^{n}$}.  
We refer the
reader there for details on the construction.  We will ultimately have
to consider submanifolds of $\R^{\infty}$.  
\end{remark}

We now proceed to define spaces of manifolds with boundary analogous to the definition given above. 
Let us fix some notation. 
For
$k \in \N$, let $\R^{k}_{+}$ denote the product $[0, \infty)\times\R^{k-1}$ and let $\partial\R^{k}_{+}$ denote the boundary $\{0\}\times\R^{k-1}$.  We denote by $\R^{\infty}_{+}$ and
  $\partial\R^{\infty}_{+}$ the direct limits of those spaces taken as
  $k \to \infty$.
\begin{defn} \label{defn: space of manifolds with boundary}
For an open subset $U \subset \R^{k}\times\R^{n-k}_{+}$, we denote by $\bPsi^{\partial}_{\theta}(U)$ the space of pairs $(M, \ell)$ 
that satisfy the following conditions:
\begin{enumerate} \itemsep.2cm
\item[\textcolor{black}{(i)}] $M \subset \R^{k}\times\R^{n-k}_{+}$ is a $d$-dimensional submanifold and $\ell$ is a $\theta$-structure on $M$;
\item[\textcolor{black}{(ii)}] $M\cap(\R^{k}\times\partial\R^{n-k}_{+}) = \partial M$;
\item[(iii)] There exists $\varepsilon > 0$ such that $M\cap([0, \varepsilon)\times\R^{k}\times\partial\R^{n-k}_{+}) = [0, \varepsilon)\times\partial M$. 
Furthermore, the restriction of $\ell$ to $[0, \varepsilon)\times\partial M \subset M$ has the factorization 
$$ \xymatrix{ T([0, \varepsilon)\times\partial M) \ar[r] & T(\partial
      M)\oplus\epsilon^{1} \ar[rr]^{\ell_{\partial M}} &&
      \theta^{*}\gamma^{2n+1}, }
$$
where the first arrow is the standard bundle map covering the projection $[0, \varepsilon)\times\partial M \longrightarrow \partial M$. 
\end{enumerate}
For $l \leq d$, we then define $\bPsi^{\partial}_{\theta_{l}}(U)$ analogously to $\bPsi_{\theta_{l}}(U)$ in Definition \ref{defn: space of manifolds (absolute)} as the space of pairs $(M, \ell)$ satisfying the above three conditions but with $M$ now an $l$-dimensional manifold. 
We define $\bPsi^{\partial}_{\theta}(\R^{\infty}_{+})$ to be the direct limit of the spaces $\bPsi^{\partial}_{\theta}(\R^{n}_{+})$ taken as $n \to\infty$. 
\end{defn}

We topologize the space $\bPsi^{\partial}_{\theta}(U)$ by embedding it into 
$\bPsi_{\theta}(U_{\infty})$, where 
$$U_{\infty} \; = \; U\cup_{\partial U}(\partial U\times(-\infty, 0]) \; \subset \; \R^{k}\times\R^{n-k}$$
is the open subset obtained by attaching the cylinder 
$\partial U\times(-\infty, 0]$ to $U$ along the boundary. 
With this embedding constructed, we then topologize $\bPsi^{\partial}_{\theta}(U)$ using the subspace topology. 
Let
$(M, \ell) \in \bPsi^{\partial}_{\theta}(U).$ 
Condition (ii) of Definition \ref{defn: space of manifolds with boundary} implies that the pushout 
$$M_{\infty} := M\cup_{\partial M}(\partial M\times(-\infty, 0])$$
is a smooth submanifold of $U_{\infty} = U\cup_{\partial U}(U\times(-\infty, 0])$.
The composition
\begin{equation} \label{equation: cylinder structure}
\xymatrix{
T(\partial M\times(-\infty, 0]) \ar[rr]^{\proj^{*}_{\partial M}} && T(\partial M)\oplus\epsilon^{1}  \ar[r]^{\cong} & TM|_{\partial M} \ar[r]^{\ell} & \theta^{*}\gamma^{d+1},
}
\end{equation}
where the middle arrow is the bundle isomorphism induced by the
collar, yields a $\theta$-structure on $\partial M\times(-\infty, 0]$.
  The $\theta$-structure $\ell$ on $M$ together with the structure
  from (\ref{equation: cylinder structure}) on $\partial
  M\times(-\infty, 0]$ agree when restricted to $TM|_{\partial M}$,
    thus they glue together to yield a $\theta$-structure
    $\ell_{\infty}$ on the push-out
$$
M_{\infty} = M\cup_{\partial M}(\partial M\times(-\infty, 0]).
$$ 
We then obtain an embedding 
\begin{equation}
T_{U}: \bPsi^{\partial}_{\theta}(U) \longrightarrow \bPsi_{\theta}(U_{\infty}), \quad \quad (M, \ell) \; \mapsto \; (M_{\infty}, \;\ell_{\infty}).
\end{equation}
In this way, $\bPsi^{\partial}_{\theta}(U)$ is topologized as a
subspace of $\bPsi_{\theta}(U_{\infty})$. 

Next, we have to consider certain subspaces of
$\bPsi_{\theta}^{\partial}(\R^{k}\times\R^{n-k}_{+})$, analogous to \cite[Definition 3.5]{GRW 09}, consisting of manifolds that are ``open in $k$ directions''.
\begin{defn} \label{den: manifolds with boundary open in k directions}
We define $\bpsi^{\partial}_{\theta}(n, k) \subset
\bPsi^{\partial}_{\theta}(\R^{k}\times\R^{n-k}_{+})$ to be the
subspace consisting of those pairs $(M, \ell)$ such that
$M \; \subset \; \R^{k}\times[0, 1)\times(-1, 1)^{n-k-1}.$
For $l \leq d$, the spaces $\bpsi^{\partial}_{\theta_{l}}(n, k) \subset \bPsi^{\partial}_{\theta_{l}}(\R^{k}\times\R^{n-k}_{+})$ are defined similarly but with manifolds of dimension $l$.
For any integer $k \geq \Z_{\geq 0}$, we define $\bpsi^{\partial}_{\theta}(\infty, k)$ to be the direct limit of the spaces $\bpsi^{\partial}_{\theta}(n, k)$ taken as $n \to \infty$.
\end{defn}

The theorem below follows from \cite[Proposition 5.7]{G 12}.
\begin{theorem} \label{theorem: homotopy type of space of long manifolds}
There is a weak homotopy equivalence, 
$
\bpsi^{\partial}_{\theta}(\infty, 1) \; \simeq \; \Omega^{\infty-1}\Sigma^{\infty}B_{+}.
$
\end{theorem}

\begin{remark}
Let $\theta_{d-1}: B_{d-1} \longrightarrow BO(d-1)$ denote the tangential structure obtained by pulling back $\theta: B \longrightarrow BO(d)$ over the standard embedding $BO(d-1) \hookrightarrow BO(d)$.
In \cite{G 12} it is actually proven that there is a weak homotopy equivalence $\bpsi^{\partial}_{\theta}(\infty, 1) \simeq \Cofibre(\Sigma^{-1}\MT\theta_{d-1} \rightarrow \MT\theta)$. 
However, by \cite[Proposition 3.1]{GMTW 08} it follows that there is a weak homotopy equivalence of spectra
$$\Cofibre(\Sigma^{-1}\MT\theta_{d-1} \rightarrow \MT\theta) \simeq
\Sigma^{\infty}B_{+}.$$ Theorem \ref{theorem: homotopy type of space
  of long manifolds} is obtained by combining \cite[Proposition 5.7]{G
  12} with \cite[Proposition 3.1]{GMTW 08}.
\end{remark}
\subsection{Smooth families and submersions} \label{subsection: smooth families of manifolds}
We now recall a construction from \cite[Section 2.4]{GRW 09}.  Let $f:
X \longrightarrow \bPsi_{\theta}^{\partial}(U)$ be a continuous map,
and write $f(x)=(M_{x}, \ell_{x})$.  Define the \textit{graph} of $f$
to be
$$
\Gamma(f) = \{\; (x, y) \in X\times U \; | \; y \in M_{x} \; \}
$$
and the \textit{vertical tangent bundle} to be 
$
T^{v}\Gamma(f) \; = \; \{ \; (x, (y, v)) \in X\times TU \; | \; (y, v) \in TM_{x} \; \}.
$
The $\theta$-structures $\ell_{x}$ determine a bundle map $\ell(f): T^{v}\Gamma(f) \longrightarrow \theta^{*}\gamma^{d}$. 
\begin{defn} \label{defn: smooth maps}
If $X$ is a smooth manifold (without boundary), a continuous map $f: X \longrightarrow \bPsi_{\theta}^{\partial}(U)$ is said to be \textit{smooth} if the following conditions are met:
\begin{enumerate} \itemsep.2cm
\item[(a)]  $\Gamma(f) \subset X\times U$ is a smooth submanifold with boundary given by $\partial \Gamma(f) \; = \; \Gamma(f)\cap(X\times\partial U)$.
\item[(b)] The projections $\Gamma(f) \longrightarrow X$ and $\partial\Gamma(f) \longrightarrow X$ are both submersions.  
\end{enumerate}
\end{defn}

The following proposition sums up the key properties regarding smooth maps into the space $\bPsi_{\theta}^{\partial}(U)$.
It is proven in the same way as \cite[Lemma 2.16 and 2.17]{GRW 09}.
\begin{proposition} \label{proposition: salient features of smooth maps}
Let $X$ be a smooth manifold of dimension $m$. 
The following statements are true about maps from $X$ into $\bPsi_{\theta}^{\partial}(U)$:
\begin{enumerate} \itemsep.3cm
\item[(a)] There is a one-to-one correspondence between smooth maps $X
  \longrightarrow \bPsi_{\theta}^{\partial}(U)$ and pairs $(\Gamma,
  \ell)$ such that $\Gamma \subset X\times U$ a $(m+ d)$-dimensional
  submanifold, such that the projection $\Gamma \longrightarrow X$ is
  a submersion, and $\ell: T^{v}\Gamma \longrightarrow
  \theta^{*}\gamma^{d}$ is $\theta$-structure on the vertical tangent
  bundle, $T^{v}\Gamma = \Ker\left(T\Gamma \rightarrow TX\right)$.

\item[(b)] Given any continuous map $f: X \longrightarrow
  \bPsi_{\theta}^{\partial}(U)$, there exists a homotopy $F:
  X\times[0,1] \longrightarrow \bPsi_{\theta}^{\partial}(U)$ with
  $F|_{X\times\{0\}} = f$ such that $F|_{X\times\{1\}}$ is a smooth
  map.  If the restriction of $f$ to some open subset $A \subset X$ is
  smooth, then the homotopy $F$ map can be chosen so
  as to be constant on $A$.
\end{enumerate}
\end{proposition}

\subsection{Cobordism category of $\theta$-manifolds} \label{subsection: cobordism categories}
We now proceed to construct the cobordism category
$\Cob^{\partial}_{\theta}$ discussed in Section \ref{section: cobordism categories intro}.  
We will first need to fix some terminology and notation.  
All of our constructions will take place inside the product
$\R\times\R^{\infty}_{+}$.  We denote by
$$x_{0}:  \R\times\R^{\infty}_{+} \longrightarrow \R$$
the projection onto the first factor.
Let 
$W \subset  \R\times\R^{\infty}_{+}$ be a submanifold with
$W\cap(\R\times\partial\R^{\infty}_{+}) = \partial W.$
For any subset $K \subset \R$, we write 
$$
W|_{K} = W\cap x_{0}^{-1}(K) \quad \text{and} \quad (\partial W)|_{K} = W|_{K}\cap(\R\times\partial\R^{\infty}_{+})
$$
if they are both manifolds. 
We will also denote
$$
(W, \partial W)|_{K} = (W|_{K}, (\partial W)|_{K}).
$$ 
If $\ell$ is a $\theta$-structure on $W$, then we 
 write $\ell|_{K}$ for the restriction of $\ell$ on $W|_{K}$. 
Consider an element $(M, \ell_{M}) \in
\bpsi^{\partial}_{\theta_{d-1}}(\infty, 0)$.  
The $\theta$-structure
$\ell$ on $M$ determines a $\theta$-structure on the product $\R\times
M$, which we denote by $\ell_{\R\times M}$, given by the composition
\begin{equation} \label{equation: product with R}
\xymatrix{
T(\R\times M) \ar[r] & TM\oplus\epsilon^{1} \ar[r]^{\ell} & \theta^{*}\gamma^{d},
}
\end{equation}
where the first arrow is the standard bundle map covering the projection $\R\times M \longrightarrow M$.

 \begin{defn}
We let the non-unital topological category $\Cob^{\partial}_{\theta}$ have $\bpsi^{\partial}_{\theta_{d-1}}(0, \infty)$ as its space of objects. 
The morphism space is the following subspace of $\R\times\bpsi^{\partial}_{\theta}(\infty+1, 1)$:
a pair $(t, (W, \ell))$ is a morphism if there exists an $\varepsilon > 0$ such that
$$
W|_{(-\infty, \varepsilon)} \; = \; (-\infty, \varepsilon)\times W|_{0} \quad \text{and} \quad W|_{(t-\varepsilon, \infty)} \; = \; (t-\varepsilon, \infty)\times W|_{t}
$$
as $\theta$-manifolds, where $W|_{0}$ and $W|_{t}$ are equipped with the $\theta$-structures induced from $\ell_{0}$ and $\ell_{t}$ on $W|_{0}$ and $W|_{t}$ as in (\ref{equation: product with R}). 
The source of such a morphism is $(W|_{0}, \ell|_{0})$ and the target is $(W|_{t}, \ell|_{t})$. 
\end{defn}

\begin{remark} \label{remark: topological description of morphism spaces}
Let $(M_{0}, \ell_{0}), (M_{1}, \ell_{1}) \in \Ob\Cob^{\partial}_{\theta}$ and let $((W, \ell), t)$ be a morphism 
from $(M_{0}, \ell_{0})$ to $(M_{1}, \ell_{1})$. 
Consider the manifold $W|_{[0, t]}$. 
Let us write $W' = W|_{[0, t]}$ and $\partial_{1}W' = (\partial W)|_{[0, t]}$. 
By Definition \ref{defn: relative cobordism}, the pair $(W', \partial_{1}W')$ is a relative cobordism between $(M_{0}, \partial M_{0})$ and $(M_{1}, \partial M_{1})$ and $((W', \partial_{1}W'), \ell|_{W'})$ is a relative $\theta$-cobordism between $((M_{0}, \partial M_{0}), \ell_{0})$ and $((M_{1}, \partial M_{1}), \ell_{1})$.
In this way, our model coincides with the model for the cobordism category used in \cite{G 12}. 
As in \cite{G 12}, there is a weak homotopy equivalence
\begin{equation} \label{equation: morphism space}
\Cob^{\partial}_{\theta}((M_{0}, \ell_{0}), (M_{1}, \ell_{1})) \;
\simeq \; \coprod_{W}\textstyle{\BDiff_{\theta}}(W,
\partial_{0}W),
\end{equation}
where the disjoint union on the right is taken over all diffeomorphism classes of 
relative cobordisms $(W, \partial_{1}W)$ between the manifold pairs $(M_{0}, \partial M_{0})$ and $(M_{1}, \partial M_{1})$. 
\end{remark}

The weak homotopy type of the classifying space $B\Cob^{\partial}_{\theta}$ was determined in \cite{G 12}. 
In \cite{G 12} it is shown that there is a weak homotopy equivalence $B\Cob^{\partial}_{\theta} \simeq \bpsi^{\partial}_{\theta}(\infty, 1)$. 
Combining this weak homotopy equivalence with Theorem \ref{theorem: homotopy type of space of long manifolds} yields:
\begin{corollary} \label{corollary: homotopy type of category of mfd pairs}
There is a weak homotopy equivalence,
$B\Cob^{\partial}_{\theta} \simeq  \Omega^{\infty-1}\Sigma^{\infty}B_{+}.$
\end{corollary}

\section{Cobordism Categories of highly connected manifolds} \label{section: cobordism categories}

\subsection{A sequence of subcategories} \label{section: cobordism category of highly-connected manifolds}
We now restrict to the dimensional case pertinent to the main theorems
stated in the introduction.  Fix a tangential structure $\theta: B
\longrightarrow BO(2n+1).$
Fix once and for all a $(2n-1)$-dimensional disk 
\begin{equation}
D \; \subset \; (-\tfrac{1}{2}, 0]\times(-1, 1)^{\infty-2} 
\end{equation}
which near $\{0\}\times\partial\R^{\infty}_{+}$ agrees with $(-1, 0]\times\partial D$.
Fix a $\theta$-structure $\ell_{D}: TD\oplus\epsilon^{2} \longrightarrow \theta^{*}\gamma^{2n+1}$. 
We let $\ell_{\R\times D}$ denote the $\theta$-structure on $\R\times D$ induced by $\ell_{D}$. 
\begin{defn} \label{defn: connected sub categories}
We define a sequence of subcategories 
$$
\Cob^{\partial, \mathcal{H}}_{\theta} \subset \Cob^{\partial, \mb{c}}_{\theta} \subset \Cob^{\partial, \mb{b}}_{\theta} \subset \Cob^{\partial, D}_{\theta} \subset \Cob^{\partial}_{\theta}
$$
as follows:
\begin{enumerate} \itemsep.3cm
\item[(i)] The topological subcategory $\Cob^{\partial, D}_{\theta} \subset \Cob_{\theta}^{\partial}$ has as its objects those $(M, \ell)$ such that 
$$
\partial M\cap[(-1, 0]\times(-1, 1)^{\infty-2}] \; = \; D,
$$
and such that the restriction of $\ell$ to $D$ agrees with $\ell|_{D}$. 
It has as its morphisms those $(t, (W, \ell))$ such that 
$$
\partial W\cap[\R\times(-1, 0]\times(-1, 1)^{\infty-2}] \; = \; \R\times D,
$$
and the restriction of $\ell$ to $\R\times D$ agrees with $\ell_{\R\times D}$. 
\item[(ii)] The topological subcategory $\Cob^{\partial, \mb{b}}_{\theta} \subset \Cob^{\partial, D}_{\theta}$ has as its objects those $(M, \ell)$ such that 
the boundary $\partial M$ is $(n-2)$-connected.  
 The morphisms consist of those $(t, (W, \ell))$
 such that the pair $((\partial W)|_{[0, t]}, \; (\partial W)|_{t})$
is $(n-1)$-connected.
\item[(iii)] The topological subcategory $\Cob^{\partial,
  \mb{c}}_{\theta} \subset \Cob^{\partial, \mb{b}}_{\theta}$ has the
  same space of objects.  Its space of morphisms consists of those
  pairs $(t, (W, \ell))$ such that the triad
  $(W|_{[0, t]}; \; (\partial W)|_{[0, t]}, \; W|_{t})$ is
  $n$-connected, i.e.\ the homomorphism $ \pi_{i}(W|_{t}, (\partial
  W)|_{t}) \; \longrightarrow \pi_{i}(W|_{[0, t]}, (\partial W)|_{[0,
      t]}) $ induced by inclusion is an isomorphism for $i < n$ and an
  epimorphism for $i = n$ (see Section \ref{subsection: preliminaries
    on triads} for a primer on triads).
\item[(iv)] 
The topological subcategory 
$\Cob^{\partial, \mathcal{H}}_{\theta} \subset \Cob^{\partial, \mb{c}}_{\theta}$ 
is the full subcategory on those objects $(M, \ell)$ such that manifold $M$ is diffeomorphic to the disk $D^{2n}$. 
\end{enumerate}
\end{defn}

We now state the main theorems regarding the subcategories defined above.
The proof of the following theorem is given in the appendix.   
\begin{theorem} \label{theorem: easy weak equivalences}
Let $2n+1 \geq 5$ and let $\theta: B \longrightarrow BO(2n+1)$ be such that $B$ is $(n-2)$-connected.  
Then the inclusions $B\Cob^{\partial, \mb{c}}_{\theta} \hookrightarrow B\Cob^{\partial, \mb{b}}_{\theta} \hookrightarrow B\Cob^{\partial, D}_{\theta} \hookrightarrow B\Cob^{\partial}_{\theta}$ are weak homotopy equivalences. 
\end{theorem}

\begin{remark}
The proof the above theorem is broken up into three steps, one for
each weak homotopy equivalence asserted.  The weak homotopy
equivalence $B\Cob^{\partial, D}_{\theta} \simeq
B\Cob^{\partial}_{\theta}$ is proven in Appendix \ref{section: surgery
  on the boundary} and holds for all tangential structures $\theta$
with no conditions on the dimension.  The equivalence
$B\Cob^{\partial, \mb{b}}_{\theta} \simeq B\Cob^{\partial,
  D}_{\theta}$ is proven in Appendix \ref{section: surgery on the
  boundary}.  This weak homotopy equivalence requires the condition
that $B$ be $(n-2)$-connected, but requires no condition on the
integer $n$.  The weak homotopy equivalence $B\Cob^{\partial,
  \mb{c}}_{\theta} \simeq B\Cob^{\partial, \mb{b}}_{\theta}$ is proven
in Appendix \ref{section: surgery on morphisms} and
requires the condition that $2n+1 \geq 5$.  The proofs of these weak
homotopy equivalences are all very similar to proofs from \cite{GRW
  14} and require very little in the way of anything new, and thus we
have relegated their proofs to the appendices.
\end{remark}

The proof of the following theorem is split between Sections \ref{section: surgery on objects} and \ref{section: contractibility of the spaces of surgery data} (using techniques from Section \ref{section: relative surgery}) and constitutes the bulk of the technical work of the paper. 
\begin{theorem} \label{theorem: handlebody equivalence}
Let $d = 2n+1 \geq 9$ and let the tangential structure $\theta: B \longrightarrow BO(2n+1)$ be such that the space $B$ is $(n-1)$-connected. 
Then the inclusion
$B\Cob^{\partial, \mathcal{H}}_{\theta, D} \hookrightarrow B\Cob^{\partial, \mb{c}}_{\theta, D}$
is a weak homotopy equivalence. 
\end{theorem}

By combining the above theorems with Corollary \ref{corollary: homotopy type of category of mfd pairs}, we obtain the weak homotopy equivalence 
$$B\Cob^{\partial, \mathcal{H}}_{\theta} \; \simeq \; \Omega^{\infty-1}\Sigma^{\infty}B_{+}$$
in the case that $2n+1 \geq 9$ and $\theta: B \longrightarrow BO(2n+1)$ is such that $B$ is $(n-1)$-connected. 
In Section \ref{section: group completion} we show how the above corollary is used to derive the main theorems stated in the introduction as outlined in Section \ref{section: cobordism categories intro}.

\subsection{Flexible models for the nerve}
In order to actually prove the theorems stated above, we will need to
work with certain semi-simplicial spaces that provide more
``flexible'' models for the nerves of the cobordism categories defined
in the previous section.  Fix a tangential structure
$$\theta: B \longrightarrow BO(2n+1).$$
For the definition below we will need some notation. 
Let $\R^{k+1}_{\text{ord}} \subset \R^{k+1}$ denote the subspace consisting of those tuples $(x_{0}, \dots, x_{k})$ such that $x_{0} < x_{1} < \cdots < x_{k}$. 
Let $\R^{k+1}_{>0} \subset \R^{k+1}$ denote the subspace consisting of those tuples $(x_{0}, \dots, x_{k})$ such that $x_{i} > 0$ for all $i$.
\begin{defn} \label{defn: flexible model 1}
We define $\mb{D}^{\partial}_{\bullet}$ to be the semi-simplicial space with $p$-simplices consisting of those tuples 
$(a, \varepsilon, (W, \ell)) \; \in \; \R^{p+1}_{\text{ord}}\times\R^{p+1}_{>0}\times\bpsi^{\partial}_{\theta}(\infty, 1)$
such that:
\begin{enumerate} \itemsep.2cm
\item[(a)] $a_{i}+\varepsilon_{i} < a_{i+1} - \varepsilon_{i+1}$ for all $i = 0, \dots, p-1$;
\item[(b)] each interval $(a_{i}-\varepsilon_{i}, a_{i}+\varepsilon_{i})$ consists entirely of regular values for the height functions $W \longrightarrow \R$ and $\partial W \longrightarrow \R$.
\end{enumerate}
The face maps $d_{i}: \mb{D}^{\partial}_{p} \longrightarrow
\mb{D}^{\partial}_{p-1}$ are defined by the formula $ (a, \varepsilon,
(W, \ell)) \mapsto (a(i), \varepsilon(i), (W, \ell)) $ where $a(i) =
(a_{0}, \dots, a_{i-1}, a_{i+1}, \dots, a_{p})$ and $\varepsilon(i)$
is defined similarly.  We define a sequence of sub-semi-simplicial
spaces
$$
\mb{D}^{\partial, \mathcal{H}}_{\bullet} \subset \mb{D}^{\partial, \mb{c}}_{\bullet} \subset \mb{D}^{\partial, \mb{b}}_{\bullet} \subset \mb{D}^{\partial, D}_{\bullet} \subset \mb{D}^{\partial}_{\bullet}
$$
as follows:
\begin{enumerate} \itemsep.2cm
\item[(i)] $\mb{D}^{\partial, D}_{\bullet} \subset
  \mb{D}^{\partial}_{\bullet}$ is the sub-semi-simplicial space whose
  $p$-simplices consist of those $(a, \varepsilon, (W, \ell))$ such
  that
$$
\partial W\cap[\R\times(-1, 0]\times(-1, 1)^{\infty-2}] \; = \; \R\times D,
$$
and the restriction of $\ell$ to $\R\times D$ agrees with $\ell_{\R\times D}$. 
\item[(ii)] $\mb{D}^{\partial, \mb{c}}_{\bullet} \subset \mb{D}^{\partial, D}_{\bullet}$ is the sub-semi-simplicial space whose $p$-simplices consist of those $(a, \varepsilon, (W, \ell))$ such that:
\begin{itemize} \itemsep.2cm
\item the manifold $(\partial W)|_{c}$ is $(n-2)$-connected for all $c \in
  \bigcup_{i=0}^{p}(a_{i}-\varepsilon_{i}, a_{i}+\varepsilon_{i})$;
\item the pair $\left((\partial W)|_{[b, c]}, (\partial W)|_{c}\right)$ is $(n-1)$-connected for all $b < c \in  \bigcup_{i=0}^{p}(a_{i}-\varepsilon_{i}, a_{i}+\varepsilon_{i})$.
\end{itemize}
\item[(iii)] $\mb{D}^{\partial, \mb{c}}_{\bullet} \subset
  \mb{D}^{\partial, \mb{b}}_{\bullet}$ is the sub-semi-simplicial
  space whose $p$-simplices consist of those $(a, \varepsilon, (W,
  \ell))$ such that the triad $(W|_{[b, c]}; (\partial W)|_{[b, c]},
  W|_{c})$ is $n$-connected for all $b < c \in
  \bigcup_{i=0}^{p}(a_{i}-\varepsilon_{i}, a_{i}+\varepsilon_{i})$.
\item[(iv)] $\mb{D}^{\partial, \mathcal{H}}_{\bullet} \subset
  \mb{D}^{\partial, \mb{c}}_{\bullet}$ is the sub-semi-simplicial
  space whose $p$-simplices consist of those $(a, \varepsilon, (W,
  \ell))$ such that the manifold $W|_{b}$ is diffeomorphic to the disk
  $D^{2n}$ for all $b \in \bigcup_{i=0}^{p}(a_{i}-\varepsilon_{i},
  a_{i}+\varepsilon_{i})$.
\end{enumerate}
\end{defn}
The tangential structure $\theta$ is dropped from the notation of these semi-simplicial spaces to save space. 
In practice it will always be clear which tangential structure we are working with. 
The semi-simplicial spaces defined above are analogous to the nerves of the posets defined in \cite[Section 2.6]{GRW 14}.
We will need a second even more flexible model for the nerve of the cobordism category, analogous to \cite[Definition 2.18]{GRW 14}.

\begin{defn} \label{defn: more flexible model}
Define $\mb{X}^{\partial}_{\bullet}$ to be the semi-simplicial space with $p$-simplices consisting of tuples $(a, \varepsilon, (W, \ell))$ with that satisfy:
\begin{enumerate}
\item[(a)] $a \in \R^{p+1}_{\text{ord}}$ and $\varepsilon \in \R^{p+1}_{>0}$ are such that $a_{i}+\varepsilon_{i} < a_{i+1}-\varepsilon_{i+1}$ for all $i = 0, \dots, p-1$.
\item[(b)] $(W, \ell) \in \bPsi^{\partial}_{\theta}((a_{0}-\varepsilon_{0}, a_{p}+\varepsilon_{p})\times\R^{\infty}_{+})$ and is contained in the subspace 
$$(a_{0}-\varepsilon_{0}, a_{p}+\varepsilon_{p})\times[0, 1)\times(-1, 1)^{\infty-1}.$$
\end{enumerate}
The face maps are defined in the obvious way. 
We emphasize that unlike with $\mb{D}^{\partial}_{\bullet}$ the intervals $(a_{i}-\varepsilon_{i}, a_{i}+\varepsilon_{i})$ are not required to consist entirely of regular values for the height function $W \longrightarrow \R$. 
We define a sequence of semi-simplicial spaces 
$$
\mb{X}^{\partial, \mathcal{H}}_{\bullet} \subset \mb{X}^{\partial, \mb{c}}_{\bullet} \subset \mb{X}^{\partial, \mb{b}}_{\bullet} \subset \mb{X}^{\partial, D}_{\bullet} \subset \mb{X}^{\partial}_{\bullet}
$$
as follows:
\begin{enumerate} \itemsep.2cm
\item[(i)] $\mb{X}^{\partial, D}_{\bullet} \subset \mb{X}^{\partial}_{\bullet}$ is the sub-semi-simplicial space whose $p$-simplices consist of those $(a, \varepsilon, (W, \ell))$ such that 
$$
\partial W\cap[(a_{0}-\varepsilon_{0}, a_{p}+\varepsilon_{p})\times(-1, 0]\times(-1, 1)^{\infty-2}] \; = \; (a_{0}-\varepsilon_{0}, a_{p}+\varepsilon_{p})\times D,
  $$
  and the restriction of $\ell$ to $(a_{0}-\varepsilon_{0},
  a_{p}+\varepsilon_{p})\times D$ agrees with $\ell_{\R\times D}$.

\item[(ii)] $\mb{X}^{\partial, \mb{b}}_{\bullet} \subset
  \mb{X}^{\partial, D}_{\bullet}$ is the sub-semi-simplicial space
  whose $p$-simplices consist of those $(a, \varepsilon, (W, \ell))$
  that satisfy:
\begin{itemize}
\item for each regular value $c \in
  \bigcup_{i=0}^{p}(a_{i}-\varepsilon_{i}, a_{i}+\varepsilon_{i})$,
  the manifold $(\partial W)|_{c}$ is $(n-2)$-connected;
\item for pair of regular values $b < c \in  \bigcup_{i=0}^{p}(a_{i}-\varepsilon_{i}, a_{i}+\varepsilon_{i})$, the pair $\left((\partial W)|_{[b, c]}, (\partial W)|_{c}\right)$ is $(n-1)$-connected. 
\end{itemize}

\item[(iii)] $\mb{X}^{\partial, \mb{c}}_{\bullet} \subset
  \mb{X}^{\partial, \mb{b}}_{\bullet}$ is the sub-semi-simplicial
  space whose $p$-simplices consist of those $(a, \varepsilon, (W,
  \ell))$ such that for each pair of regular values $b < c \in
  \bigcup_{i=0}^{p}(a_{i}-\varepsilon_{i}, a_{i}+\varepsilon_{i})$,
  the triad
$$\left( W|_{[b, c]}; \; (\partial W)|_{[b, c]}, \; W|_{c}\right)$$ 
is $n$-connected. 

\item[(iv)] $\mb{X}^{\partial, \mb{c}}_{\bullet} \subset
  \mb{X}^{\partial, \mb{b}}_{\bullet}$ is the sub-semi-simplicial
  space whose $p$-simplices consist of those $(a, \varepsilon, (W,
  \ell))$ such that for every regular value $c \in
  \bigcup_{i=0}^{p}(a_{i}-\varepsilon_{i}, a_{i}+\varepsilon_{i})$, the
  manifold $W|_{c}$ is diffeomorphic to the disk $D^{2n}$.
\end{enumerate}
\end{defn}
The proposition below is proven by following the exact same arguments
used in \cite[Section 2]{GRW 14} and so we omit the proof.
\begin{proposition} \label{proposition: X-D equivalence}
There are weak homotopy equivalences:
$$\begin{aligned}
B\Cob^{\partial}_{\theta} \ &\simeq &\ \ |\mb{D}^{\partial}_{\bullet}| \ &\simeq & \ \ |\mb{X}^{\partial}_{\bullet}|, \\
B\Cob^{\partial, D}_{\theta} \ &\simeq &\ \ |\mb{D}^{\partial, D}_{\bullet}| \ &\simeq & \ \ |\mb{X}^{\partial, D}_{\bullet}|,\\
B\Cob^{\partial, \mb{b}}_{\theta} \ &\simeq & \ \ |\mb{D}^{\partial, \mb{b}}_{\bullet}| \ &\simeq& \ \ |\mb{X}^{\partial, \mb{b}}_{\bullet}|,\\ 
B\Cob^{\partial, \mb{c}}_{\theta} \ &
\simeq & \  |\mb{D}^{\partial, \mb{c}}_{\bullet}| \  &\simeq&\  \
|\mb{X}^{\partial, \mb{c}}_{\bullet}|, \\
B\Cob^{\partial, \mathcal{H}}_{\theta} \ &\simeq&  \
|\mb{D}^{\partial, \mathcal{H}}_{\bullet}| \
&\simeq & \ \ |\mb{X}^{\partial, \mathcal{H}}_{\bullet}|.
\end{aligned}$$
\end{proposition}

\section{Relative Cobordism and Surgery} \label{section: relative surgery}
\subsection{Preliminaries on triads} \label{subsection: preliminaries on triads}
We begin by introducing some basic definitions and terminology
pertaining to triads of topological spaces.  A \textit{triad} of
topological spaces is a triple $(X; A, B)$ where $X$ is a topological
space and $A, B \subset X$ are subspaces.  A map of triads $(X; A, B)
\longrightarrow (X'; A', B')$ is defined to be a map $f: X
\longrightarrow X'$ that sends $A$ to $A'$ and $B$ to $B'$.
Homotopies between maps of triads are defined in the obvious way.

We will need to work with the homotopy groups of triads.
This requires us to work with the \textit{upper half disk}.
For $n \in \N$, let $D^{n}_{+} \subset \R^{n}_{+}$ denote the unit upper half-disk, i.e. the subset given by 
\begin{equation}\label{eq:half-disk}
  D^{n}_{+}  :=
  \{ \bar{x} \in \R^{n} \;  | \; |\bar{x}| \leq 1, \; x_{1} \geq 0 \;\}.
\end{equation} 
The boundary has the decomposition $\partial D^{n}_{+} =
\partial_{0}D^{n}_{+}\cup\partial_{1}D^{n}_{+}$ where
\begin{equation}\label{eq:half-disk01}
\partial_{0}D^{n}_{+} = \{\bar{x} \in D^{n}_{+} \; | \; x_{1} = 0 \} \quad \text{and} \quad \partial_{1}D^{n}_{+} = \{\bar{x} \in D^{n}_{+} \; | \; |\bar{x}| = 1\; \}.
\end{equation} 
We denote by $\partial_{0, 1}D^{n}_{+}$ the intersection
$\partial_{0}D^{n}_{+}\cap\partial_{1}D^{n}_{+}$.  In this way,
$(D^{n}_{+}; \partial_{0}D^{n}_{+}, \partial_{1}D^{n}_{+})$ forms a
triad of topological spaces.
\begin{defn} \label{defn: homotopy groups of triads}
Let $(X; A, B)$ be a triad of topological spaces.
For $n \in \N$, the set $\pi_{n}(X; A, B)$ is defined to be the set of homotopy classes of triad maps $(D^{n}_{+}; \partial_{0}D^{n}_{+}, \partial_{1}D^{n}_{+}) \longrightarrow (X; A, B)$.
\end{defn}
In the case that $n \geq 3$, the set $\pi_{n}(X; A, B)$ forms a group and when $n \geq 4$, this group $\pi_{n}(X; A, B)$ is abelian. 
For any such triad $(X; A, B)$, there is a long exact sequence 
\begin{equation} \label{equation: exact sequence of a triad}
\xymatrix{
\cdots \ar[r] & \pi_{n}(A, A\cap B) \ar[r] & \pi_{n}(X, B) \ar[r] & \pi_{n}(X; A, B) \ar[r] & \pi_{n-1}(A, A\cap B) \ar[r] & \cdots 
}
\end{equation}
A triad $(X; A, B)$ is said to be \textit{$n$-connected} if
$\pi_{i}(X; A, B) = 0$ for all $i \leq n$.  By the above exact
sequence, it follows that a triad $(X; A, B)$ is
\textit{$n$-connected} if and only if the map $\pi_{i}(A, A\cap B)
\longrightarrow \pi_{i}(X, B)$ is an isomorphism for $i < n$ and an
epimorphism when $i = n$.  One also has homology groups of a triad.
For $n \in \Z_{\geq 0}$, the group $H_{n}(X; A, B)$ is defined simply
to be the relative homology group $H_{n}(X, A\cup B)$.  There is an
exact sequence of homology groups analogous to the exact sequence of 
(\ref{equation: exact sequence of a triad}).

We will also need to work with triads of manifolds.  A
\textit{manifold triad} is a triad $(W; P, Q)$ where $W$ is a
manifold, $P, Q \subset \partial W$ are codimension-$0$ submanifolds such
that 
$$
\partial W = P\cup Q \quad \text{and} \quad \partial P = P\cap Q = \partial Q.
$$  
If $(W, \partial_{1}W)$ is a relative cobordism between the
manifold pairs $(M, \partial M)$ and $(N, \partial N)$, then $(W;
\partial_{0}W, \partial_{1}W)$ with $\partial_{0}W = M\sqcup N$, is a
manifold triad.  \textcolor{black}{For instance,} $(D^{n}_{+},
\partial_{0}D^{n}_{+}, \partial_{1}D^{n}_{+})$ is a manifold triad as
well.

\subsection{Relative cobordism and surgery} \label{subsection: relative coboridism morse theory}
In this section we review some basic results regarding relative surgery, relative cobordism, and Morse theory on manifolds with boundary that were developed in \cite{R 12}.
Most of the results that we state are basic and we refer the reader to \cite{R 12} for the proofs and examples. 
We begin with a definition.
\begin{defn}[Handle attachment and subtraction]
  \label{defn: handle attachment and subtraction}
Let $M$ be a manifold of dimension $m$ with non-empty boundary.  Let
$p, q \in \Z_{\geq 0}$ be such that $p+q = m$.  We have two basic
operations defined below:
\begin{itemize} \itemsep.2cm
\item
Let $\varphi: S^{p-1}\times D^{q} \longrightarrow \partial M$ be an
embedding.  Let $M'$ be the manifold defined by setting $M' =
M\cup_{\varphi}(D^{p}\times D^{q})$ and then smoothing corners.  The
manifold pair $(M', \partial M')$ is said to have been obtained from
$(M, \partial M)$ by \textit{handle attachment} of index $p$.

\item
Let $\phi: (D^{p}\times D^{q}, S^{p-1}\times D^{q})
\longrightarrow (M, \partial M)$ be an embedding.  Let $M''$ be the
manifold obtained by setting $M'' =
M\setminus\phi(D^{p}\times\Int(D^{q}))$ and smoothing corners.  The
manifold pair $(M'', \partial M'')$ is said to have been obtained from
$(M, \partial M)$ by \textit{handle subtraction} of index $p$.
\end{itemize}
 \end{defn}
In the above definition, both manifold pairs $(M', \partial M')$ and
$(M'', \partial M'')$ are relative cobordant to $(M, \partial M)$.
Below we will study the relative cobordisms that realize these surgery
operations of handle attachment and subtraction.

For what follows below, let $(M, \partial M)$ and $(N, \partial N)$ be
compact manifold pairs of dimension $d-1$ and let $(W, \partial_{1}W)$
be a relative cobordism between them.  We will refer to $((W,
\partial_{1}W), M, N)$ as a \textit{relative cobordism triple}.  Let
$p, q \in \N$ be such that $p + q = d$. 
Below we use
  the half-disk $D^n_+$ and the triad $(D^{n}_{+};
  \partial_{0}D^{n}_{+}, \partial_{1}D^{n}_{+})$ as above, see
  (\ref{eq:half-disk}) and (\ref{eq:half-disk01}).
\begin{defn}[Left half-handle attachment]
Let 
$$(\varphi_{1}, \varphi_{0, 1}):  (\partial_{1}D^{p}_{+}\times D^{q},\; \partial_{0, 1}D^{p}_{+}\times D^{q}) \longrightarrow (N, \partial N)$$
be an embedding. 
Let $((W', \partial_{1}W'), M, N')$ be the relative cobordism triple defined by setting:
\begin{itemize} \itemsep.2cm
\item $W' = W\cup_{\varphi_{1}}(D^{p}_{+}\times D^{q})$,
\item $\partial_{1}W' = \textcolor{black}{\partial_{1}W\cup_{\varphi_{0, 1}}(\partial_{0}D^{p}\times D^{q})}$, 
\item $N' = [N\setminus\varphi_{1}(\partial_{1}D^{p}_{+}\times D^{q})]\bigcup_{\varphi_{1}}(D^{p}_{+}\times\partial D^{q})$, 
\end{itemize}
Constructed in this way, $((W', \partial_{1}W'), M, N')$ is said to have been obtained from $((W, \partial_{1}W), M, N)$ by \textit{left half-handle attachment} of index $p$. 
 \end{defn}
 With the above definition, it follows that 
$$
\partial N' = [\partial N\setminus\varphi_{0,1}(\partial_{0,1}D^{p}_{+}\times D^{q})]\bigcup_{\varphi_{0,1}}(\partial_{0}D^{p}_{+}\times\partial D^{q}),
$$
and so the closed manifold $\partial N'$ is obtained from $\partial N$ by standard $(p-2)$-surgery. 
The following proposition follows from \cite[Corollary 2.20]{R 12}.
\begin{proposition}
Let $((W', \partial_{1}W'), M, N')$ be obtained from $((W, \partial_{1}W), M, N)$ by left-half-handle attachment of index $p$. 
Then the manifold pair $(N', \partial N')$ is obtained from $(N, \partial N)$ by handle-subtraction of index $p-1$. 
\end{proposition}
We may consider $W$ and $\partial_{1}W$ as subspaces of $W'$ and $\partial_{1}W'$. 
Furthermore, we have a triad $(W'; \partial_{1}W', W)$ with $\partial_{1}W'\cap W = \partial_{1}W$. 
The following proposition computes some of the homotopy groups of the triad $(W'; \partial_{1}W', W)$ and of the pairs $(W', W)$ and $(\partial_{1}W', \partial_{1}W)$.
\begin{proposition} \label{proposition: left handle elementary}
Let $((W', \partial_{1}W'), M, N')$ be a relative cobordism obtained from $((W, \partial_{1}W), M, N)$ by left half-handle attachment of index $p$.
Then 
$\pi_{i}(W', W) = 0$ for all $i \in \Z$.
\end{proposition}
\begin{proof}
The space $D^{p}_{+}\times D^{q}$ deformation retracts onto
$\partial_{1}D^{p}_{+}\times D^{q}$.  Choosing such a deformation
retraction $D^{p}_{+}\times D^{q} \stackrel{\simeq} \longrightarrow
\partial_{1}D^{p}_{+}\times D^{q}$ yields a deformation retraction of
$W'$ onto $W \subset W'$.  It follows that the inclusion $W
\hookrightarrow W'$ is a homotopy equivalence and thus $\pi_{i}(W', W)
= 0$ for all $i \in \Z$.  This concludes the proof of the proposition.
\end{proof}
We now consider the operation dual to left-half-handle attachment,
that is \textit{right-half-handle} attachment.
\begin{defn}[Right half-handle attachment]
Consider the relative cobordism $((W, \partial_{1}W), M, N)$.  Let
$(\varphi, \varphi_{0}): (\partial D^{p}\times D^{q}_{+}, \partial
D^{p}\times\partial_{0}D^{q}_{+}) \longrightarrow (N, \partial N)$ be
an embedding and let $((W'', \partial_{1}W''), M, N'')$ be the
relative cobordism defined by setting:
\begin{itemize} \itemsep.2cm
\item $W'' = W\bigcup_{\varphi}(D^{p}\times D^{q}_{+})$, 
\item $\partial_{1}W'' = \partial_{1}W\bigcup_{\varphi_{0}}(D^{p}\times\partial_{0}D^{q})$, 
\item $N'' = [N\setminus\varphi(\partial D^{p}\times D^{q}_{+})]\bigcup_{\varphi}(D^{p}\times\partial_{1}D^{q}_{+})$, 
\end{itemize}
Constructed as above, the relative cobordism $((W'', \partial_{1}W''),
M, N'')$ is said to be obtained from $((W, \partial_{1}W), M, N)$ by
\textit{right-half-handle attachment} of index $p$.
\end{defn}
With the above definition, it follows that
$$\partial N'' = [\partial N\setminus\varphi_{0}(\partial D^{p}\times\partial_{0}D^{q}_{+})]\bigcup_{\varphi_{0}}(D^{p}\times\partial_{0,1}D^{q}_{+})$$
and thus the closed manifold $\partial N''$ is obtained from $\partial N$ by $(p-1)$-surgery. 
The next proposition follows from \cite[Corollary 2.20]{R 12}.
\begin{proposition}
Let $((W'', \partial_{1}W''), M, N'')$ be obtained from $((W, \partial_{1}W), M, N)$ by \textit{right half-handle attachment} of index $p$. 
Then it follows that $(N'', \partial N'')$ is obtained from $(N, \partial N)$ by handle attachment of index $p$. 
\end{proposition}

We now consider the triad $(W''; \partial W'', W)$ and the pairs $(W'', W)$ and $(\partial_{1}W'', \partial_{1}W)$.
The following proposition identifies the homotopy groups of the triad $(W''; \partial_{1}W'', W)$.
\begin{proposition} \label{proposition: right half handle elementary}
Let $((W'', \partial_{1}W''), M, N'')$ be a relative cobordism obtained from $((W, \partial_{1}W), M, N)$ 
by a right-half-handle attachment of index $p$. 
Then both $(W'', W)$ and $(\partial_{1}W'', \partial_{1}W)$ are $(p-1)$-connected and
$H_{i}(W''; \partial_{1}W'', W) = 0$ for all $i \in \Z$.
Furthermore if both $W$ and $\partial_{1}W$ are simply connected then $\pi_{i}(W''; \partial_{1}W'', W) = 0$ for $i \leq p$. 
\end{proposition}
\begin{proof}
Using a deformation retraction 
$D^{p}\times D^{q}_{+} \stackrel{\simeq} \longrightarrow (D^{p}\times\partial_{0}D^{q}_{+})\cup(\partial D^{p}\times D^{q}_{+}),$ 
we can construct a deformation retraction $W'' \stackrel{\simeq} \longrightarrow \partial_{1}W''\cup W$ that fixes the subspace $W \subset \partial_{1}W''\cup W$. 
We obtain 
a homotopy equivalence of triads $(W''; \partial_{1}W'', W) \simeq (\partial_{1}W''\cup W; \partial_{1}W'', W)$
 and isomorphisms
$$H_{*}(W''; \partial_{1}W'', W) \; \cong \; H_{*}(\partial_{1}W''\cup W; \partial_{1}W'', W) \; = \; H_{*}((\partial_{1}W''\cup W,  \partial_{1}W'' \cup W) \; = \; 0.$$
This proves the first part of the proposition. 
For the second proposition suppose that $W$ and $\partial_{1}W$ are simply connected.
The manifold $W''$ is homotopy equivalent to a a CW complex obtained from $W$ by attaching a single $p$-dimensional cell along the subspace $\partial_{1}W \subset W$. 
It follows that $(W'', W)$ and $(\partial_{1}W'', \partial_{1}W)$ are $(p-1)$-connected.
Since $W$ and $\partial_{1}W$ are both simply connected the Hurewicz theorem implies that there are isomorphisms 
$$
\pi_{p}(W'', W) \cong \Z \quad \text{and} \quad \pi_{p}(\partial_{1}W'', \partial_{1}W) \cong \Z.
$$
Furthermore the inclusion map $(\partial_{1}W'', \partial_{1}W) \hookrightarrow (W'', W)$ induces an isomorphism on $\pi_{p}$. 
The exact sequence on homotopy groups associated to the triad $(W''; \partial_{1}W'', W)$ reads as 
$$
\xymatrix{
\cdots \ar[r] & \pi_{p}(\partial_{1}W'', \partial_{1}W) \ar[r]^{\cong} &  \pi_{p}(W'', W) \ar[r]  & \pi_{p}(W''; \partial_{1}W'', W) \ar[r] & 0 \ar[r] & \cdots 
}
$$ and thus it follows that $\pi_{i}(W''; \partial_{1}W'', W) = 0$ for
all $i \leq p$.  This concludes the proof of the proposition.
\end{proof}
In the section to follow we will need to use Morse theory to analyze
the homotopical properties of relative cobordisms.  Let $((W, \partial
W_{1}), M, N)$ be a relative cobordism triple.  We will need to
consider functions $f: W \longrightarrow \R$ that contain only
singularities of Morse type.
\begin{remark}
We assume that $W$ is given a Riemannian metric which is consistent
with a decomposition of $W$ as a triad, i.e., this metric is a
product-metric near the boundary and corresponding corners.  We use
this convention throughout the paper.
\end{remark}
\begin{defn} \label{defn: Morse function}
A function $f: W \longrightarrow [0, 1]$ is said to be a
\textit{relative Morse function} with respect to
$((W, \partial_{1}W); M, N)$ if the following conditions are met:
\begin{enumerate} \itemsep.2cm
\item[(i)]  $f(M) = 0$ and $f(N) = 1$. 
\item[(ii)] $f$ has only Morse critical points, and all critical
  points occur on $W \setminus\partial_{0}W$, where recall
  $\partial_{0}W = M\sqcup N$ (we do allow for critical points to
  occur on $\partial_{1}W$ however).
\item[(iii)] the gradient $\nabla f$ is everywhere
  tangent to $\partial_{1}W$.
\end{enumerate}
\end{defn} 
Let $f: W \longrightarrow [0, 1]$ be a relative Morse function for
$((W, \partial_{1}W), M, N)$ as in the above definition.  Let $a \in
[0, 1]$ be a regular value for $f$ and denote by $W|_{a}$ the
\textit{regular level-set} $f^{-1}(a)$.  It follows from the above
definition that $W|_{a}$ is a compact manifold with boundary given
by$$W|_{a}\cap\partial_{1}W = \partial(W|_{a}).$$ We will need to
study critical points of $f$ that occur on the face $\partial_{1}W
\subset \partial W$.
\begin{defn} \label{defn: boundary critical points}
A critical point $z \in \partial_{1}W \subset W$ is said to be \textit{boundary stable} if the tangent space to the unstable manifold of $z$ lies entirely in $T_{z}(\partial_{1}W)$, otherwise the critical point $z$ is called \textit{boundary unstable}. 
The index of a critical point $z \in \partial_{1}W$ is defined as the dimension of the stable manifold of $z$.
If $z$ is boundary unstable, this index is the same as the index of $z$ regarded as a critical point of the restriction $f|_{\partial_{1}W}: \partial_{1}W \longrightarrow [0, 1]$. 
If $z$ is boundary stable, we have $\ind_{f}(z) = \ind_{f|_{\partial_{1}W}}(z) + 1$. 
\end{defn}

\begin{example} \label{example: local form morse}
Recall $\R^{n}_{+} = [0, \infty)\times\R^{n-1}$. 
Consider the functions 
$$h, f: \R^{p}_{+}\times\R^{q} \longrightarrow \R$$
given by the formulas
$$\begin{aligned}
h(x_{1}, \dots, x_{p}, y_{1}, \dots, y_{q}) &= 
-(x_{1}^{2} + \cdots + x_{p}^{2}) + y_{1}^{2} + \cdots y_{q}^{2}, \\
f(x_{1}, \dots, x_{p}, y_{1}, \dots, y_{q}) &= x_{1}^{2} + \cdots + x_{p}^{2} - (y_{1}^{2} + \cdots + y_{q}^{2}).
\end{aligned}$$
Directly from the definition it follows that 
the origin $(0, \dots, 0) \in \R^{p}_{+}\times\R^{q}$ is a boundary-stable critical point of index $p$ for the function $h$, while it is a boundary-unstable critical point of index $q$ for the function $f$. 
It is proven in \cite{R 12} that any Morse function (as defined in Definition \ref{defn: Morse function}) is locally equivalent to one of the above two examples on a neighborhood of a boundary-stable or boundary-unstable critical point. 
\end{example}

The next lemma describes the homotopical properties of a relative cobordism possessing a relative Morse function with a single critical point.
This lemma is proven in \cite{R 12}.
\begin{lemma} \label{lemma: relative homotopy group morse function}
Let $((W, \partial_{1}W), M, N)$ be a relative cobordism equipped with a relative Morse function $f: W \longrightarrow [0, 1]$ with just one critical point $z \in \Int(\partial_{1}W) \subset \partial W$. 
\begin{enumerate} \itemsep.2cm
\item[(a)] Suppose that the critical point $z \in  \Int(\partial_{1}W)$ is boundary stable of index $p$.
Then the cobordism $((W, \partial_{1}W), M, N)$ is obtained from 
$\left((M\times I, \partial M\times I), \; M\times\{0\}, M\times\{1\}\right)$
by a a left-half handle attachment of index $p$. 
\item[(b)]  Suppose that the critical point $z \in  \Int(\partial_{1}W)$ is boundary unstable of index $p$.
Then the cobordism $((W, \partial_{1}W), M, N)$ is obtained from 
$\left((M\times I, \partial M\times I), \; M\times\{0\}, M\times\{1\}\right)$
by a right-half handle attachment of index $p$. 
\end{enumerate}
\end{lemma}
Combining this lemma with the above results about left and right half-handle attachments yields the following corollary. 
\begin{corollary} \label{lemma: morse level sets}
Let $((W, \partial_{1}W), M, N)$ be a relative cobordism equipped with a relative Morse function $f: W \longrightarrow [0, 1]$ with just one critical point $z \in \Int(\partial_{1}W) \subset \partial W$. 
\begin{enumerate} \itemsep.2cm
\item[(a)] Suppose that the critical point $z \in  \Int(\partial_{1}W)$ is boundary stable of index $p$.
Then the manifold pair $(N, \partial N)$ is obtained from $(M, \partial M)$ by handle-subtraction of index $p-1$.
Furthermore,  
$\pi_{i}(W, M) = 0$ for all $i \in \Z$ while the triad $(W; \partial_{1}W, M)$ is $(p-1)$-connected and the pair $(\partial_{1}W, \partial M)$ is $(p-2)$-connected. 
\item[(b)] Suppose that the critical point $z \in  \Int(\partial_{1}W)$ is boundary unstable of index $p$.
Then the manifold pair $(N, \partial N)$ is obtained from $(M, \partial M)$ by a handle-attachment of index $p$.
Furthermore, $H_{i}(W; \partial_{1}W, M) = 0$ for all $i \in \Z$ while the pairs $(\partial_{1}W, \partial M)$ and $(W, M)$ are both $(p-1)$-connected. 
If both $M$ and $\partial M$ are simply connected then it follows further that $\pi_{i}(W; \partial_{1}W, M) = 0$ for all $i \leq p$. 
\end{enumerate}
\end{corollary}

\section{Surgery on Objects} \label{section: surgery on objects}
The task of this section is to prove Theorem \ref{theorem: handlebody equivalence} which states that the inclusion $B\Cob^{\partial, \mathcal{H}}_{\theta} \hookrightarrow B\Cob^{\partial, \mb{c}}_{\theta}$ is a weak homotopy equivalence 
whenever $d = 2n + 1 \geq 9$ and the tangential structure $\theta: B \longrightarrow BO(2n+1)$ is such that the space $B$ is $(n-1)$-connected.  
This theorem will be proven in stages.
The first stage to proving Theorem \ref{theorem: handlebody equivalence} is to do surgery on the interior of the objects of $\Cob^{\partial, \mb{c}}_{\theta}$, below the middle dimension. 
We begin by making a new definition. 
Recall that $d = 2n+1$ is the dimension of the morphisms in the cobordism category $\Cob^{\partial, \mb{c}}_{\theta}$.
\begin{defn} \label{defn: category c prime}
The topological subcategory $\Cob^{\partial, \mb{c}'}_{\theta} \subset \Cob^{\partial, \mb{c}}_{\theta}$ is the full subcategory on those objects $(M, \ell)$ such that the manifold $M$ is $(n-1)$-connected. 
\end{defn}

The theorem below is the first intermediate result on the way to proving Theorem \ref{theorem: handlebody equivalence}.
This theorem follows from \cite[Theorem 4.1]{GRW 14} and the constructions from \cite[Section 4]{GRW 14}.
\begin{theorem} \label{theorem: surgery below middle dimension}
Let the tangential structure $\theta: B \longrightarrow BO(2n+1)$ be such that the space $B$ is $(n-1)$-connected. 
Then the inclusion $B\Cob^{\partial, \mb{c}'}_{\theta} \hookrightarrow B\Cob^{\partial, \mb{c}}_{\theta}$ is a weak homotopy equivalence. 
\end{theorem}
\begin{proof}[Proof Sketch]
We first filter $\Cob^{\partial, \mb{c}}_{\theta}$ by a sequence of intermediate subcategories contained in $\Cob^{\partial, \mb{c}'}_{\theta}$.
For each $l \in \Z_{\geq -1}$, we let $\Cob^{\partial, \mb{c}, l}_{\theta} \subset \Cob^{\partial, \mb{c}}_{\theta}$ denote the full subcategory on those objects $(M, \ell)$ such that the manifold $M$ is $l$-connected. 
By \cite[Theorem 4.1]{GRW 14}, the inclusion $\Cob^{\partial, \mb{c}, l}_{\theta} \hookrightarrow \Cob^{\partial, \mb{c}, l-1}_{\theta}$ induces a weak homotopy equivalence 
$B\Cob^{\partial, \mb{c}, l}_{\theta} \simeq B\Cob^{\partial, \mb{c}, l-1}_{\theta}$
 whenever $l < n$ and the space $B$ is $l$-connected. 
By the definition we have $\Cob^{\partial, \mb{c}, n-1}_{\theta} = \Cob^{\partial, \mb{c}'}_{\theta}$ and by iterating this result, we obtain the weak homotopy equivalence $B\Cob^{\partial, \mb{c}}_{\theta} \simeq B\Cob^{\partial, \mb{c}'}_{\theta}$ whenever $B$ is $(n-1)$-connected. 
\end{proof}

\begin{remark}
The key result \cite[Theorem 4.1]{GRW 14} which was used in the above proof, is not stated in terms of the category $\Cob^{\partial}_{\theta}$. 
The result applies to a cobordism category $\Cob_{\theta}$ whose objects consist of closed $\theta$-manifolds and morphisms are given by cobordisms between them. 
It can be checked however that all constructions used in the proof of \cite[Theorem 4.1]{GRW 14} carry through to our situation; one just applies each construction to the interior of the manifolds involved (objects and morphisms) and the boundaries can be ignored entirely. 
\end{remark}

In order to prove Theorem \ref{theorem: handlebody equivalence} it will suffice to prove the following result. 
\begin{theorem} \label{theorem: disk object equivalence}
Suppose that $d = 2n + 1 \geq 9$. 
Then the inclusion $B\Cob^{\partial, \mathcal{H}}_{\theta} \hookrightarrow B\Cob^{\partial, \mb{c}'}_{\theta}$ is a weak homotopy equivalence. 
\end{theorem}

Let $(M, \ell) \in \Ob\Cob^{\partial, \mb{c}'}_{\theta}$.  
The
manifold $M$ is $(n-1)$-connected and the boundary $\partial M$ is $(n-2)$-connected.  
It follows that $H_{i}(M, \partial M) = 0$ whenever $i \neq n, 2n$.  
Our goal is to perform a
sequence of surgeries on $(M, \ell)$ so as to obtain a new object
$(M', \ell')$ with $M' \cong D^{2n}$ and thus $(M', \ell') \in
\Ob\Cob^{\partial, \mathcal{H}}_{\theta}$.  
To do this it will
suffice to kill the group $H_{n}(M, \partial M) \cong \pi_{n}(M, \partial M)$ using a sequence of \textit{handle subtraction surgeries}
in degree $n$ (this is the same surgery procedure used in the proof of
the \textit{$\pi$--$\pi$-theorem} \cite[p. 39]{W 70}).  
We will construct a new surgery move, one that implements the
operation of handle subtraction on objects in the category
$\Cob^{\partial, \mb{c}'}_{\theta}$.  
This move can be thought of as a relative version of the surgery move from \cite[Section 5]{GRW 14}.  
The construction of this surgery move (which will be carried
out in Section \ref{subsection: surgery move on objects}) will require
the use of techniques covered in the following subsection regarding
relative cobordism and surgery.

\subsection{The surgery move} \label{subsection: surgery move on objects}
We construct a relative version of the \emph{surgery
    move} from \cite[Section 5.2]{GRW 14}
that will be used to
implement the operation of handle subtraction on the objects of the
cobordism category $\Cob^{\partial, \mb{c}'}_{\theta}$. Choose a
smooth function $\sigma: \R \longrightarrow \R$ which is the identity
on $(-\infty, \tfrac{1}{2})$, has nowhere negative derivative, and
$\sigma(t) = 1$ for all $t \geq 1$.
We set
\begin{equation} \label{equation: base manifold K}
K = \left\{(x, y) \in \R^{n+1}\times\R^{n+1}_{+} \; \; \left| \; \; |y|^{2} = \sigma(|x|^{2} - 1),  \; x_{1} \in (-6, -2) \; \right.\right\},
\end{equation}
where $x_{1}$ denotes the first coordinate of $x \in \R^{n+1}$. 
We have
$$\partial K = K \cap(\R^{n+1}\times\partial\R^{n+1}_{+}).$$
Furthermore we have the containments,
$$
K \subset \R^{n+1}\times D^{n+1}_{+} \quad \text{and} \quad \partial K \subset  \R^{n+1}\times\partial_{0}D^{n+1}_{+}.
$$
Before carrying out our main construction we make several basic
observations about the pair $(K, \partial K)$ that can easily be
checked by hand.  For what follows, let
$$K_{|x|^{2} \leq 2} \subset K \quad \text{and} \quad  (\partial K)_{|x|^{2} \leq 2} \subset \partial K$$ 
be the subsets defined by 
\begin{equation} \label{equation: critical region}
\begin{aligned}
K_{|x|^{2}\leq 2} \; &= \; \ \ \textcolor{black}{K\cap\left\{(x, y) \in
  \R^{n+1}\times\R^{n+1}_{+} \ \ \left| \
  |x|^{2} \leq 2 \ \right. \right\}}, \\ (\partial
K)_{|x|^{2} \leq 2} \; &= \; \textcolor{black}{\partial K\cap\left\{(x, y) \in
\R^{n+1}\times\partial\R^{n+1}_{+} \ \left| \ |x|^{2} \leq 2\ \right. \right\}}.
\end{aligned}
\end{equation}
Similarly we define
\begin{equation} \label{equation: x bigger than 2}
K_{|x|^{2}> 2} = K\setminus K_{|x|^{2} \leq 2} \quad \text{and} \quad
(\partial K)_{|x|^{2}> 2} = \partial K\setminus(\partial K)_{|x|^{2}
  \leq 2}.
\end{equation}
\begin{proposition} \label{proposition: observation 1}
The pair $(K, \partial K)$ has the following properties:
\begin{itemize} \itemsep.2cm
\item
There are equalities
$$\begin{aligned}
\textcolor{black}{K_{|x|^2> 2}} \; &= \; (\R^{n+1}\setminus B_{\sqrt{2}}(0))\times \partial_{1}D^{n+1}_{+},\\
\textcolor{black}{(\partial K)_{|x|^2> 2}} \; &= \; (\R^{n+1}\setminus B_{\sqrt{2}}(0))\times\partial_{0, 1}D^{n+1}_{+},
\end{aligned}$$
where $B_{\sqrt{2}}(0) \subset \R^{n+1}$ is the closed ball of radius $\sqrt{2}$.
\item
There are diffeomorphisms 
$$
K|_{|x|^2 \leq 2} \; \cong \; S^{n}\times D^{n+1}_{+}, \quad
(\partial K)|_{|x|^{2} \leq 2} \; \cong S^{n}\times\partial_{0}D^{n+1}_{+}.
$$

\end{itemize}
\end{proposition}
Let $h: K \longrightarrow \R$ denote the projection of $K$ onto the
first coordinate of the ambient space $\R^{n+1}\times\R^{n+1}_{+}$.
Given regular values $a, b \in \R$ of the height function $h$ we will
need to analyze the structure of the relative cobordism $(K|_{[a, b]},
(\partial K)|_{[a, b]})$ between the pairs $(K|_{a}, (\partial
K)|_{a})$ and $(K|_{b}, (\partial K)|_{b})$.  The next proposition
enables us to do this.  Recall from Section \ref{subsection: relative coboridism morse theory} the definitions of \textit{boundary stable}
and \textit{boundary unstable critical points} of a Morse function on
a manifold with boundary.
\begin{proposition} \label{proposition: height-morse function}
The height function $h$ is a Morse function with
exactly two critical points: $(-1, 0, \dots, 0; 0)$ and $(+1, 0,
\dots, 0; 0)$, both of which lie on $\partial K$.  The critical point
$(-1, 0, \dots, 0; 0)$ is boundary stable of index $n+1$.  The
critical point $(+1, 0, \dots, 0; 0)$ is boundary unstable of index
$n$.
\end{proposition}
\begin{proof}
This follows by examination of the formula used to define the manifold $K$ and applying Example \ref{example: local form morse}. 
\end{proof}

Using the results from Section \ref{subsection: relative coboridism
  morse theory} (particularly Lemma \ref{lemma: relative homotopy group morse function}) about the homotopical properties of Morse functions on relative cobordisms, we obtain the
following corollary.
\begin{corollary} \label{corollary: relative connectivity}
Let $a, b \in \R$ be any two regular values of the height function $h: K \longrightarrow \R$. 
Then the triad $\left(K|_{[a, b]}; \; (\partial K)|_{[a, b]}, \; K|_{b}\right)$ is $n$-connected and the pair $\left((\partial K)|_{[a, b]}, \; (\partial K)|_{b}\right)$ is $(n-1)$-connected. 
\end{corollary}
\begin{proof}
We need to consider the negative of the height function, i.e.\ the function $-h$. 
For the function $-h$, the point $(-1, 0, \dots, 0; 0)$ is a boundary unstable critical point of index $n$ and $(+1, 0,
\dots, 0; 0)$ is a boundary stable critical point of index $n+1$ (compare to Proposition \ref{proposition: height-morse function}). 
The corollary then follows by applying Lemma \ref{lemma: relative homotopy group morse function} Corollary \ref{lemma: morse level sets} to the relative cobordism $\left((K|_{[a, b]}, (\partial K)|_{[a, b]}), \; K|_{a}, K|_{b}\right)$ with respect to the Morse function $-h: K|_{[a, b]} \longrightarrow [-b, -a]$.
\end{proof}

Using the manifold $K \subset \R^{n+1}\times\R^{n+1}_{+}$, we begin
our main construction.  We proceed to construct a one-parameter family
of manifolds
$$\mathcal{P}_{t} \subset \R\times(-6, -2)\times\R^{n}\times\R^{n+1}_{+}$$
with boundary given by
$$\partial\mathcal{P}_{t} \; = \; \mathcal{P}_{t}\cap
\textcolor{black}{\left( \R\times(-6,
  -2)\times\R^{n}\times\partial\R^{n+1}_{+}\right)},
$$
and with initial manifold given by
$$\mathcal{P}_{0} = \{0\}\times K \quad \text{and} \quad \partial\mathcal{P}_{0} = \{0\}\times\partial K.$$ 
The family is constructed by moving both critical points of $h$ (which is the projection onto $(-6, -2)$) down as $t \in [0, 1]$ increases, and varying the distance between the corresponding critical values. 
\begin{remark}
In this construction, the boundary family 
$\partial\mathcal{P}_{t} \; \subset \; \R\times(-6, -2)\times\R^{n}\times\partial\R^{n+1}_{+}$
will correspond exactly to the family of manifolds used in \cite[Section 5.2]{GRW 14}.
\end{remark}
Our first step along the way to constructing $\mathcal{P}_{t}$ is to
construct a one-parameter family of submanifolds
$$K_{\omega} \; \subset \; \R\times\R^{n+1}\times D^{n+1}_{+} \quad \text{with $\omega \in [0, 1]$,}$$ 
such that $K_{1} = \{0\}\times K$. 
Let 
\begin{equation}  \label{equation: real function mu}
\mu: \R \longrightarrow [0, 1]
\end{equation}
be a smooth function that satisfies the conditions:
\begin{itemize} \itemsep.2cm
\item  $\mu^{-1}(0) = [2, \infty)$,
\item $\mu^{-1}(1) = (-\infty, \sqrt{2}]$, 
\item $\mu' < 0$ on $(\sqrt{2}, 2)$.
\end{itemize}
Using $\mu$ we define a one-parameter family of embeddings 
$$
\varphi_{\omega}: \R^{n+1}\times D^{n+1}_{+} \longrightarrow \R\times\R^{n+1}\times D^{n+1}_{+}, \quad \quad \omega \in [0, 1],
$$
$$
(x, y) \mapsto \left(x_{1}(1 - \omega)\mu(|x|), \; \; x_{1}(1 - (1 - \omega)\mu(|x|)), \; \; x_{2}, \dots, x_{n+1},\; \; y\right).
$$
We define $K_{\omega}$ by setting,
\begin{equation}
K_{\omega} = \varphi_{\omega}(K).
\end{equation}
Observe that for each $\omega$ we have 
$$\partial K_{\omega} = \varphi_{\omega}(\partial K) \; \subset \; \R\times\R^{n+1}\times\partial_{0}D^{n+1}_{+}.$$

We make the following observation about the family $K_{\omega}$.  The
next proposition can be checked by hand (it is almost identical to the
observations made in \cite[p. 321]{GRW 14}).  Recall from Proposition
\ref{proposition: observation 1} the submanifold $ K_{|x|^{2}\leq 2}
\; = \; \ \ \textcolor{black}{K\cap\left\{(x, y) \in
  \R^{n+1}\times\R^{n+1}_{+} \ \ \left| \ |x|^{2} \leq 2
  \ \right. \right\}} \; \subset \; K.  $
\begin{proposition} \label{proposition: observation 2}
When $\omega > 0$, 
the critical points of the height function $h: K_{\omega} \longrightarrow \R$ are $\varphi_{\omega}(-1, 0, \dots, 0; 0)$ and $\varphi_{\omega}(+1, 0, \dots, 0; 0)$.
Both critical points are Morse and lie on the boundary; the first is boundary-stable of index $n+1$ and the second is boundary unstable of index $n$.
When $\omega = 0$, the function $h: K_{0} \longrightarrow \R$ is constantly $0$ on $\varphi_{0}(\textcolor{black}{K|_{|x|^2\leq 2}})$
but has no other critical values. 
\end{proposition}

In our next step along the way to building $\mathcal{P}_{t}$,
we construct a two-parameter family
$$\mathcal{P}_{t, \omega} \; \subset \; \R\times(-6, 2)\times\R^{n}\times D^{n+1}_{+} \quad \text{with $(t, \omega) \in [0, 1]^{2}.$}$$ 
Pick a smooth one-parameter family of embeddings 
$$\lambda_{s}: (-6, -2) \longrightarrow (-6, 2), \quad s \in [0,1],$$ 
that satisfies:
\begin{itemize} \itemsep.2cm
\item $\lambda_{0} = \Id$,
\item $\lambda_{s}|_{(-6, -5)} = \Id$ for all $s$,
\item  $\lambda_{1}(-4) = -1$,
\item $\lambda_{1}(-3) = 1$.
\end{itemize}
We then get a family of embeddings
$$
\zeta_{t}: \R\times(-6, -2)\times\R^{n}\times\R^{n+1}_{+} \longrightarrow \R\times(-6, 2)\times\R^{n}\times\R^{n+1}_{+}, \quad t \in [0, 1],
$$
by setting 
$$
\zeta_{t} = {\textstyle \Id_{\R}\times\lambda_{t}\times\Id_{\R^{n}}\times\Id_{\R^{n+1}_{+}}}.
$$
We define \textcolor{black}{a two parameter family of manifolds}
\begin{equation} \label{equation: definition of 2nd parameter}
\mathcal{P}_{t, \omega} \; := \; \zeta_{t}^{-1}(K_{\omega}), \quad \textcolor{black}{(t, \omega) \in [0, 1]^{2}.}
\end{equation}

We obtain our final one-parameter family $\mathcal{P}_{t}$ as follows. 
Let $P: [0, 1] \longrightarrow [0, 1]^{2}$ be a piecewise linear path with 
$P(0) = (0, 0),$ $P(\tfrac{1}{2}) = (1, 0),$ and $P(1) = (1, 1).$
We then define:
\begin{equation} \label{equation: piecewise linear path}
\mathcal{P}_{t} \; := \; \mathcal{P}_{P(t)}.
\end{equation}
Now suppose that the initial manifold pair $\mathcal{P}_{0}$ is endowed with a $\theta$-structure $\ell$. 
We will need to extend this structure $\ell$ over the whole family so as to obtain a family of $\theta$-manifolds $(\mathcal{P}_{t}, \ell_{t})$ with $\ell_{0} = \ell$. 
\begin{proposition}
Suppose that the initial manifold $\mathcal{P}_{0}$ is equipped with $\theta$-structure $\ell$. 
Then we can upgrade the construction, continuously in $\ell$, to obtain a one-parameter family 
$$
(\mathcal{P}_{t}, \ell_{t}) \; \in \; \bPsi^{\partial}_{\theta}(\R\times(-6, -2)\times\R^{n}\times\R^{n+1}_{+})
$$
with $\ell_{0} = \ell$.
\end{proposition}
\begin{proof}
First observe that the initial manifold
$\mathcal{P}_{0}$ is equal to 
$$K_{1}|_{(-6, -2)} = \{0\}\times K|_{(-6, -2)} \; \; \subset \; \; \R\times(-6, -2)\times\R^{n}\times\R^{n+1}_{+}.$$
Since 
$
K_{\omega} = \varphi_{\omega}(K),
$
by composing $\ell$ with the differential of the embedding $\varphi_{\omega}$, 
the structure $\ell$ determines a 
structure $\ell_{\omega}$ on $K_{\omega}$ for each $\omega \in [0, 1]$. 
These structures vary continuously in $\omega$ making $(K_{\omega}, \ell_{\omega})$ into a family of $\theta$-manifolds. 
We then define a family of $\theta$-structures on the two parameter family 
$$\mathcal{P}_{t, \omega} \; = \; \zeta_{t}^{-1}(K_{\omega}),$$
(see (\ref{equation: definition of 2nd parameter}))
 by pulling-back $\ell_{\omega}$ along $\zeta_{t}$.  
This gives a continuous, two parameter family of $\theta$-manifolds denoted by 
$(\mathcal{P}_{t, \omega}, \; \ell_{t, \omega}).$
By setting $(t, \omega)$ equal to $P_{s}$ with $s \in [0, 1]$, where $P: [0, 1] \longrightarrow [0, 1]^{2}$ is the path from (\ref{equation: piecewise linear path}), 
we obtain a the desired one-parameter family of elements 
$$
(\mathcal{P}_{t}, \ell_{t}) \; \in \; \bPsi^{\partial}_{\theta}(\R\times(-6, -2)\times\R^{n}\times\R^{n+1}_{+}).
$$
This concludes the proof of the proposition. 
\end{proof}

The following proposition sums up the basic properties of this one-parameter family. 
\begin{proposition} \label{proposition: basic properties of object surgery move}
There exists a one-parameter family 
$$
\mathcal{P}_{t} \; \in \; \bPsi^{\partial}_{d}(\R\times(-6, -2)\times\R^{n}\times\R^{n+1}_{+})
$$
with the following properties:
\begin{enumerate} \itemsep.3cm
\item[(i)]
The initial manifold $\mathcal{P}_{0}$ is equal to the submanifold
$$K_{1}|_{(-6, -2)} \; = \; \{0\}\times K|_{(-6, -2)} \; \; \subset \; \; \R\times(-6, -2)\times\R^{n}\times\R^{n+1}_{+}.$$ 

\item[(ii)] For all $t \in [0, 1]$, $\mathcal{P}_{t}$ agrees with $\mathcal{P}_{0}$ outside of the set
$\textcolor{black}{(-2, 2)\times(-5, -2)\times B_{\sqrt{2}}(0)\times D^{n+1}_{+}.}$

\item[(iii)] For all $t \in [0, 1]$ and each pair of regular values $-6 < a < b < -2$ of the height function 
$h: \mathcal{P}_{t} \longrightarrow (-6, -2),$
the triad
$
(\mathcal{P}_{t}|_{[a, b]}; \; (\partial\mathcal{P}_{t})|_{[a, b]}, \; \mathcal{P}_{t}|_{b})
$
is $n$-connected and the pair $((\partial\mathcal{P}_{t})|_{[a, b]}, \; (\partial\mathcal{P}_{t})|_{b})$ is $(n-1)$-connected. 

\item[(iv)] If $a$ is outside of $(-4, -3)$ and is a regular value of $h: \mathcal{P}_{t} \longrightarrow (-6, -2)$, then the manifold pair  $(\mathcal{P}_{t}|_{a}, \; (\partial\mathcal{P}_{t})|_{a})$ is diffeomorphic to 
$
(\mathcal{P}_{0}|_{a}, \; (\partial\mathcal{P}_{0})|_{a}).
$
If $a$ is inside $(-4, -3)$ and is a regular value of $h$ then 
$(\mathcal{P}_{t}|_{a}, \; (\partial\mathcal{P}_{t})|_{a})$ is either diffeomorphic to $(\mathcal{P}_{0}|_{a}, \; (\partial\mathcal{P}_{0})|_{a})$, or is obtained from it by handle-subtraction of index $n$. 

\item[(v)] The critical values of $h: \mathcal{P}_{1} \longrightarrow (-6, -2)$ are $-4$ and $-3$. 
For $a \in (-4, -3)$, the manifold pair $(\mathcal{P}_{1}|_{a}, \; \partial\mathcal{P}_{1}|_{a})$ is is diffeomorphic to the manifold pair obtained from 
$$(\mathcal{P}_{0}|_{a},  \; \partial\mathcal{P}_{0}|_{a}) \; = \; (\{0\}\times\R^{n}\times\partial_{1}D^{n+1}_{+}, \; \{0\}\times\R^{n}\times\partial_{0, 1}D^{n+1}_{+})$$
by handle-subtraction of index $n$
along the standard embedding. 
\end{enumerate}
Furthermore, suppose that the initial manifold $\mathcal{P}_{0}$ is equipped with a $\theta$-structure $\ell$. 
Then we can upgrade the construction, continuously in $\ell$, to obtain a one-parameter family 
$$
\mathcal{P}_{t}(\ell) \; \in \; \bPsi^{\partial}_{\theta}(\R\times(-6, -2)\times\R^{n}\times\R^{n+1}_{+})
$$
with $\mathcal{P}_{0}(\ell) = (\mathcal{P}_{0}, \ell)$. 
\end{proposition}
\begin{proof}
Properties (i) and (ii) follow immediately from the construction. 
Property (iii) follows from Corollary \ref{corollary: relative connectivity}. 
Using Proposition \ref{proposition: height-morse function}, properties (iv) and (v) follow from Lemma \ref{lemma: morse level sets} which describes how the level sets of a relative Morse function change when passing through a critical point on the boundary. 
\end{proof}

There is one result regarding $\theta$-structures on the manifold $K$ that will prove useful latter on in Section \ref{section: contractibility of the spaces of surgery data} (see the proof of Proposition \ref{proposition: nonempty surgery data}). 
Recall from (\ref{equation: critical region}) and (\ref{equation: x bigger than 2}) the submanifolds $K_{|x|^{2}\leq 2}, K_{|x|^{2}>2} \; \subset \; K$ where $ K_{|x|^{2}>2} = K\setminus K_{|x|^{2}\leq 2}$.
\begin{proposition} \label{proposition: extension of theta-structure}
Let $\ell$ be any $\theta$-structure on $K_{|x|^{2}>2}$. 
Then there exists a $\theta$-structure $\ell'$ on $K$ such that $\ell'$ agrees with $\ell$ when restricted to $K_{|x|^{2}>2} \subset K$.
\end{proposition}
\begin{proof}
Let $K_{|x|^{2} = 2}$ denote the intersection $\Cl(K_{|x|^{2}> 2})\cap K_{|x|^{2}\leq 2}$, where $\Cl(K_{|x|^{2}> 2})$ is the closure of $K_{|x|^{2}\leq 2}$ in $K$.
Let $\bar{\ell}$ denote the restriction of the structure $\ell$ to $K_{|x|^{2} = 2}$. 
To prove the proposition we will need to show that $\bar{\ell}$ can be extended over $K_{|x|^{2}\leq 2}$.
Recall from Proposition \ref{proposition: observation 1} that 
$$K_{|x|^{2}> 2} \; = \; (\R^{n+1}\setminus B_{\sqrt{2}}(0))\times \partial_{1}D^{n+1}_{+}$$ 
and that 
$$K_{|x|^{2}\leq 2} \cong S^{n}\times D^{n+1}_{+}.$$ 
It follows that
$K_{|x|^{2} = 2}$ is diffeomorphic to
$S^{n}\times\partial_{1}D^{n+1}_{+}$ and that $K_{|x|^{2} \leq 2}$
deformation retracts onto $K_{|x|^{2} = 2}$.
By this homotopy
equivalence $K_{|x|^{2} = 2} \simeq K_{|x|^{2}\leq 2}$, we see that
there is no obstruction to the extension of $\bar{\ell}$ over
$K_{|x|^{2}\leq 2}$.  
This completes the proof of Proposition \ref{proposition: extension of theta-structure}.
\end{proof}
\subsection{Surgery data}
We construct a semi-simplicial space of surgery data similar the one
constructed in \cite[Section 5.3]{GRW 14}.  First we must choose once
and for all, smoothly in the data $(a_{i}, \varepsilon_{i}, a_{p},
\varepsilon_{p})$, increasing diffeomorphisms
\begin{equation} \label{equation: reparametrisation}
\psi = \psi(a_{i}, \varepsilon_{i}, a_{p}, \varepsilon_{p}): (-6, -2) \stackrel{\cong} \longrightarrow (a_{i} - \varepsilon_{i}, a_{p} + \varepsilon_{p}),
\end{equation}
sending $[-4, -3]$ linearly onto $[a_{i}-\tfrac{1}{2}\varepsilon_{i},
  a_{i} + \tfrac{1}{2}\varepsilon_{i}]$.  

\begin{defn} \label{defn: objects surgery data}
Fix once and for all an infinite set $\Omega$.
Let $x = (a, \varepsilon, (W, \ell_{W})) \in \mb{D}^{\partial,\mb{c}'}_{p}$. 
With $a = (a_{0}, \dots, a_{p})$, for each $i$ we write $M_{i} = W|_{a_{i}}$ and $\partial M_{i} = (\partial W)|_{a_{i}}$.
Define a set $\mb{Y}_{0}(x)$ to consist of tuples $(\Lambda, \delta, e, \ell)$, where:
\begin{itemize} \itemsep.3cm
\item $\Lambda \subset \Omega$ is a finite subset;
\item $\delta: \Lambda \longrightarrow [p]$ is a function;
\item  
$
e: \Lambda\times\R\times(-6, -2)\times\R^{n}\times D^{n+1}_{+} \longrightarrow \R\times [0, 1)\times(-1, 1)^{\infty-1}
$
is an embedding that maps $\Lambda\times\R\times(-6, -2)\times\R^{n}\times\partial_{0}D^{n+1}_{+}$ into the subspace  $\R\times\{0\}\times(-1, 1)^{\infty-1}$;
\item $\ell$ is a $\theta$-structure on $\Lambda\times K$ (where $K$ is the manifold from (\ref{equation: base manifold K})).
\end{itemize}
For $i \in [p]$, we write $\Lambda_{i} = \delta^{-1}(i) \subset \Lambda$ and 
$$
e_{i}: \Lambda_{i}\times(a_{i} - \varepsilon_{i}, a_{p}+\varepsilon_{p})\times\R^{n}\times D^{n+1}_{+} \longrightarrow \R\times[0, 1)\times(-1, 1)^{\infty-1}
$$
for the embedding obtained by restricting $e$ and reparametrising using (\ref{equation: reparametrisation}). 
We let $\ell_{i}$ denote the restriction of $\ell$ to $\Lambda_{i}\times K|_{(-6, -2)}$.  
This data is required to satisfy the following conditions:
\begin{enumerate} \itemsep.2cm

\item[(i)] 
The embedding $e$ satisfies:
$$\begin{aligned}
e^{-1}(W) &= \Lambda\times\{0\}\times(-6, -2)\times\R^{n}\times \partial_{1}D^{n+1}_{+}, \\
e^{-1}(\partial W) &= \Lambda\times\{0\}\times(-6, -2)\times\R^{n}\times \partial_{0, 1}D^{n+1}_{+}.
\end{aligned}$$
We denote by $\partial_{1}e$ and $\partial_{0,1}e$ the embeddings obtained by restricting $e$ to the subspaces $\Lambda\times\{0\}\times(-6, -2)\times\R^{n}\times \partial_{1}D^{n+1}_{+}$ and $\Lambda\times\{0\}\times(-6, -2)\times\R^{n}\times \partial_{0, 1}D^{n+1}_{+}$ respectively. 
\item[(ii)]
Let $i \in [p]$. 
For 
$t \in \bigcup_{k = i}^{p}(a_{k}-\varepsilon_{k}, a_{k}+\varepsilon_{k})$, we have
$$(x_{0}\circ e_{i})^{-1}(t) = \Lambda_{i}\times\R\times\{t\}\times\R^{n}\times D^{n+1}_{+}.$$
In other words, the embedding $e_{i}$ is height-preserving over the intervals $(a_{k}-\varepsilon_{k}, a_{k}+\varepsilon_{k})$ for $k = i, \dots, p$. 

\item[(iii)] 
Recall from (\ref{equation: x bigger than 2}) the submanifold $(K_{|x|^{2} >2})|_{(-6, -2)} \subset K|_{(-6, -2)}$. 
Condition (i) together with Proposition \ref{proposition: observation 1} implies that 
$$e^{-1}(W)\cap(\Lambda\times K|_{(-6, -2)}) = \Lambda\times(K_{|x|^{2} >2})|_{(-6, -2)}.$$ 
We require that the restriction of $\ell_{W}$ to the image 
$$e(\Lambda\times(K_{|x|^{2} >2})|_{(-6, -2)}) \subset W$$ 
agrees with the $\theta$-structure $\ell$.
\end{enumerate}
For each $i \in [p]$, the restriction of $e$ to $\Lambda_{i}\times\{0\}\times\{a_{i}\}\times \R^{n}\times \partial_{1}D^{n+1}_{+}$
yields the embedding of manifold pairs
\begin{equation} \label{equation: induced slice embedding}
\left(\Lambda_{i}\times\{0\}\times\{a_{i}\}\times \R^{n}\times \partial_{1}D^{n+1}_{+}, \; \;\Lambda_{i}\times\{0\}\times\{a_{i}\}\times \R^{n}\times \partial_{0, 1}D^{n+1}_{+}\right) \; \longrightarrow \; (M_{i}, \partial M_{i}).
\end{equation}
Denote by $M'_{i}$ the manifold obtained by setting 
$$
M'_{i} \; = \; M_{i}\setminus e_{i}(\Lambda_{i}\times\{0\}\times\{a_{i}\}\times \R^{n}\times \partial_{1}D^{n+1}_{+}).$$
Hence, $M'_{i}$ is obtained from $M_{i}$ by a collection of disjoint handle subtractions of index-$n$ via (\ref{equation: induced slice embedding}). 
For our final condition, we insist that: 
\begin{enumerate}
\item[(iv)] for each $i \in [p]$ the resulting manifold $M'_{i}$ is diffeomorphic to the disk $D^{2n}$. 
\end{enumerate}
\vspace{.3cm}
\end{defn}

We now proceed to use the construction from the above definition to define an augmented bi-semi-simplicial space. 
\begin{defn} \label{defn: second simplicial coordinate n-1}
Let $x = (a, \varepsilon, (W, \ell_{W})) \in \mb{D}^{\partial,\mb{c}'}_{p}$ and let $q \in \Z_{\geq 0}$. 
We define $\mb{Y}_{q}(x) \subset (\mb{Y}_{0}(x))^{\times(q+1)}$ to be the subset consisting of those $(q+1)$-tuples 
$$
\left((\Lambda^{0}, \delta^{0}, e^{0}, \ell^{0}), \dots, (\Lambda^{q}, \delta^{q}, e^{q}, \ell^{q})\right) \; \in \; (\mb{Y}_{0}(x))^{\times(q+1)}
$$
such that 
$\Image(e^{i})\cap\Image(e^{j}) = \emptyset$ whenever $i \neq j$,
where $\Image(e^{i})$ denotes the image of $e^{i}$.
For each $x$, the assignment $q \mapsto \mb{Y}_{q}(x)$ defines a semi-simplicial set $\mb{Y}_{\bullet}(x)$.
We obtain a bi-semi-simplicial set $\mb{D}^{\partial,\mathcal{H}}_{\bullet, \bullet}$ by setting,
$
\mb{D}^{\partial,\mathcal{H}}_{p, q} = \{(x, y) \; | \; x \in \mb{D}^{\mathcal{L}, l-1}_{p}, \; y \in \mb{Y}^{l}_{q}(x) \; \}.
$
We topologize $\mb{D}^{\partial,\mathcal{H}}_{p, q}$ as a subspace of 
$$
\mb{D}^{\partial,\mb{c}'}_{p}\times\left(\prod_{\Lambda \subset \Omega}C^{\infty}\left(\Lambda\times\R\times(-6, -2)\times\R^{n}\times D^{n+1}_{+}, \; \; \R\times [0, 1)\times(-1, 1)^{\infty-1}\right)\right)^{\times(p+1)(q+1)}
$$
and thus the assignment $[p]\times[q] \mapsto \mb{D}^{\partial,\mathcal{H}}_{p, q}$ defines a bi-semi-simplicial space.
\end{defn}
For each $q$ we have a projection
$\mb{D}^{\partial, \mathcal{H}}_{p, q} \longrightarrow \mb{D}^{\partial, \mb{c}'}_{p},$  $(x, y) \mapsto x.$
Setting $\mb{D}^{\partial, \mathcal{H}}_{\bullet, -1} = \mb{D}^{\partial, \mb{c}'}_{\bullet}$, these projections yield an augmented bi-semi-simplicial space 
$
\mb{D}^{\partial, \mathcal{H}}_{\bullet, \bullet} \; \longrightarrow \;  \mb{D}^{\partial, \mathcal{H}}_{\bullet, -1} = \mb{D}^{\partial, \mb{c}'}_{\bullet}.
$
The proof of the following theorem is given in Section \ref{section: contractibility of the spaces of surgery data} and is the most technical part of this paper. 
\begin{theorem} \label{theorem: contractibility of space of surgery data}
Suppose that $d = 2n+1 \geq 9$.  
Then the augmentation $\mb{D}^{\partial, \mathcal{H}}_{\bullet, \bullet} \; \longrightarrow \;   \mb{D}^{\partial, \mb{c}'}_{\bullet}$ induces the weak homotopy equivalence 
$
|\mb{D}^{\partial, \mathcal{H}}_{\bullet, \bullet}| \; \simeq \; |\mb{D}^{\partial, \mb{c}'}_{\bullet}|.
$
\end{theorem}

\begin{remark}
The above theorem, which is proven in Section \ref{section: contractibility of the spaces of surgery data}, is the only place that uses the condition $d = 2n+1 \geq 9$.
\end{remark}

\subsection{Implementation of the surgery move}
In this section we show how to implement the surgery move that was constructed in Section \ref{subsection: surgery move on objects}.
For $p \in \Z_{\geq 0}$, we will need to work with the space 
$\mb{D}^{\partial, \mathcal{H}}_{p, 0}$.
\begin{Construction} \label{constriction: implementation of surgery move}
For what follows let 
$$x = (a, \varepsilon, (W, \ell_{W})) \in \mb{D}^{\partial, \mathcal{H}}_{p} \quad \text{and}
\quad y = (\Lambda, \delta, e, \ell) \in \mb{Y}_{0}(x).$$
In this way $(x, y) \in \mb{D}^{\partial, \mathcal{H}}_{p, 0}$.
For each $i = 0, \dots, p$, we have an embedding 
$e_{i}$ and a $\theta$-structure $\ell_{i}$ on $\Lambda_{i}\times K|_{(-6, -2)}$, where $\Lambda_{i} = \delta^{-1}(i)$.  
From this data we may construct a one-parameter family of $\theta$-manifolds 
\begin{equation} \label{defn: glued in surgery move}
\mathcal{K}^{t}_{e_{i}, \ell_{i}}(W, \ell_{W}) \in \bPsi^{\partial}_{\theta}((a_{0}-\varepsilon_{0}, a_{p}+\varepsilon_{p})\times\R_{+}^{\infty})
\end{equation}
as follows.
First, using the reparametrisation (\ref{equation: reparametrisation}) we may reparametrise the first coordinate of the family $\Lambda_{i}\times\mathcal{P}_{t}(\ell_{i})$ to obtain a new family 
$$\Lambda_{i}\times\bar{\mathcal{P}_{t}}(\ell_{i}) \in \bPsi_{d}^{\partial}(\Lambda_{i}\times (a_{i}-\varepsilon_{i}, a_{p}+\varepsilon_{p})\times\R^{n}\times\R^{n+1}_{+}).$$
This one-parameter family has all of the same properties as in Proposition \ref{proposition: basic properties of object surgery move}, except now properties (iv) and (v) apply to the interval $(a_{i}-\tfrac{1}{2}\varepsilon_{i}, a_{i}+ \tfrac{1}{2}\varepsilon_{i})$ (recall that the reparametrisation (\ref{equation: reparametrisation}) sends $(-4, -3)$ to $(a_{i}-\tfrac{1}{2}\varepsilon_{i}, a_{i}+ \tfrac{1}{2}\varepsilon_{i})$).
We set $\mathcal{K}^{t}_{e_{i}, \ell_{i}}(W, \ell_{W})$ equal to $W|_{(a_{0}-\varepsilon_{0}, a_{p}+\varepsilon_{p})}$ outside of the image of $e_{i}$, and we set it equal to equal to $e_{i}(\Lambda_{i}\times\bar{\mathcal{P}_{t}}(\ell_{i}))$ on the set
$$
\textcolor{black}{e_{i}(\Lambda_{i}\times\{0\}\times(a_{i}-\varepsilon_{i}, a_{p}+\varepsilon_{p})\times\R^{n}\times D^{n+1}_{+}).}
$$
\end{Construction}
We make the following observation about the family $\mathcal{K}^{t}_{e_{i}, \ell_{i}}(W, \ell_{W})$, which follows directly from Proposition \ref{proposition: basic properties of object surgery move} (v) and (vi). 
\begin{proposition} \label{proposition: observation 1 about glued in move}
Let $(a, \varepsilon, (W, \ell_{W}), e, \ell) \in \mb{D}^{\partial, \mathcal{H}}_{p, 0}$. 
For $i = 0, \dots, p$, let $W_{t}$ denote the family $\mathcal{K}^{t}_{e_{i}, \ell_{i}}(W, \ell_{W})$.
The pair $(W_{t}, \partial W_{t})$ has the following properties:
\begin{enumerate} \itemsep.2cm
\item[(i)] 
Let $j \neq i$, and let $b \in (a_{j}-\tfrac{1}{2}\varepsilon_{j}, a_{j}+\tfrac{1}{2}\varepsilon_{j})$ be a regular value for $x_{0}: W_{t} \longrightarrow \R$.
Then $b$ is a regular value for $x_{0}: W \longrightarrow \R$ as well and 
$(W_{t}|_{b}, (\partial W_{t})|_{b})$ is diffeomorphic to $(W|_{b}, (\partial W)|_{b})$. 
\item[(ii)] 
If $c \in (a_{i}-\tfrac{1}{2}\varepsilon_{i}, a_{i}+\tfrac{1}{2}\varepsilon_{i})$ is a regular value for the height function $x_{0}: W_{t} \longrightarrow \R$, then $(W_{t}|_{c}, (\partial W_{t})|_{c})$ is either diffeomorphic to $(W|_{c}, (\partial W)|_{c})$ or to the disk $(D^{2n}, \partial D^{2n})$. 
Furthermore if $c$ is regular value for $x_{0}: W_{1} \longrightarrow \R$, then $(W_{1}|_{c}, (\partial W_{1})|_{c})$ is diffeomorphic to the disk $(D^{2n}, \partial D^{2n})$. 
\end{enumerate}
\end{proposition}

We will now need to iterate the construction from (\ref{defn: glued in surgery move}). 
For what follows let  
$$x = (a, \varepsilon, (W, \ell_{W})) \in \mb{D}^{\partial, \mathcal{H}}_{p} \quad \text{and} \quad
y = (\Lambda, \delta, e, \ell) \in \mb{Y}_{0}(x).$$
For each tuple $t = (t_{0}, \dots t_{p}) \in [0, 1]^{p+1}$, we form the element 
\begin{equation} \label{equation: iterated surgery move}
\mathcal{K}^{t}_{e, \ell}(W, \ell_{W}) = \mathcal{K}^{t_{p}}_{e_{p}, \ell_{p}}\circ\cdots\circ\mathcal{K}^{t_{0}}_{e_{0}, \ell_{0}}(W, \ell_{W}) \; \in \; \bPsi^{\partial}_{\theta}((a_{0}-\varepsilon_{0}, a_{p}+\varepsilon_{p})\times\R_{+}^{\infty})
\end{equation}
by iterating Construction \ref{constriction: implementation of surgery move}. 
Since the embeddings $e_{1}, \dots, e_{p}$ are pairwise disjoint, the above formula does indeed make sense. 
It also follows that one could permute the order of the above composition and still obtain the same resulting family of manifolds.
We will need the following lemma.
\begin{lemma} \label{lemma: surgery implimentation}
The tuple $(a, \tfrac{1}{2}\varepsilon, \mathcal{K}^{t}_{e, \ell}(W, \ell_{W}))$ is an element of $\mb{X}^{\partial, \mb{c}'}_{p}$. 
If $t_{i} = 1$, then for each regular value $b \in (a_{i} - \tfrac{1}{2}\varepsilon_{i}, a_{i}+\tfrac{1}{2}\varepsilon_{i})$ of $x_{0}: \mathcal{K}^{t}_{e, \ell}(W, \ell_{W}) \longrightarrow \R$, the manifold $\mathcal{K}^{t}_{e, \ell}(W, \ell_{W})|_{b}$ is diffeomorphic to the $2n$-dimensional disk $D^{2n}$. 
\end{lemma}
Proving Lemma \ref{lemma: surgery implimentation} will require a
preliminary proposition.  For what follows, let $((W, \partial_{1}W), M, N)$ be a $(2n+1)$-dimensional
relative cobordism triple.  Let
\begin{equation} \label{equation: fixed cylinder}
\varphi: (D^{n}\times I, \; \partial D^{n}\times I) \; \longrightarrow \; (W, \; \partial_{1}W)
\end{equation}
be an embedding with 
$$\varphi(D^{n}\times\{0\}) \subset M \quad \text{and} \quad \varphi(D^{n}\times\{1\}) \subset N.$$
Let 
$(W'; (\partial_{1}W)', N')$ be the triad obtained by setting 
$$
W' = W\setminus\varphi(D^{n}\times I),  \quad (\partial_{1}W)' = \partial_{1}W\setminus\varphi(\partial D^{n}\times I), \quad
N' = N\setminus\varphi(D^{n}\times\{1\}).
$$
 \begin{proposition} \label{proposition: cylinder deletion connectivity}
Let $2n+1 \geq 9$ and let $((W, \partial_{1}W), M, N)$ and the
embedding $\varphi$ be exactly as above.  Suppose that
$(\partial_{1}W, \partial N)$ is $1$-connected and that $(W, N)$ is
$2$-connected.  Then the map $ \pi_{i}(W'; (\partial_{1}W)', N')
\longrightarrow \pi_{i}(W; \partial_{1}W, N) $ is an isomorphism for
$i \leq n$.
\end{proposition}
\begin{proof}
Let $(\partial N)'$ equal the intersection $N'\cap\partial_{1}W$. 
We will prove that 
$\pi_{i}(W', N') \longrightarrow \pi_{i}(W, N)$ 
is an isomorphism for $i \leq n$ and that 
$\pi_{i}((\partial_{1}W)', (\partial N)') \longrightarrow \pi_{i}(\partial_{1}W, \partial N)$
is an isomorphism for $i \leq n-1$ and epimorphism for $i = n$. 
Using the five-lemma, 
these two facts together imply the statement of the proposition. 

Let us first focus on the map $\pi_{i}(W', N') \longrightarrow
\pi_{i}(W, N)$.  We will show that the homomorphism $\pi_{n}(W', N')
\longrightarrow \pi_{n}(W, N)$ is an isomorphism.  The proof for other
values of $i \leq n$ is similar (and actually easier).  To prove
surjectivity let $f: (D^{n}, \partial D^{n}) \longrightarrow (W, N)$
represent an element of $\pi_{n}(W, N)$.  Since $n < \dim(W)/2$, by
putting the map into general position we may assume that $f$ is an
embedding.  By applying Theorem \ref{theorem: half whitney trick}, we
may construct an isotopy $f_{t}: (D^{n}, \partial D^{n})
\longrightarrow (W, N)$ with $f_{0} = f$ such that
$f_{1}(D^{n})\cap\varphi(D^{n}\times[0,1]) = \emptyset$ and thus,
$f_{1}$ represents an element of $\pi_{n}(W', N')$.  This proves that
the map $\pi_{n}(W', N') \longrightarrow \pi_{n}(W, N)$ is surjective.

We now prove that $\pi_{n}(W', N') \longrightarrow \pi_{n}(W, N)$ is injective. 
Let 
$$
F: (D^{n+1}_{+}, \partial_{0}D^{n+1}_{+}) \; \longrightarrow \; (W, N)
$$
be a map with $F(\partial_{1}D^{n+1}_{+})\cap\varphi(D^{n}\times[0,1]) = \emptyset$.
It will suffice to construct a homotopy 
$$F_{t}: (D^{n+1}_{+}, \partial_{0}D^{n+1}_{+}) \; \longrightarrow \; (W, N),$$ 
fixed on $(\partial_{1}D^{n+1}_{+}, \partial_{0,1}D^{n+1}_{+})$, with $F_{0} = F$ such that $F_{1}(D^{n+1}_{+})\cap\varphi(D^{n}\times[0,1]) = \emptyset$. 
By general position we may first assume that the restriction $F|_{\partial_{1}D^{n+1}_{+}}:  (\partial_{1}D^{n+1}_{+}, \partial_{0,1}D^{n+1}_{+}) \longrightarrow (W', N')$ is an embedding and then by applying \cite[Theorem 1]{H 69}, we deform $F$, through a deformation that fixes $\partial_{1}D^{n+1}_{+}$, so as to make $F$ into an embedding. 
Since the pair $(W, N)$ is $2$-connected, and $(D^{n+1}_{+}, \partial_{0}D^{n+1}_{+})$ and $(D^{n}\times[0,1], D^{n}\times\{1\})$ are both $1$-connected, 
we may apply Theorem \ref{theorem: higher half whitney trick} to obtain an isotopy 
$$F_{t}: (D^{n+1}_{+}, \partial_{0}D^{n+1}_{+}) \; \longrightarrow \; (W, N),$$ 
that is constant on $(\partial_{1}D^{n+1}_{+}, \partial_{0,1}D^{n+1}_{+})$, with $F_{0} = F$, such that $F_{1}(D^{n+1}_{+})\cap\varphi(D^{n}\times[0,1]) = \emptyset$.
This finishes the proof that $\pi_{n}(W', N') \longrightarrow \pi_{n}(W, N)$ is an isomorphism. 

We now prove that $\pi_{i}((\partial_{1}W)', (\partial N)') \longrightarrow \pi_{i}(\partial_{1}W, \partial N)$ is an isomorphism for $i \leq n-1$ and epimorphism for $i = n$. 
In the case that $i < n-1$ the proof follows by general position. 
Similarly, the proof that 
$$\pi_{n-1}((\partial_{1}W)', (\partial N)') \longrightarrow \pi_{n-1}(\partial_{1}W, \partial N)$$ 
is surjective follows by general position as well. 
To prove that this map is injective, let 
$$H: (D^{n}_{+}, \partial_{0}D^{n}_{+}) \longrightarrow (\partial_{1}W, \partial N)$$ 
be a map with 
$$\left(H(\partial_{1}D^{n}_{+}), \; H(\partial_{0,1}D^{n}_{+})\right)
\; \subset \; ((\partial_{1}W)', \; (\partial N)').$$ As above it will
suffice to construct a homotopy $H_{t}: (D^{n}_{+},
\partial_{0}D^{n}_{+}) \longrightarrow (\partial_{1}W, \partial N)$,
which is fixed on $(\partial_{1}D^{n}_{+},
\partial_{0,1}D^{n}_{+})$ and such that
$H(D^{n}_{+})\cap\varphi(\partial D^{n}\times[0,1]) = \emptyset$.  The
existence of such a homotopy follows from Theorem \ref{theorem: hudsen embedding theorem}.  
  This proves that $\pi_{n-1}((\partial_{1}W)', (\partial N)') \longrightarrow \pi_{n-1}(\partial_{1}W, \partial N)$
is an isomorphism. 
The proof that $\pi_{n}((\partial_{1}W)', (\partial N)') \longrightarrow \pi_{n}(\partial_{1}W, \partial N)$ is surjective follows by a similar argument applied to a map 
$(D^{n}, \partial D^{n-1}) \longrightarrow (\partial_{1}W, \partial N)$. 
One can use Theorem \ref{theorem: hudsen embedding theorem} again to deform any such map so that it factors through $((\partial_{1}W)', (\partial N)') \hookrightarrow (\partial_{1}W, \partial N)$.
This concludes the proof of the proposition. 
\end{proof}
\begin{proof}[Proof of Lemma \ref{lemma: surgery implimentation}]
Let us denote $W_{t}:=\mathcal{K}^{t}_{e, \ell}(W, \ell_{W})$.
  We will first verify that for any two regular values $a < b \in
  \bigcup_{i=0}^{p}(a_{i}-\tfrac{1}{2}\varepsilon_{i},
  a_{i}+\tfrac{1}{2}\varepsilon_{i})$ of the height function $x_{0}:
  W_{t} \longrightarrow \R$, the triad \ \ $ \left(W_{t}|_{[a, b]}; \;
  (\partial W_{t})|_{[a, b]}, \; W_{t}|_{b}\right) $ \ \ is
  $n$-connected and that the pair \ $ \left((\partial W_{t})|_{[a,
      b]}, \; (\partial W_{t})|_{b}\right) $ \ is $(n-1)$-connected,
  and thus $W_{t} \in \mb{X}^{\partial, \mb{c}'}_{p}$
  for all $t \in [0,1]$.  We will focus on proving that the triad
  $\left(W_{t}|_{[a, b]}; \; (\partial W_{t})|_{[a, b]}, \;
  W_{t}|_{b}\right)$ is $n$-connected. 
  A proof that the pair $\left((\partial W_{t})|_{[a, b]}, \; (\partial W_{t})|_{b}\right)$
  is $(n-1)$-connected is similar.
  
  Using $\psi$ from (\ref{equation:
  reparametrisation}), we may consider the interval
$(a_{i}-\tfrac{1}{2}\varepsilon_{i},
a_{p}+\tfrac{1}{2}\varepsilon_{p})$ to be the range of the height
function on $\mathcal{P}_{t}$.  In this way we may write
$\mathcal{P}_{t}|_{[a, b]}$.  For each $t$, the pair $(W_{t}|_{[a, b]}, \; (\partial W)_{t}|_{[a, b]})$ is obtained from $(W|_{[a, b]}, \; W|_{b})$ by removing embedded copies of
$(\mathcal{P}_{0}|_{[a, b]}, \; (\partial\mathcal{P}_{0})|_{[a, b]})$
and gluing in copies of $(\mathcal{P}_{t}|_{[a, b]}, \; (\partial\mathcal{P}_{t})|_{[a, b]})$.  
Let $(X, \partial X)$ denote
the pair obtained from $(W|_{[a, b]}, \; (\partial W)|_{[a, b]})$ by deleting the copies of $(\mathcal{P}_{0}|_{[a, b]}, \; (\partial\mathcal{P}_{0})|_{[a, b]})$.  
Similarly we denote by
$(X|_{b}, \; (\partial X)|_{b})$ the pair obtained from $(W|_{b},
(\partial W)|_{b})$ by removing $(\mathcal{P}_{0}|_{b}, \;
(\partial\mathcal{P}_{0})|_{b})$.  
We have pushout squares (of pairs of spaces),
\begin{equation} \label{equation: pushout square 1b}
\xymatrix{ \textcolor{black}{\left(X|_{b}, \; (\partial X)|_{b}\right)}
  \ar[rr] \ar[d] && \textcolor{black}{\left(W_{t}|_{b}, \; (\partial
    W)|_{b}\right)} \ar[d] \\ \textcolor{black}{\left(X, \; \partial
    X\right)} \ar[rr] && \textcolor{black}{\left(W_{t}|_{b}\cup X, \;
    (\partial W_{t})|_{b}\cup\partial X\right)}, }
\end{equation}
and 
\begin{equation} \label{equation: pushout square 2b}
\xymatrix{
  \textcolor{black}{\left(\coprod_{\lambda \in \Lambda}\mathcal{P}_{t}|_{b_{\lambda}}, \; \coprod_{\lambda \in \Lambda}(\partial\mathcal{P}_{t})|_{b_{\lambda}}\right)} \ar[d] \ar[rr] && \textcolor{black}{\left(W_{t}|_{b}\cup X, \;
    (\partial W_{t})|_{b}\cup\partial X\right)}  \ar[d] \\
  \textcolor{black}{\left(\coprod_{\lambda \in \Lambda}\mathcal{P}_{t}|_{[a_{\lambda}, b_{\lambda}]}, \;
    \coprod_{\lambda \in \Lambda}(\partial\mathcal{P}_{t})|_{[a_{\lambda}, b_{\lambda}]}\right)}\ar[rr] && \textcolor{black}{\left(W_{t}|_{[a, b]}, \; (\partial W_{t})|_{[a, b]}\right)}.
}
\end{equation}
By Proposition \ref{proposition: cylinder deletion connectivity} implies that the left-vertical map of (\ref{equation: pushout square 1b}) is $n$-connected. 
Since the diagram is a homotopy-pushout (of space pairs) it follows that the right-vertical map 
\begin{equation} \label{equation: vertical map 1}
\textcolor{black}{\left(W_{t}|_{b}, \; (\partial W)|_{b}\right)} \; \longrightarrow \; \left(W_{t}|_{b}\cup X, \; (\partial W_{t})|_{b}\cup\partial X\right)
\end{equation}
is $n$-connected as well. 
Proposition \ref{proposition: basic properties of object surgery move} (condition (iii)) implies that the left-vertical map of (\ref{equation: pushout square 2b}) is $n$-connected. 
Since the diagram is a homotopy-pushout it follows that the right-vertical map 
\begin{equation} \label{equation: vertical map 2}
\textcolor{black}{\left(W_{t}|_{b}\cup X, \; (\partial W_{t})|_{b}\cup\partial X\right)}  \longrightarrow 
\textcolor{black}{\left(W_{t}|_{[a, b]}, \; (\partial W_{t})|_{[a, b]}\right)}
\end{equation}
is $n$-connected. 
By composing (\ref{equation: vertical map 1}) and (\ref{equation: vertical map 2}) we see that the inclusion 
$$
\textcolor{black}{\left(W_{t}|_{b}, \; (\partial W)|_{b}\right)} \longrightarrow
\textcolor{black}{\left(W_{t}|_{[a, b]}, \; (\partial W_{t})|_{[a, b]}\right)}
$$ is $n$-connected, and thus the triad
$\textcolor{black}{\left(W_{t}|_{[a, b]}; \; (\partial W_{t})|_{[a,
      b]}, \; W_{t}|_{b} \right)}$ is $n$-connected.  The fact that
the pair $\textcolor{black}{\left((\partial W_{t})|_{[a, b]}, (\partial
  W_{t})|_{b}\right)}$ is $(n-1)$-connected follows essentially the
same argument using the second part of Proposition \ref{proposition:
  cylinder deletion connectivity} together with Proposition
\ref{proposition: basic properties of object surgery move}.

The other statements in the Lemma then follow from Proposition \ref{proposition: observation 1 about glued in move}. 
\end{proof}
\subsection{Proof of Theorem \ref{theorem: disk object equivalence}} \label{subsection: disk equivalence}
With the results in the previous sections established, the proof of Theorem \ref{theorem: disk object equivalence} now is formally the same as the proof of \cite[Theorem 5.3]{GRW 14}.
We provide an outline.
As a first step we need the following result which follows from \cite[Section 6.1]{GRW 14} (see the proof on page 327).
\begin{proposition} \label{proposition: inclusion of 0 simplices}
The map induced by inclusion of $0$-simplices
$
|\mb{D}^{\partial, \mathcal{H}}_{\bullet, 0}| \longrightarrow |\mb{D}^{\partial, \mathcal{H}}_{\bullet, \bullet}|
$
is a weak homotopy equivalence. 
\end{proposition}
Combining the above proposition with Theorem \ref{theorem: contractibility of space of surgery data} it follows that the composition
$$
|\mb{D}^{\partial, \mathcal{H}}_{\bullet, 0}| \longrightarrow |\mb{D}^{\partial, \mathcal{H}}_{\bullet, \bullet}| \longrightarrow |\mb{D}^{\partial, \mb{c}'}_{\bullet}| = |\mb{D}^{\partial, \mathcal{H}}_{\bullet, -1}|
$$
is a weak homotopy equivalence, where the second map is the induced by the augmentation.

Using Definition \ref{defn: glued in surgery move} and its iterated version from (\ref{equation: iterated surgery move}), we define a map 
$$
\mathcal{F}_{p}: [0, 1]^{p+1}\times\mb{D}^{\partial, \mathcal{H}}_{p, 0} \; \longrightarrow \; \mb{X}^{\partial, \mb{c}'}_{p}, \quad \quad
(t, \; (a, \varepsilon, (W, \ell_{W}), e, \ell)) \; \mapsto \; (a, \tfrac{1}{2}\varepsilon, \mathcal{K}^{t}_{e, \ell}(W, \ell_{W})),
$$
which has the desired range by Lemma \ref{lemma: surgery implimentation}. 
Lemma \ref{lemma: surgery implimentation} also implies that $\mathcal{F}_{p}$ sends $(1, \dots, 1)\times\mb{D}^{\partial, \mathcal{H}}_{p, 0}$ to 
$\mb{X}^{\partial, \mathcal{H}}_{p} \; \subset \; \mb{X}^{\partial, \mb{c}'}_{p}$.
Furthermore, by Proposition \ref{proposition: basic properties of object surgery move} we have equality $\mathcal{K}^{0}_{e_{i}, \ell_{i}}(W', \ell_{W'}) = (W', \ell_{W'})$.
We thus obtain the formula 
\begin{equation}
d_{i}\mathcal{F}_{p}(d^{i}t, \; x) \; = \; \mathcal{F}_{p-1}(t, \; d_{i}x)
\end{equation}
where $d^{i}: [0, 1]^{p} \longrightarrow [0, 1]^{p+1}$ is the embedding that adds a zero in the $i$th coordinate and $d_{i}$ is the $i$th face map associated to the semi-simplicial space $\mb{X}^{\partial, \mb{c}'}_{\bullet}$.

We wish to assemble the maps $\mathcal{F}_{p}$ to a homotopy $\mathcal{F}: [0, 1]\times|\mb{D}^{\partial, \mathcal{H}}_{p, 0}| \longrightarrow |\mb{X}^{\partial, \mb{c}'}_{p}|$.
By following \cite[p. 312 -313]{GRW 14} verbatim, we obtain the following proposition.
\begin{proposition} \label{proposition: assembled homotopy}
The maps $\mathcal{F}_{p}$ induce a homotopy 
$\mathcal{F}: [0, 1]\times|\mb{D}^{\partial, \mathcal{H}}_{\bullet, 0}| \longrightarrow |\mb{X}^{\partial, \mb{c}'}_{\bullet}|$
with the following properties:
\begin{enumerate} \itemsep.2cm
\item[\textcolor{black}{(i)}] The map $\mathcal{F}(0, \; \underline{\hspace{.2cm}})$ is equal to the composition 
$$
|\mb{D}^{\partial, \mathcal{H}}_{\bullet, 0}| \longrightarrow |\mb{D}^{\partial, \mathcal{H}}_{\bullet, \bullet}|\longrightarrow |\mb{D}^{\partial, \mb{c'}}_{\bullet}| \longrightarrow |\mb{X}^{\partial, \mb{c}'}_{\bullet}|,
$$
which is a weak homotopy equivalence by Theorem \ref{theorem: contractibility of space of surgery data}, and Propositions \ref{proposition: X-D equivalence} and \ref{proposition: inclusion of 0 simplices}; 
\item[\textcolor{black}{(ii)}] 
The map $\mathcal{F}(1, \; \underline{\hspace{.2cm}})$ factors through the subspace 
$|\mb{X}^{\partial, \mathcal{H}}_{\bullet}| \; \subset \; |\mb{X}^{\partial, \mb{c}'}_{\bullet}|.$
\end{enumerate}
\end{proposition}
Observe that if $x = (a, \varepsilon, (W, \ell_{W})) \in \mb{D}^{\partial, \mb{c}'}_{p}$ is contained in the subspace $\mb{D}^{\partial, \mathcal{H}}_{p} \subset \mb{D}^{\partial, \mb{c}'}_{p}$, then the \textit{empty surgery data} determined by setting $\Lambda = \emptyset$ yields a valid element of $\mb{Y}_{0}(x)$.
Thus, we obtain a simplicial map 
$\mb{D}^{\partial, \mathcal{H}}_{p}  \longrightarrow \mb{D}^{\partial, \mathcal{H}}_{p, 0},$ $x \mapsto (x, \emptyset),$
and an embedding 
$|\mb{D}^{\partial, \mathcal{H}}_{\bullet}| \hookrightarrow |\mb{D}^{\partial, \mathcal{H}}_{\bullet, 0}|.$
Using the homotopy $\mathcal{F}$ we obtain the diagram
$$
\xymatrix{
|\mb{D}^{\partial, \mathcal{H}}_{\bullet}| \ar[d]_{\simeq} \ar@{^{(}->}[rr] && |\mb{D}^{\partial, \mathcal{H}}_{\bullet, 0}| \ar[dll]_{\mathcal{F}(1, \; \underline{\hspace{.2cm}})} \ar[d]^{\simeq}_{\mathcal{F}(0, \; \underline{\hspace{.2cm}})} \\
|\mb{X}^{\partial, \mathcal{H}}_{\bullet}| \ar@{^{(}->}[rr] && |\mb{X}^{\partial, \mb{c}'}_{\bullet}|
}
$$
where the upper-triangle is commutative and the lower-triangle is homotopy commutative. 
Homotopy commutativity implies that the inclusion $|\mb{X}^{\partial, \mathcal{H}}_{\bullet}| \hookrightarrow  |\mb{X}^{\partial, \mb{c}'}_{\bullet}|$ is a weak homotopy equivalence. 
This completes the proof of Theorem \ref{theorem: disk object equivalence}. 
\begin{remark}
The proof of Theorem \ref{theorem: disk object equivalence} given above assumes Theorem \ref{theorem: contractibility of space of surgery data}. 
This theorem is the most technical component of the paper and 
we give a proof in the following section.
\end{remark}

\section{Contractibility of the Space of Surgery Data} \label{section: contractibility of the spaces of surgery data}
In this section we prove Theorem \ref{theorem: contractibility of
  space of surgery data}.  
  The next subsection is devoted to developing a certain simplicial technique used to prove these theorems. 
 
\subsection{Topological flag complexes} \label{subsection: topological flag complexes}
   We will need to use the following definition. 
\begin{defn} \label{defn: topolgical flag complexes}
Let $X_{\bullet} \longrightarrow X_{-1}$ be an augmented
semi-simplicial space.  We say that it is an \text{augmented topological flag complex} if for each $p \in \Z_{\geq 0}$:
\begin{enumerate}
\item[(i)] 
the map $X_{p} \longrightarrow X_{0}\times_{X_{-1}}\cdots\times_{X_{-1}}X_{0}$ to the $(p+1)$-fold fibred product is a homeomorphism onto its image, which is an open subset; 
\item[(ii)] a tuple $(v_{0}, \dots, v_{p}) \in X_{0}\times_{X_{-1}}\cdots\times_{X_{-1}}X_{0}$ is in $X_{p}$ if and only if $(v_{i}, v_{j}) \in X_{1}$ for all $i < j$. 
\end{enumerate}
If elements $v, w \in X_{0}$ lie in the same fibre over $X_{-1}$ and
$(v, w) \in X_{1}$, we say that $w$ is orthogonal to $v$.  If $X_{-1}$
is a single point, we omit the adjective augmented.
\end{defn}
A key theorem that we will need to use is the following from \cite[Theorem 6.2]{GRW 14}.
\begin{theorem}[Galatius, Randal-Williams\textcolor{black}{,} 2014] \label{theorem: flag complex equivalence}
Let $X_{\bullet} \longrightarrow X_{-1}$ be an augmented topological flag complex. 
Suppose that the following are true:
\begin{enumerate}
\item[(i)] The map $\varepsilon: X_{0} \longrightarrow X_{-1}$ has local lifts of any map from a disk. 
\item[(ii)] $\varepsilon: X_{0} \longrightarrow X_{-1}$ is surjective. 
\item[(iii)] For any $x \in X_{-1}$ and any (non-empty) finite set $\{v_{1}, \dots, v_{m}\} \subset \varepsilon^{-1}(x)$, there exists an element $v \in \varepsilon^{-1}(x)$ such that $(v_{i}, v) \in X_{1}$ for all $i$.  
\end{enumerate}
Then the map $|X_{\bullet}| \longrightarrow X_{-1}$ is a weak homotopy
equivalence.
\end{theorem}
Let $X_{\bullet} \longrightarrow X_{-1}$ be an augmented topological flag complex.
To prove that the induced map $|X_{\bullet}| \longrightarrow X_{-1}$ is a weak homotopy equivalence, it will suffice to verify conditions (i), (ii), and (iii) from the statement of the above theorem. 
For certain augmented topological flag complexes, it is difficult to verify condition (iii) for an arbitrary set of vertices $\{v_{1}, \dots, v_{m}\} \subset \varepsilon^{-1}(x)$.
We will need to use a strengthened version of Theorem \ref{theorem: flag complex equivalence} in which condition (iii) is slightly weakened.  
We must first introduce some new terminology. 

\begin{defn} \label{notation: relation on the zero simplices}
Let $X_{\bullet} \rightarrow X_{-1}$ be an augmented topological flag complex. 
A subset 
$$\mathcal{R} \subset X_{0}\times_{X_{-1}}X_{0}$$ 
is said to be a \textit{symmetric relation} (on the zero-simplices) if it is invariant under the ``coordinate permutation map'' $(v, w) \mapsto (w, v)$. 
For any $x \in X_{-1}$, a subset $\{v_{1}, \dots, v_{m}\} \subset \varepsilon^{-1}(x)$ is said to be in \textit{general position} (with respect to $\mathcal{R}$) if $(v_{i}, v_{j}) \in \mathcal{R}$ for all $i \neq j$.
\end{defn}
For our main result we will need to consider a symmetric relation $\mathcal{R}$ that is open and dense as a subset of $X_{0}\times_{X_{-1}}X_{0}$.
The following result is a strengthening of Theorem \ref{theorem: flag complex equivalence}. 
\begin{theorem} \label{theorem: improved flag complex theorem}
Let $X_{\bullet} \longrightarrow X_{-1}$ be an augmented topological flag complex that satisfies conditions (i) and (ii) of Theorem \ref{theorem: flag complex equivalence}.
Let $\mathcal{R} \subset X_{0}\times_{X_{-1}}X_{0}$ be an open and dense symmetric relation with the property that $X_{1} \subset \mathcal{R}$.
Suppose that $X_{\bullet} \longrightarrow X_{-1}$ satisfies the following further condition:
\begin{enumerate} \itemsep.3cm
\item[(iii)*] Let $x \in X_{-1}$. 
Let $\{v_{1}, \dots, v_{m}\} \subset \varepsilon^{-1}(x)$ be a non-empty subset in general position 
and let $\{w_{1}, \dots, w_{k}\} \subset \varepsilon^{-1}(x)$ be an arbitrary subset (possibly empty) such that $(v_{i}, w_{j}) \in X_{1}$ for all $i, j$. 
Given any such $x$, $\{v_{1}, \dots, v_{m}\}$, and $\{w_{1}, \dots, w_{k}\}$, there exists $v \in \varepsilon^{-1}(x)$ such that $(v, v_{i}) \in X_{1}$ and $(v, w_{j}) \in X_{1}$ for all $i, j$.
\end{enumerate}
Then the induced map $|X_{\bullet}|  \longrightarrow X_{-1}$ is a weak homotopy equivalence. 
\end{theorem}

\begin{remark}
The above theorem is based on \cite{GRW 14b}, \textit{Erratum to: stable moduli spaces of high-dimensional manifolds}; its proof is essentially an abstraction of the techniques from this erratum.
The above theorem will be used to prove that the augmented topological flag complex $\mb{D}^{\partial, \mathcal{H}}_{p, \bullet} \longrightarrow \mb{D}^{\partial, \mathcal{H}}_{p, -1}$ 
induces a weak homotopy equivalence $|\mb{D}^{\partial, \mathcal{H}}_{p, \bullet}| \simeq \mb{D}^{\partial, \mathcal{H}}_{p, -1}$ for all $p \in \Z_{\geq 0}$.
For this particular topological flag complex, it is not possible to verify condition of (iii) of Theorem \ref{theorem: flag complex equivalence}. 
However, we can prove that it does satisfy (iii)* (see Lemma \ref{proposition: disjunction condition iii}) with respect to a certain open-dense symmetric relation defined in Definition \ref{defn: transversality relation}.
\end{remark}

\begin{proof}[Proof of Theorem \ref{theorem: improved flag complex theorem}]
We will need a new construction. 
Let $X_{\bullet} \longrightarrow X_{-1}$ and $\mathcal{R} \subset X_{0}\times_{X_{-1}}X_{0}$ be as in the statement of Theorem \ref{theorem: improved flag complex theorem}.
We let $E_{\bullet, \bullet} \longrightarrow E_{-1}$ be the augmented bi-semi-simplicial space defined by setting $E_{p,q}$ equal to the space of lists $(v_{0}, \dots, v_{p}, w_{0}, \dots, w_{q}) \in X_{0}\times_{X_{-1}}\cdots\times_{X_{-1}}X_{0}$ that satisfy the following conditions:
\begin{itemize} \itemsep.2cm
\item $\{v_{0}, \dots, v_{p}\} \subset X_{0}$ is in general position with respect to $\mathcal{R}$;
\item $(w_{0}, \dots, w_{q}) \in X_{q}$;
\item $(v_{i}, w_{j}) \in X_{1}$ for all $i, j$. 
\end{itemize}
With this definition, the augmentation map $X_{\bullet} \longrightarrow X_{-1}$ may be identified with $E_{-1, \bullet} \longrightarrow E_{-1, -1}$. 
To prove that $|E_{-1, \bullet}| \longrightarrow E_{-1, -1}$ is a weak homotopy equivalence it will suffice to prove the following:
\begin{enumerate} \itemsep.2cm
\item[(a)] $|E_{s, \bullet}| \longrightarrow E_{s, -1}$ is a weak homotopy equivalence for all $s \geq 0$,
\item[(b)] $|E_{\bullet, t}| \longrightarrow E_{-1, t}$ is a weak homotopy equivalence for $t \geq -1$.

\end{enumerate}
Indeed, (a) geometrically realizes to the weak homotopy equivalence $|E_{\bullet, \bullet}| \longrightarrow |E_{\bullet, -1}|$.
The second map, (b), specializes to the weak homotopy equivalence $|E_{\bullet, -1}| \longrightarrow E_{-1, -1}$ and geometrically realizes to the weak homotopy equivalence $|E_{\bullet, \bullet}| \longrightarrow E_{-1, -1}$. 
These three weak homotopy equivalences fit together into the commutative diagram
$$
\xymatrix{
|E_{\bullet, \bullet}| \ar[rr]^{\simeq} \ar[d]^{\simeq} && |E_{\bullet, -1}| \ar[d]^{\simeq} \\
|E_{-1, \bullet}| \ar[rr] && |E_{-1, -1}|,
}
$$
from which it follows that the bottom horizontal map is also a weak homotopy equivalence which proves the theorem. 
To complete the theorem we will need to verify statements (a) and (b) from above. 

\textbf{Proof of (a):} It will suffice to apply Theorem \ref{theorem:
  flag complex equivalence} to the augmented topological flag complex
$E_{s, \bullet} \longrightarrow E_{s, -1}$ for each $s \geq 0$.  By
assumption, condition (i) of Theorem \ref{theorem: flag complex
  equivalence} holds for $X_{\bullet} \longrightarrow X_{-1}$ and it
follows by the definition of $E_{\bullet, \bullet} \longrightarrow
E_{-1}$ that this condition holds for $E_{s,\bullet} \longrightarrow
E_{s, -1}$ for each $s \geq 0$.  We need to verify conditions (ii) and
(iii).  Let $s \geq 0$ and let $x := (v_{0}, \dots, v_{s}) \in E_{s,
  -1}$.  To verify (ii) we need to find $w \in X_{0}$ such that
$(v_{i}, w) \in X_{1}$, which implies that $(v_{0}, \dots, v_{s}, w)
\in E_{s, 0}$.  Since the set $\{v_{0}, \dots, v_{s}\} \subset X_{0}$
is in general position with respect to $\mathcal{R}$, the existence of
such an element $w$ follows from the fact that $X_{\bullet}
\longrightarrow X_{-1}$ satisfies condition (iii)*.  This establishes
condition (ii) for $E_{s,\bullet} \longrightarrow E_{s, -1}$ for $s
\geq 0$.

Now to verify condition (iii). 
Let $x := (v_{0}, \dots, v_{s}) \in E_{s, -1}$ and let $\{w_{1}, \dots, w_{k}\} \subset X_{0}$ be a non-empty subset with the property that $(v_{0}, \dots, v_{s}, w_{j}) \in E_{s, 0}$ for all $j = 1, \dots, k$. 
To prove that condition (iii) holds we will need to find $w \in X_{0}$ such that $(w, v_{i}) \in X_{1}$ and $(w, w_{j}) \in X_{1}$ for all $i, j$. 
Since $\{v_{0}, \dots, v_{s}\}$ is non-empty and in general position with respect to $\mathcal{R}$, the existence of such an element $w \in X_{0}$ follows again from the fact that $X_{\bullet} \longrightarrow X_{-1}$ satisfies condition (iii)*.
This proves that $E_{s, \bullet} \longrightarrow E_{s, -1}$ satisfies condition (iii) whenever $s \geq 0$ and
thus $|E_{s,\bullet}| \longrightarrow E_{s, -1}$  is a weak homotopy equivalence when $s \geq 0$. 
This concludes the proof of (a). 

\textbf{Proof of (b):} 
We apply Theorem \ref{theorem: flag complex equivalence} to the augmented topological flag complex $E_{\bullet, t} \longrightarrow E_{-1, t}$ for $t \geq -1$. 
As before, condition (i) follows by hypothesis. 
We need to verify conditions (ii) and (iii).
Let $t \geq -1$ and let $(w_{0}, \dots, w_{t}) \in E_{-1, t}$.
To prove that condition (ii) holds we need to find $v \in X_{0}$ such that $(v, w_{0}, \dots, w_{t}) \in E_{0, t}$.
By definition, $(w_{i}, w_{j}) \in X_{1}$ for $i < j$ and since $X_{1} \subset \mathcal{R}$ by assumption, it follows that the set $\{w_{0}, \dots, w_{t}\} \subset X_{0}$ is in general position with respect to $\mathcal{R}$. 
The existence of $v \in X_{0}$ with $(v, w_{0}, \dots, w_{t}) \in E_{0, t}$ follows from the fact that $X_{\bullet} \longrightarrow X_{-1}$ satisfies condition (iii)*.
In the case that $t = -1$ (and thus the list $(w_{0}, \dots, w_{t})$ is empty) follows from the fact that $X_{\bullet} \longrightarrow X_{-1}$ satisfies condition (ii).
These observations establish condition (ii) for $E_{\bullet, t} \longrightarrow E_{-1, t}$ for $t \geq -1$.

We now proceed to verify condition (iii).
Let $t \geq -1$.
Let $(w_{0}, \dots, w_{t}) \in E_{-1, t}$ and let $\{v_{1}, \dots, v_{k}\} \subset X_{0}$
be a non-empty subset with the property that $(v_{i}, w_{0}, \dots, w_{t}) \in E_{0, t}$ for $i = 1, \dots, k$. 
To prove that condition (iii) holds we will need to find $v \in X_{0}$ such that $(v, v_{i}) \in \mathcal{R}$ for $i = 1, \dots, k$ and $(v, w_{j}) \in X_{1}$ for $j = 0, \dots, t$.

For $i = 1, \dots, k$, let $U_{i} \subset X_{0}$ denote the subset consisting of all $y \in X_{0}$ for which $(y, v_{i}) \in \mathcal{R}$. 
Since $\mathcal{R} \subset X_{0}\times_{X_{-1}}\times X_{0}$ is open and dense it follows that each $U_{i}$ is an open dense subset of $X_{0}$, and thus $U := \cap_{i=1}^{k}U_{i}$ is an open and dense subset of $X_{0}$. 
In particular, the intersection $U = \cap_{i=1}^{k}U_{i}$ is non-empty. 
Recall that $(v_{1}, w_{j}) \in X_{1}$ for all $j = 1, \dots, t$. 
Since $X_{1} \subset X_{0}\times_{X_{-1}}X_{0}$ is an open subset, there exists an open neighborhood $V \subset X_{0}$ of $v_{1}$ such that $(y, w_{j}) \in X_{1}$ for all $y \in V$ and $j = 0, \dots, t$. 
 Since the subset $U$ is dense in $X_{0}$, it follows that the intersection $V\cap U$ is non-empty. 
We then choose $v \in V\cap U$, and it follows that $(v, w_{j}) \in X_{1}$ for all $j = 0, \dots, t$ and $(v, v_{i}) \in \mathcal{R}$ for all $i = 1, \dots, k$.
It follows that the augmented topological flag complex $E_{\bullet, t} \longrightarrow E_{-1, t}$ satisfies condition (iii) of Theorem \ref{theorem: flag complex equivalence} for $t \geq -1$ and thus the induced map $|E_{\bullet, t}| \longrightarrow E_{-1, t}$ is a weak homotopy equivalence. 
This concludes the proof of (b) and the the proof of the theorem. 
\end{proof}

\subsection{Proof of Theorem \ref{theorem: contractibility of space of surgery data}} \label{subsection: contractibility of surgery objects}
Theorem \ref{theorem: contractibility of space of surgery data} states that when $d = 2n+1 \geq 9$, the
augmentation map $\mb{D}^{\partial, \mathcal{H}}_{\bullet, \bullet}
\longrightarrow \mb{D}^{\partial, \mathcal{H}}_{\bullet, -1}$ induces
the weak homotopy equivalence $|\mb{D}^{\partial,
  \mathcal{H}}_{\bullet, \bullet}| \simeq |\mb{D}^{\partial,
  \mathcal{H}}_{\bullet, -1}|$, where recall that $\mb{D}^{\partial,
  \mathcal{H}}_{\bullet, -1} = \mb{D}^{\partial, \mb{c}'}_{\bullet}$.
 For each $p \in \Z_{\geq 0}$, we obtain an
augmented semi-simplicial space
\begin{equation} \label{equation: flag complex for fixed p}
\mb{D}^{\partial, \mathcal{H}}_{p, \bullet} \longrightarrow \mb{D}^{\partial, \mathcal{H}}_{p, -1}.
\end{equation}
It follows easily from the definition (Definition \ref{defn: objects
  surgery data}) that for each $p$ this is an augmented topological
flag complex as defined in the previous section.  
In order to prove Theorem \ref{theorem: contractibility of space of surgery data}, it will suffice to prove that for each $p \in \Z_{\geq 0}$, the induced map $|\mb{D}^{\partial, \mathcal{H}}_{p, \bullet}| \longrightarrow \mb{D}^{\partial, \mathcal{H}}_{p, -1}$ is a weak homotopy equivalence. 
Geometrically realizing the $p$-coordinate then yields the weak homotopy equivalence $|\mb{D}^{\partial, \mathcal{H}}_{\bullet, \bullet}|
\simeq |\mb{D}^{\partial, \mathcal{H}}_{\bullet, -1}|$ and the proof
of Theorem \ref{theorem: contractibility of space of surgery data}.

The augmented topological flag complex from (\ref{equation: flag complex for fixed p}) is difficult to analyze directly. 
We will need to replace this augmented topological flag complex 
with one that is more ``flexible''.
Let us define the \textit{core} 
\begin{equation} \label{equation: the core}
C \; = \; (-6, -2)\times\{0\}\times D^{n+1}_{+} \;  \; \subset \; \; (-6, -2)\times\R^{n+1}\times D^{n+1}_{+}.
\end{equation}
\begin{defn} \label{defn: flexible verison}
For $x = (a, \varepsilon, (W, \ell_{W})) \in \mb{D}^{\partial, \mb{c}'}_{p}$, let $\widetilde{\mb{Y}}_{\bullet}(x)$ be the semi-simplicial space defined as in Definition \ref{defn: objects surgery data} except now we only ask the map $e$ to be a smooth embedding on a neighborhood of the core
$$
\Lambda\times C \; \; \subset \;  \; \Lambda\times (-6, -2)\times\R^{n+1}\times D^{n+1}_{+}.
$$
Let $\widetilde{\mb{D}}^{\partial, \mathcal{H}}_{\bullet, \bullet}$ be the  bi-semi-simplicial space defined by setting
$
\widetilde{\mb{D}}^{\partial, \mathcal{H}}_{p, q} = \{(x, y) \; | \; x \in \mb{D}^{\partial, \mb{c}'}_{p}, \; y \in \widetilde{\mb{Y}}_{\bullet}(x)\}.
$
By setting $\widetilde{\mb{D}}^{\partial, \mathcal{H}}_{\bullet, -1} = \mb{D}^{\partial, \mb{c}'}_{\bullet}$,
we obtain an augmented bi-semi-simplicial space 
$\widetilde{\mb{D}}^{\partial, \mathcal{H}}_{\bullet, \bullet} \longrightarrow \widetilde{\mb{D}}^{\partial, \mathcal{H}}_{\bullet, -1}$
defined in the same way as before. 
\end{defn}

The lemma below is proven in exactly the same way as \cite[Proposition 6.15]{GRW 14} and so we omit the proof.
\begin{lemma} \label{lemma: flexible version equivalence}
The inclusion map $\mb{D}^{\partial, \mathcal{H}}_{\bullet, \bullet} \hookrightarrow \widetilde{\mb{D}}^{\partial, \mathcal{H}}_{\bullet, \bullet}$ induces a weak homotopy equivalence on each bi-degree, and thus induces a weak homotopy equivalence on geometric realization.
\end{lemma}
By Lemma \ref{lemma: flexible version equivalence}, in order to prove
Theorem \ref{theorem: contractibility of space of surgery data}, it
will suffice to prove that for each $p \in \Z_{\geq 0}$, the induced map 
$|\widetilde{\mb{D}}^{\partial, \mathcal{H}}_{p, \bullet}| \longrightarrow \widetilde{\mb{D}}^{\partial, \mathcal{H}}_{p, -1}$ 
is a weak homotopy equivalence. 
We will prove this by applying Theorem \ref{theorem: improved flag complex theorem}. 
We will need a suitable open, dense, symmetric relation on the zero simplices of $\widetilde{\mb{D}}^{\partial,\mathcal{H}}_{p, \bullet}$ an in Definition \ref{notation: relation on the zero simplices}. 
Recall the core 
$$C \; = \; (-6, -2)\times\{0\}\times D^{n+1}_{+} \;  \; \subset \; \; (-6, -2)\times\R^{n+1}\times D^{n+1}_{+}.$$
We will denote 
$$\partial_{1}C := (-6, -2)\times\{0\}\times\partial_{1}D^{n+1}_{+} \quad \text{and} \quad \partial_{0, 1}C := (-6, -2)\times\{0\}\times\partial_{0, 1}D^{n+1}_{+}.$$
\begin{defn} \label{defn: transversality relation}
We define $\mathcal{T} \subset \widetilde{\mb{D}}^{\partial, \mathcal{H}}_{p, 0}\times_{\widetilde{\mb{D}}^{\partial, \mathcal{H}}_{p, -1}}\widetilde{\mb{D}}^{\partial, \mathcal{H}}_{p, 0}$ to be the subset consisting of those 
$$\left((a, \varepsilon, (W, \ell_{W})), \; (\Lambda_{1}, \delta_{1}, e_{1}, \ell_{1}), (\Lambda_{2}, \delta_{2}, e_{2}, \ell_{2})\right)$$ 
such that the images 
$$\begin{aligned}
e_{1}(\Lambda_{1}\times\partial_{1}C), \; e_{2}(\Lambda_{2}\times\partial_{1}C) &\subset W, \quad \text{and} \\
e_{1}(\Lambda_{1}\times\partial_{0,1}C), \; e_{2}(\Lambda_{2}\times\partial_{0,1}C) &\subset \partial W,
\end{aligned}$$
are transversal as submanifolds of $W$ of and $\partial W$ respectively. 
This subset $\mathcal{T}$ is clearly a symmetric relation in the sense of Definition \ref{notation: relation on the zero simplices}. 
By \textit{Thom's transversality theorem} it follows that it is an open and dense subset of the fibred product $\widetilde{\mb{D}}^{\partial, \mathcal{H}}_{p, 0}\times_{\widetilde{\mb{D}}^{\partial, \mathcal{H}}_{p, -1}}\widetilde{\mb{D}}^{\partial, \mathcal{H}}_{p, 0}$.
\end{defn}
We may now proceed to apply Theorem \ref{theorem: improved flag
  complex theorem} to the augmented topological flag complex
$$\widetilde{\mb{D}}^{\partial, \mathcal{H}}_{p, \bullet}
\longrightarrow \widetilde{\mb{D}}^{\partial, \mathcal{H}}_{p, -1}.$$
We will need to verify conditions (i), (ii), and (iii)* with respect
to the open, dense, symmetric relation $\mathcal{T}$ defined in
Definition \ref{defn: transversality relation}.  The first condition,
that the map $\varepsilon: \widetilde{\mb{D}}^{\partial,
  \mathcal{H}}_{p, 0} \longrightarrow \widetilde{\mb{D}}^{\partial,
  \mathcal{H}}_{p, -1}$ has local lifts, follows directly from
\cite[Proposition 6.10]{GRW 14}.  We will need to verify the
conditions (ii) and (iii)*.  We start by verifying (iii)*.  We will
need to use the following preliminary result which is an application
of the results proven in Appendix \ref{section: A generalization of the Whitney trick}.

\begin{proposition} \label{proposition: inductive disjunction}
Let $n \geq 4$.
Let $(W; \partial_{0}W, \partial_{1}W)$ be a manifold triad of dimension $2n+1$.
For $i = 0, \dots, k$, let $(P_{i}; \partial_{0}P_{i}, \partial_{1}P_{i})$ be $(n+1)$-dimensional manifold triads with $(P_{0}; \partial_{0}P_{0}, \partial_{1}P_{0})$ compact. 
Let 
$$
g_{i}: (P_{i}; \partial_{0}P_{i}, \partial_{1}P_{i}) \longrightarrow (W; \partial_{0}W, \partial_{1}W), \quad \quad i = 0, \dots, m,
$$
be embeddings.
Suppose that the following conditions are met:
\begin{enumerate} \itemsep.2cm
\item[(a)]  $(\partial_{0}P_{i}, \partial_{0,1}P_{i})$ is $0$-connected and $(P_{i}, \partial_{1}P_{i})$ is $1$-connected for $i = 0, \dots, k$;
\item[(b)] $(\partial_{0}W, \partial_{0, 1}W)$ is $1$-connected and $(W, \partial_{1}W)$ is $2$-connected;
\item[(c)] the collection of submanifolds $g_{1}(P_{1}), \dots, g_{k}(P_{k}) \subset W$ is pairwise transverse;
\end{enumerate}
Then there exists an isotopy 
$$f^{t}: (P_{0}; \partial_{0}P_{0}, \partial_{1}P_{0}) \; \longrightarrow \; (W; \partial_{0}W, \partial_{1}W), \quad \quad t \in [0,1],$$
with $f^{0} = g_{0}$ such that $f^{1}(P_{0})\cap g_{i}(P_{i}) = \emptyset$ for all $i = 1, \dots, k$. 
\end{proposition}
\begin{proof}
We prove this by induction on $k$. 
If $k = 0$ then there is nothing to show. 
This establishes the base case. 
For the induction step let 
$$
g_{0}, \dots, g_{k}: (P_{i}; \; \; \partial_{0}P_{i}, \; \; \partial_{1}P_{i}) \; \longrightarrow \; (W;\; \partial_{0}W,\; \partial_{1}W) 
$$
be as in the statement of the proposition. 
By the induction assumption we may assume that 
$$g_{0}(P_{0})\cap g_{i}(P_{i}) = \emptyset \quad \quad \text{for $i = 1, \dots, k-1$.}$$
Let $(W'; \partial_{1}W', \partial_{0}W')$ be the (open) manifold triad obtained by setting
$$\begin{aligned}
W' \; &= \; W\setminus\left[\cup_{i=1}^{k-1}g_{i}(P_{i})\right], \\
\partial_{0}W' \; &= \; \partial_{0}W\setminus\left[\cup_{i=1}^{k-1}g_{i}(\partial_{0}P_{i})\right], \\ 
\partial_{1}W' \; &= \; \partial_{1}W\setminus\left[\cup_{i=1}^{k-1}g_{i}(\partial_{1}P_{i})\right].
\end{aligned}$$
Similarly, we let $(V; \partial_{1}V, \partial_{0}V) \subset (P_{k}; \partial_{0}P_{k}, \partial_{1}P_{k})$ be the manifold triad obtained by setting
$$\begin{aligned}
V \; &= \; P_{k}\setminus g_{k}^{-1}(\cup_{i=1}^{k-1}g_{i}(P_{i})), \\
\partial_{0}V \; &= \; \partial_{0}P_{k}\setminus g_{k}^{-1}(\cup_{i=1}^{k-1}g_{i}(\partial_{0}P_{i})), \\
\partial_{1}V \; &= \; \partial_{1}P_{k}\setminus g^{-1}_{k}(\cup_{i=1}^{k-1}g_{i}(\partial_{1}P_{i})). 
\end{aligned}
$$
We let 
$$\bar{g}_{k}: (V; \partial_{0}V, \partial_{1}V) \longrightarrow (W'; \partial_{1}W', \partial_{0}W')$$ 
be the embedding obtained by restricting $g_{k}$.
To prove the proposition it will suffice to find an isotopy 
$$
f^{t}: (P_{0}; \partial_{0}P_{0}, \partial_{1}P_{0}) \longrightarrow (W'; \partial_{0}W', \partial_{1}W'), \quad \quad t \in [0,1],
$$
with $f^{0} = g_{0}$ such that $f^{1}(P_{0})\cap\bar{g}_{k}(V) = \emptyset$.  
We construct this isotopy in two stages. 
The first stage is to apply Theorem \ref{theorem: half whitney trick} to obtain an isotopy that pushes $g_{0}(\partial_{0}P_{0}) \subset \partial_{0}W'$ off of the submanifold $\bar{g}_{k}(\partial_{0}V) \subset \partial_{0}W'$. 
For the second stage, assuming that $g_{0}(\partial_{0}P_{0})\cap\bar{g}_{k}(\partial_{0}V) = \emptyset$, we will apply Theorem \ref{theorem: higher half whitney trick} to obtain an isotopy of $g_{0}$, relative to $\partial_{0}P_{0}$, that pushes the rest of $g_{0}(P_{0})$ off of $g_{k}(V)$. 

\textbf{Stage 1:} As discussed above, we want to apply Theorem \ref{theorem: half whitney trick} (the \textit{half Whitney trick}) to the embeddings 
$$\begin{aligned}
g_{k}|_{\partial_{0}P_{0}}: (\partial_{0}P_{0}, \partial_{0,1}P) &\longrightarrow (\partial_{0}W', \partial_{0,1}W'), \\
\bar{g}_{k}|_{\partial_{0}V}: (\partial_{0}V, \partial_{0,1}V) &\longrightarrow (\partial_{0}W', \partial_{0,1}W'),
\end{aligned}$$
to find an isotopy that pushes the image of $g_{0}|_{\partial_{0}P_{0}}$ off of the image of $\bar{g}_{k}|_{\partial_{0}V}$.
In order to do this we need to verify that the manifold pairs involved satisfy certain connectivity conditions. 
To apply Theorem \ref{theorem: half whitney trick} it is required that the co-domain pair $(\partial_{0}W', \partial_{0,1}W')$ be $1$-connected and that the pairs $(\partial_{0}P_{0}, \partial_{0,1}P)$ and $(\partial_{0}V, \partial_{0,1}V)$ both be $0$-connected. 
The pair $(\partial_{0}P_{0}, \partial_{0,1}P)$ is $0$-connected by assumption (see condition (a) in the statement of the proposition).
The pair $(\partial_{0}W', \partial_{0,1}W')$ is obtained from $(\partial_{0}W, \partial_{0,1}W)$ by deleting a finite collection of $n$-dimensional submanifolds.  
Since $n \geq 4$ and $\dim(\partial_{0}W) = 2n \geq 8$, it follows by a standard general position argument that $\pi_{i}(\partial_{0}W', \partial_{0,1}W') \cong \pi_{i}(\partial_{0}W, \partial_{0,1}W) = 0$ for $i \leq 1$,
hence the pair $(\partial_{0}W', \partial_{0,1}W')$ is $1$-connected.
We now verify that $(\partial_{0}V, \partial_{0,1}V)$ is $0$-connected. 
Recall that by hypothesis, the embeddings $g_{1}|_{\partial_{0}P_{1}}, \dots, g_{k}|_{\partial_{0}P_{k}}$ are pairwise transverse. 
This implies that for each $i = 1, \dots, k-1$, the pre-image 
$$(g_{k}|_{\partial_{0}P_{k}})^{-1}\left(\cup_{i=1}^{k-1}g_{i}(\partial_{0}P_{i})\right) \; \subset \; \partial_{0}P_{k}$$ 
is also a zero-dimensional submanifold. 
The manifold pair $(\partial_{0}V, \partial_{0,1}V)$ is obtained by deleting this zero-dimensional submanifold from $(\partial_{0}P_{k}, \partial_{0,1}P_{k})$. 
Since $n \geq 4$, deleting a finite collection of points from $(\partial_{0}P_{k}, \partial_{0,1}P_{k})$ does not alter its homotopy $0$-type. 
It follows that the pair $(\partial_{0}V, \partial_{0,1}V)$ is indeed $0$-connected.

With the above connectivity conditions verified, we may obtain an isotopy 
$$h^{t}: (\partial_{0}P_{0}, \partial_{0,1}P_{0}) \longrightarrow (\partial_{0}W', \partial_{0,1}W'), \quad t \in [0,1],$$ 
with $h^{0} = g_{0}|_{\partial_{0}P_{0}}$ such that $h_{1}(\partial_{0}P_{0})\cap \bar{g}_{k}(\partial_{0}V) = \emptyset$.
Then, by applying the isotopy extension theorem we obtain an isotopy 
$$
f^{t}: (P_{0}; \partial_{0}P_{0}, \partial_{1}P_{0}) \longrightarrow (W'; \partial_{0}W', \partial_{1}W'), \quad \quad t \in [0,1],
$$
with $f^{0} = g_{0}$ such that $f^{1}(\partial_{0}P)\cap \bar{g}_{k}(\partial_{0}V) = \emptyset$.
 This completes the first stage of the inductive step.
 
 \textbf{Stage 2:} 
 By what was proven above we may assume that $g_{0}(\partial_{0}P_{0})\cap\bar{g}_{k}(\partial_{0}V) = \emptyset$.
 We want to apply Theorem \ref{theorem: higher half whitney trick} to obtain an isotopy of $g_{0}$, relative to $\partial_{0}P_{0}$, that pushes the rest of $g_{0}(P_{0})$ off of $\bar{g}_{k}(V)$.
In order to apply Theorem \ref{theorem: higher half whitney trick} we will need to verify that $(W', \partial_{1}W')$ is $2$-connected, and that both $(V, \partial_{1}V)$ and $(P_{0}, \partial_{1}P_{0})$ are $1$-connected. 
The pair $(P_{0}, \partial_{1}P_{0})$ is $1$-connected by condition (a) in the statement of the proposition.
The pair $(W', \partial_{1}W')$ is obtained from $(W, \partial_{1}W)$ by deleting a finite collection of $(n+1)$-dimensional submanifolds. 
Since $n \geq 4$ and $\dim(W) = 2n+1 \geq 9$, it follows by general position that $\pi_{i}(W', \partial_{1}W') \cong \pi_{i}(W, \partial W) = 0$ for $i \leq 2$, and thus $(W', \partial_{1}W')$ is $2$-connected.
We now verify that $(V, \partial_{1}V)$ is $1$-connected. 
Since $g_{k}$ is transverse to $g_{i}$ for $i = 1, \dots, k-1$, it follows that the pre-image $g^{-1}_{k}(g_{i}(P_{i})) \subset P_{k}$ is a $1$-dimensional submanifold for each $i = 1, \dots, k-1$.
The manifold $V$ is obtained from $P_{k}$ by deleting the submanifolds $g^{-1}_{k}(g_{i}(P_{i})) \subset P_{k}$. 
Since $\dim(P_{k}) \geq n +1 \geq 5$, it follows by general position that $\pi_{i}(P_{k}, \partial_{1}P_{k}) \cong \pi_{i}(V, \partial_{1}V) = 0$ for $i \leq 3$ and thus $(V, \partial_{1}V)$ is $1$-connected (it is actually $3$-connected). 

Since $(W', \partial_{1}W')$, $(V, \partial_{1}V)$, $(D^{n}\times[0,1], \partial D^{n}\times[0,1])$ satisfy all of the required connectivity conditions in the statement of Theorem \ref{theorem: higher half whitney trick}, we obtain an isotopy 
$$
f^{t}: (P_{0}; \partial_{0}P_{0}, \partial_{1}P_{0}) \longrightarrow (W'; \partial_{1}W', \partial_{0}W'), \quad t \in [0,1], 
$$
with $f^{0} = g_{0}$ and $f^{t}|_{\partial_{0}P_{0}} = g_{0}$ for all $t \in [0, 1]$, such that $f^{1}(P_{0})\cap \bar{g}_{k}(V) = \emptyset$. 
This concludes the induction step and the proof of the proposition.
\end{proof}
We can now verify condition (iii)* of
Theorem \ref{theorem: flag complex equivalence} which translates to
the following lemma.
This lemma should be compared to \cite[Proposition 6.19]{GRW 14}. 
\begin{lemma} \label{proposition: disjunction condition iii}
Let $x = (a, \varepsilon, (W, \ell_{W})) \in \mb{D}^{\partial, \mb{c}'}_{p}$. 
Let $\{v_{1}, \dots, v_{m}\} \subset \widetilde{\mb{Y}}_{0}(x)$ be a non-empty subset that is in general position with respect to $\mathcal{T}$.
Let $\{w_{1}, \dots, w_{k}\} \subset \widetilde{\mb{Y}}_{0}(x)$ be an arbitrary subset (possibly empty and not necessarily in general position) with the property that $(v_{i}, w_{j}) \in \widetilde{\mb{Y}}_{1}(x)$ 
for all $i, j$.
Then there exists 
$u \in \widetilde{\mb{Y}}_{0}(x)$ such that 
$$(u, v_{i}) \in \widetilde{\mb{Y}}_{1}(x) \quad \text{and} \quad (u, w_{j}) \in \widetilde{\mb{Y}}_{1}(x)$$ 
for all $i = 1, \dots, m$ and $j = 1, \dots, k$.
\end{lemma}
\begin{proof}[Proof of Lemma \ref{proposition: disjunction condition iii}]
Let the sets $\{v_{1}, \dots, v_{m}\}, \{w_{1}, \dots, w_{k}\} \subset \widetilde{\mb{Y}}_{0}(x)$ be given and let us write $v_{i} = (\Lambda_{v}^{i}, \delta_{v}^{i}, e_{v}^{i}, \ell_{v}^{i})$ and $w_{j} = (\Lambda_{w}^{j}, \delta_{w}^{j}, e_{w}^{j}, \ell_{w}^{j})$ for $i = 1, \dots, m$ and $j = 1, \dots, k$. 
To produce the desired zero-simplex $u \in \widetilde{\mb{Y}}_{0}(x)$ we do the following. 
We first set $u = v_{1}$ and write $u = (\Lambda, \delta, e, \ell)$.
By hypothesis, the image of the core $e(\Lambda\times C)$ is disjoint from $e^{j}_{w}(\Lambda^{j}_{w}\times C)$ for all $j = 1, \dots, k$.
However, it may happen that the image of the core $e(\Lambda\times C)$ has non-empty intersection with $e^{i}_{v}(\Lambda^{i}_{v}\times C)$ for $i = 1, \dots, m$.
To prove the lemma, we will need to make $e(\Lambda\times C)$ disjoint from the $e^{i}_{v}(\Lambda^{i}_{v}\times C)$ while keeping it disjoint from the $e^{j}_{w}(\Lambda^{j}_{w}\times C)$.
Recall from Definition \ref{defn: transversality relation} 
the submanifolds
$$\partial_{1}C := (-6, -2)\times\{0\}\times\partial_{1}D^{n+1}_{+} \quad \text{and} \quad \partial_{0, 1}C := (-6, -2)\times\{0\}\times\partial_{0, 1}D^{n+1}_{+}.$$
Recall that  
$$\left(e^{i}_{v}(\Lambda^{i}_{v}\times\partial_{1}C),\;  e^{i}_{v}(\Lambda^{i}_{v}\times\partial_{0,1}C)\right) \; \subset \; \left(W, \; \partial W\right)$$
for all $i = 0, \dots, p$ (the conditions holds for the $e^{j}_{w}$). 
Our first step is to deform $e(\Lambda\times\partial_{1}C) \subset W$ through an isotopy of embeddings, making it 
disjoint from $e^{i}_{v}(\Lambda^{i}_{v}\times\partial_{1}C)$ for $i = 1, \dots, m$, while keeping it disjoint from the $e^{j}_{w}(\Lambda^{j}_{w}\times\partial_{1}C)$.
This will be achieved by an application of Proposition \ref{proposition: inductive disjunction}.  
To begin we denote 
$$
W' := W\setminus\left[\cup_{j=1}^{k}e^{j}_{w}(\Lambda^{j}_{w}\times \partial_{1}C)\right]. 
$$
By assumption we have $e(\Lambda\times\partial_{1}C), \; e^{i}_{v}(\Lambda^{i}_{v}\times\partial_{1}C) \; \subset \; W'$ for $i = 1, \dots, m$.
For each $l = 0, \dots, p$ let us denote:
$$\begin{aligned}
\partial_{1}C^{i}_{l} &:= e^{i}_{v}(\Lambda^{i}_{v}\times\partial_{1}C)\cap W'|_{[a_{l}, a_{l+1}]}, \\
\partial_{0, 1}C^{i}_{l} &:= e^{i}_{v}(\Lambda^{i}_{v}\times\partial_{0, 1}C)\cap(\partial W')|_{[a_{l}, a_{l+1}]}, \\
\partial_{1}\widehat{C}^{i}_{l} &:= e^{i}_{v}(\Lambda^{i}_{v}\times\partial_{1}C)\cap(W'|_{a_{l}}\sqcup W'|_{a_{l+1}}).
\end{aligned} $$
We define
$$
\partial_{1}C_{l}, \quad \partial_{0, 1}C_{l}, \quad \text{and} \quad \partial_{1}\widehat{C}_{l} 
$$
similarly using the embedding $e$.
It follows from Definition \ref{defn: objects surgery data} (condition (ii)) that each sub-manifold-triad 
$$(\partial_{1}C^{i}_{l}; \; \partial_{0, 1}C^{i}_{l}, \; \partial_{1}\widehat{C}^{i}_{l}) \; \; \subset \; \; (W'|_{[a_{l}, a_{l+1}]}; \;(\partial W')|_{[a_{l}, a_{l+1}]}, \;  W'|_{a_{l}}\sqcup W'|_{a_{l+1}})$$
is diffeomorphic to a finite disjoint union of copies of 
$$
(D^{n}\times[0, 1]; \; \partial D^{n}\times[0, 1], \; D^{n}\times\{0, 1\}).
$$ 
By the assumption that $\{v_{1}, \dots, v_{m}\}$ is in general position with respect to the relation $\mathcal{T}$, it follows that for each $l$, the collection of submanifold triads, 
$$(\partial_{1}C_{l}^{i}; \; \partial_{0, 1}C_{l}^{i}, \; \partial_{1}\widehat{C}^{i}_{l}) \quad \quad i = 1, \dots, m$$ 
is pairwise transverse. 
Since $(a, \varepsilon, (W, \ell_{W})) \in \mb{D}^{\partial, \mb{c}'}_{p}$, the pairs
$$(W|_{[a_{l}, a_{l+1}]}, (\partial W)|_{[a_{l}, a_{l+1}]}) \quad  \text{and} \quad (W|_{a_{l}}\sqcup W|_{a_{l+1}}, (\partial W)|_{a_{l}}\sqcup (\partial W)|_{a_{l+1}})$$ 
are $2$-connected and $1$-connected respectively, and 
the manifold $W'$ is obtained from $W$ by deleting a finite collection of $(n+1)$-dimensional manifolds.
It follows by general position that the pairs
$$\left(W'|_{[a_{l}, a_{l+1}]}, (\partial W')|_{[a_{l}, a_{l+1}]}\right) \quad  \text{and} \quad \left(W'|_{a_{l}}\sqcup W'|_{a_{l+1}}, (\partial W')|_{a_{l}}\sqcup (\partial W')|_{a_{l+1}}\right)$$ 
are $2$-connected and $1$-connected respectively.

We may apply Proposition \ref{proposition: inductive disjunction} to 
$$
(\partial_{1}C_{0}; \; \partial_{0,1}C_{0}, \; \partial_{1}\widehat{C}_{0}) \; \hookrightarrow \; (W'|_{[a_{0}, a_{1}]};  \; (\partial W')|_{[a_{0}, a_{1}]}, \;  W'|_{a_{0}}\sqcup W'|_{a_{1}})
$$
and the collection of inclusion maps
\begin{equation} \label{equation: collection of inclusion maps}
(\partial_{1}C_{0}^{i}; \; \partial_{0, 1}C_{0}^{i}, \; \partial_{1}\widehat{C}^{i}_{0}) \hookrightarrow (W'|_{[a_{0}, a_{1}]}; \;(\partial W')|_{[a_{0}, a_{1}]}, \;  W'|_{a_{0}}\sqcup W'|_{a_{1}}), 
\end{equation}
for $i = i, \dots, m$,
so as to obtain an isotopy that pushes $\partial_{1}C_{0} \subset W'|_{[a_{0}, a_{1}]}$ off of 
$$\partial_{1}C_{0}^{i} \; \subset \; W'|_{[a_{0}, a_{1}]}.$$ 
In our application of Proposition \ref{proposition: inductive disjunction} the inclusion of $(\partial_{1}C_{0}; \; \partial_{0,1}C_{0}, \; \partial_{1}\widehat{C}_{0})$  plays the role of the embedding $g_{0}$ and the inclusion maps of (\ref{equation: collection of inclusion maps}) play the role of the embeddings $g_{1}, \dots, g_{k}$.
With $\partial_{1}C_{0} \subset W'|_{[a_{0}, a_{1}]}$ made disjoint from $\partial_{1}C_{0}$, by applying Proposition \ref{proposition: inductive disjunction} again we may construct an isotopy of
$$
(\partial_{1}C_{1}; \; \partial_{0,1}C_{1}, \; \partial_{1}\widehat{C}_{1}) \; \hookrightarrow \; (W'|_{[a_{1}, a_{2}]};  \; (\partial W')|_{[a_{1}, a_{2}]}, \;  W'|_{a_{1}}\sqcup W'|_{a_{2}}),
$$
constant on the face $\partial_{1}\widehat{C}_{1}\cap W'|_{a_{1}}$, that makes $\partial_{1}C_{l} \subset W'|_{[a_{1}, a_{2}]}$ disjoint from the 
$$\partial_{1}C_{1}^{i} \; \subset \; W'|_{[a_{1}, a_{2}]}$$
for all $i = 1, \dots, m$.
By continuing this process indicatively for the submanifolds $\partial_{1}C_{0}, \dots, \partial_{1}C_{p}$, 
we obtain an embedding 
$
\Lambda\times\partial_{1}C \; \hookrightarrow \; W
$
with image disjoint from $e_{v}^{i}(\Lambda\partial_{1}C) \subset W$ for all $i = 1, \dots, m$.
By application of isotopy extension, 
we obtain a new embedding 
$$
e': \Lambda\times(-6, -2)\times\R^{n+1}\times D^{n+1}_{+} \longrightarrow (a_{0}-\varepsilon_{0}, a_{p}+\varepsilon_{p})\times\R^{\infty}_{+},
$$
that determines an element of $\mb{Z}_{0}(x)$, such that $e'(\Lambda\times\partial_{1}C) \subset W$ is disjoint from the 
$$e^{i}_{v}(\Lambda^{i}_{v}\times \partial_{1}C), \; e^{j}_{w}(\Lambda^{j}_{w}\times\partial_{1}C) \; \subset \; W$$ 
for all $i = 1, \dots, m$ and $j = 1, \dots, k$.

Let us now denote $\widetilde{C} := C\setminus\partial_{1}C$. 
Notice that $e'(\Lambda\times\widetilde{C}) \subset \R\times\R^{\infty}_{+}\setminus W$. 
It remains still to make $e(\Lambda\times\widetilde{C})$ disjoint from the submanifolds $e^{i}_{v}(\Lambda^{i}_{v}\times \widetilde{C})$ and $e^{j}_{w}(\Lambda^{i}_{v}\times \widetilde{C})$.
Since the ambient space $\R\times\R^{\infty}_{+}\setminus W$ is infinite dimensional, this condition is achieved by simply putting $e(\Lambda\times\widetilde{C})$ into general position with respect to the submanifolds
$$e^{i}_{v}(\Lambda^{i}_{v}\times \widetilde{C}), \; e^{j}_{w}(\Lambda^{i}_{v}\times \widetilde{C}) \subset \R\times\R^{\infty}_{+}\setminus W.$$
By what was proven in the previous paragraph, it follows that $e(\Lambda\times\widetilde{C})$ is already disjoint from these submanifolds near $W$. 
The desired disjointness condition can be achieved by a small isotopy that is fixed on a neighborhood of $W$. 
As a result we obtain an embedding $e''$ with $e''(\Lambda\times C)$ disjoint from $e^{i}_{v}(\Lambda^{i}_{v}\times C)$ and $e^{i}_{v}(\Lambda^{j}_{w}\times C)$ for all $i, j$. 
This embedding determines the desired element of $\mb{Z}_{0}(x)$ and concludes the proof of the proposition.
\end{proof}
\begin{remark}
In the proof of the above lemma we had to use Proposition \ref{proposition: inductive disjunction} (which is a corollary of Theorem \ref{theorem: higher half whitney trick} from and Theorem \ref{theorem: half whitney trick} from the appendix). 
This crucial result requires the condition that $2n+1 \geq 9$. 
This is the only place in the paper where this dimensional assumption is required. 
If Proposition \ref{proposition: inductive disjunction} could be strengthened to work in lower dimensions, then so could the above lemma. 
\end{remark}
The proof of condition (ii) will require the following ingredient.
\begin{lemma} \label{lemma: pi-pi theorem}
For $n \geq 3$, let $M$ be a compact, $2n$-dimensional manifold. 
Suppose that $(M, \partial M)$ is $(n-1)$-connected and that $\partial M$ is $(n-2)$-connected. 
Then there exists a finite set $\Lambda$ together with an embedding 
$$
\varphi: (\Lambda\times D^{n}\times \R^{n}, \; \Lambda\times S^{n-1}\times \R^{n}) \longrightarrow (M, \partial M)
$$
such that manifold 
$$
M'  \; = \; M\setminus\varphi(\Lambda\times D^{n}\times \R^{n})
$$
obtained by deleting the image of $\varphi$ is 
 diffeomorphic to the disk $D^{2n}$. 
 In particular, the manifold $M$ can be made into a disk by a collection of disjoint handle-subtraction surgeries of index $n$.
\end{lemma} 
Lemma \ref{lemma: pi-pi theorem} is a special case of the
\textit{$\pi-\pi$-theorem} (see \cite[Section 4]{W 70}).  
Below we give a simple independent proof.
\begin{proof}[Proof of Lemma \ref{lemma: pi-pi theorem}]
By \textit{Lefschetz duality}, the long exact sequence on homology groups associated to the pair $(M, \partial M)$ reduces to the four term exact sequence
$$
\xymatrix{
0 \ar[r] & H_{n}(\partial M) \ar[r] & H_{n}(M) \ar[r] & H_{n}(M, \partial M) \ar[r] & H_{n-1}(\partial M)  \ar[r] & 0.
}
$$
By the fact that $M$ is $(n-1)$-connected, the group $H_{n}(M, \partial M)$ is finitely generated free-abelian.
Let $g \in \Z_{\geq 0}$ be the $\Z$-rank of $H_{n}(M, \partial M)$.
Using the Hurewicz theorem it then follows that there are maps
\begin{equation} \label{equation: full basis}
\phi_{i}: (D^{n}, \; S^{n-1}) \; \longrightarrow \; (M, \partial M) \quad \text{for $i = 1, \dots, g$}
\end{equation}
such that the list
$
(\phi_{1}, \dots, \phi_{g}) 
$
represents a $\Z$-basis for $H_{n}(M, \partial M)$. 
By Theorem \ref{theorem: hudsen embedding theorem} we may assume that each $\phi_{i}$ is an embedding and then by applying Theorem \ref{subsection: the half Whitney trick} inductively, we may assume that the embeddings $\phi_{i}$ are pairwise disjoint. 
Thus by setting $\Lambda = \{1, \dots, g\}$, we obtain an embedding
$$
\phi: (\Lambda\times D^{n}, \Lambda\times S^{n-1}) \longrightarrow (M, \partial M) 
$$
with $\phi|_{\{i\}\times D^{n}} = \phi_{i}$. 
Since the disk is contractible, it follows that the normal bundle of each $\phi_{i}$ is trivial.
It follows that the embedding $\phi$ extends to an embedding 
$$
\varphi: (\Lambda\times D^{n}\times\R^{n}, \; \Lambda\times S^{n-1}\times\R^{n}) \longrightarrow (M, \partial M). 
$$
We now set
$$M' = M\setminus\varphi(\Lambda\times D^{n}\times\R^{n}).$$
By general position it follows that both $M'$ and $\partial M'$ are $(n-2)$-connected. 
We will also need something a little bit stronger which requires more a bit more than a general position argument to prove. 
\begin{claim} \label{claim: n-1 connectivity}
The manifold $M'$ is $(n-1)$-connected. 
\end{claim}
We assume this claim for now and will give its proof after the proof of the current lemma. 
Now, by the excision theorem we have an isomorphism 
$$
H_{i}(M, \; \partial M\cup\varphi(\Lambda\times D^{n}\times\R^{n})) \; \cong \; H_{i}(M', \partial M')
$$
for all $i$. 
Since $(\phi_{1}, \dots, \phi_{g})$ represents a $\Z$-basis for $H_{n}(M, \partial M)$, it follows that 
$$H_{n}(M, \; \partial M\cup\varphi(\Lambda\times D^{n}\times\R^{n})) = 0$$ 
and so $H_{n}(M', \partial M') = 0$. 
Using Claim \ref{claim: n-1 connectivity}, the long exact sequence associated to the pair $(M', \partial M')$ implies that $M'$ is weakly contractible (it is simply connected and all of its homology groups vanish). 
Since $2n \geq 6$ it follows by \cite[Proposition A (page 108)]{M 65} that $M'$ is diffeomorphic to a disk. 
The proof of Lemma \ref{lemma: pi-pi theorem} will be complete once we prove Claim \ref{claim: n-1 connectivity}.
\end{proof}
\begin{proof}[Proof of Claim \ref{claim: n-1 connectivity}.]
Recall the embedding 
$$\varphi: (\Lambda\times D^{n}\times\R^{n}, \; \Lambda\times S^{n-1}\times\R^{n}) \longrightarrow (M, \partial M).$$
 For each $i \in \Lambda = \{1, \dots, g\}$, let $\varphi_{i}: (D^{n}\times\Int(D^{n}), S^{n-1}\times\Int(D^{n}))  \longrightarrow (M, \partial M)$ denote the embedding obtained by restricting $\varphi$ to $\{i\}\times D^{n}\times\Int(D^{n})$. 
To prove the claim it will suffice to show that each manifold 
$$M_{i} := M\setminus\varphi_{i}(D^{n}\times\Int(D^{n}))$$ 
is $(n-1)$-connected. 
It follows by general position that it is $(n-2)$-connected. 
Let $x_{i} \in H_{n}(M, \partial M)$ be the class represented by $\varphi_{i}|_{D^{n}\times\{0\}}: (D^{n}, \partial D^{n}) \longrightarrow (M, \partial M)$.
Let  
$$
\lambda: H_{n}(M)\otimes H_{n}(M, \partial M) \longrightarrow \Z,
$$
be the intersection pairing, which is non-singular by Lefschetz-Duality. 
Let 
$$y \in H_{n}(M, M_{i}) \cong H_{n}(D^{n}\times D^{n}, D^{n}\times S^{n-1}) \cong \Z$$ 
be a generator. 
The map induced by inclusion
\begin{equation} \label{equation: inclusion intersection map}
H_{n}(M) \longrightarrow H_{n}(M, M_{i}),
\end{equation}
is given by the formula $z \mapsto \lambda(z, x_{i})\cdot y$. 
Since $(x_{1}, \dots, x_{g})$ is a $\Z$-basis for $H_{n}(M, \partial M)$ and $\lambda(\underline{\hspace{.3cm}}, \; \underline{\hspace{.3cm}} )$ is non-singular, it follows that there exists $v \in H_{n}(M)$ such that $\lambda(v, x_{i}) = 1$.  
Thus, the map (\ref{equation: inclusion intersection map}) is surjective.
From the exact sequence 
$$
\xymatrix{
H_{n}(M) \ar[r] & H_{n}(M, M_{i}) \ar[r] & H_{n-1}(M_{i}) \ar[r] & H_{n-1}(M) = 0, 
}
$$
it then follows that $H_{n-1}(M_{i}) = 0$ for $i = 1, \dots, g$. 
This concludes the proof of Claim \ref{claim: n-1 connectivity} and the proof of Lemma \ref{lemma: pi-pi theorem}.
\end{proof}

The condition (ii) translates to the following lemma.
\begin{lemma} \label{proposition: nonempty surgery data}
For any element $x = (a, \varepsilon, (W, \ell_{W})) \in \mb{D}^{\partial, \mb{c}'}_{p}$, the space $\widetilde{\mb{Y}}_{0}(x)$ is non-empty. 
\end{lemma}
\begin{proof}
Let $a = (a_{0}, \dots, a_{p})$. 
For $i = 0, \dots, p$, we let $(M_{i}, \partial M_{i})$ denote the level sets $(W|_{a_{i}}, (\partial W)|_{a_{i}})$. 
By Lemma \ref{lemma: pi-pi theorem}, for each $i$ there exists an embedding 
\begin{equation} \label{equation: phi surgery embedding}
\varphi_{i}: (\Lambda_{i}\times\R^{n}\times\partial_{1}D^{n+1}_{+}, \;
\Lambda_{i}\times\R^{n}\times\partial_{0,1}D^{n+1}_{+})
\longrightarrow (M_{i}, \partial M_{i}),
\end{equation}
such that the manifold
$$
M'_{i} \; = \; M\setminus\varphi_{i}(\Lambda_{i}\times\R^{n}\times\partial_{1}D^{n+1}_{+})
$$
is diffeomorphic to the disk $D^{2n}$. 

The first step to producing an element of 
 $\widetilde{\mb{Y}}_{0}(x)$ will be to extend the embeddings $\varphi_{i}$ to maps
$$
e_{i}: \textcolor{black}{\left(\Lambda_{i}\times[a_{i}, a_{p}]\times\R^{n}\times\partial_{1}D^{n+1}_{+}, \; \; \Lambda_{i}\times[a_{i}, a_{p}]\times\R^{n}\times\partial_{0,1}D^{n+1}_{+}\right)} \longrightarrow \textcolor{black}{\left(W|_{[a_{i}, a_{p}]}, \; (\partial W)|_{[a_{i}, a_{p}]}\right)},
$$
such that:
\begin{itemize} \itemsep.2cm
\item  the restriction to 
$\Lambda_{i}\times[a_{i}, a_{p}]\times\{0\}\times\partial_{1}D^{n+1}_{+}$
is an embedding for $i = 0, \dots, p$;
\item 
$e_{i}(\Lambda_{i}\times[a_{i}, a_{p}]\times\{0\}\times\partial_{1}D^{n+1}_{+})\bigcap e_{j}(\Lambda_{j}\times[a_{j}, a_{p}]\times\{0\}\times\partial_{1}D^{n+1}_{+}) \; = \; \emptyset$ \; \;  whenever $i \neq j$.
\end{itemize}
\vspace{.2cm}

For each $i = 0, \dots, p$, the triad 
$(W|_{[a_{i}, a_{p}]}; \; W|_{a_{p}}, \; (\partial W)|_{[a_{i}, a_{p}]})$ 
is $n$-connected and hence, the homomorphism 
\begin{equation} \label{equation: pi-n surjectivity}
\pi_{n}(W|_{a_{p}}, \; (\partial W)|_{a_{p}}) \longrightarrow \pi_{n}(W|_{[a_{i}, a_{p}]}, \; (\partial W)|_{[a_{i}, a_{p}]})
\end{equation}
is surjective. 
It follows from surjectivity of (\ref{equation: pi-n surjectivity}) that for each $\lambda \in \Lambda_{i}$, the embedding 
$$
\left(\{\lambda\}\times\R^{n}\times\partial_{1}D^{n+1}_{+}, \; \{\lambda\}\times\R^{n}\times\partial_{0,1}D^{n+1}_{+}\right) \; \longrightarrow \;
\left(W|_{[a_{i}, a_{p}]}, \; (\partial W)|_{[a_{i}, a_{p}]}\right),
$$
given by restricting (\ref{equation: phi surgery embedding}) and then including $(M_{i}, \partial M_{i})$ into $(W|_{[a_{i}, a_{p}]}, \; (\partial W)|_{[a_{i}, a_{p}]})$,
extends to a map 
$$
e_{i, \lambda}: \textcolor{black}{\left(\{\lambda\}\times[a_{i}, a_{p}]\times\R^{n}\times\partial_{1}D^{n+1}_{+}, \; \{\lambda\}\times[a_{i}, a_{p}]\times\R^{n}\times\partial_{0,1}D^{n+1}_{+}\right)} \;
\longrightarrow \; 
\textcolor{black}{\left(W|_{[a_{i}, a_{p}]}, \; (\partial W)|_{[a_{i}, a_{p}]}\right)}.
$$
By direct application of Theorem \ref{theorem: hudsen embedding theorem}, we may assume that the restriction of $e_{i, \lambda}$ to the submanifold $\{\lambda\}\times[a_{i}, a_{p}]\times\{0\}\times\partial_{1}D^{n+1}_{+}$ is an embedding.  
We will need to deform the embeddings $e_{i, \lambda}$ so as to make these submanifolds 
$$e_{i, \lambda}(\{\lambda\}\times[a_{i}, a_{p}]\times\{0\}\times\partial_{1}D^{n+1}_{+}) \; \subset \; W|_{[a_{i}, a_{p}]}$$
for $\lambda \in \Lambda_{i}$ pairwise disjoint. 
This is done by 
applying the same argument employed in the proof of Proposition \ref{proposition: disjunction condition iii}, using Proposition \ref{proposition: inductive disjunction}. 
By combining the maps $e_{i, \lambda}$ we obtain the desired map 
$$e_{i}: \left(\Lambda_{i}\times[a_{i},
    a_{p}]\times\R^{n}\times\partial_{1}D^{n+1}_{+}, \; \;
  \Lambda_{i}\times[a_{i},
    a_{p}]\times\R^{n}\times\partial_{0,1}D^{n+1}_{+}\right)
\longrightarrow \left(W|_{[a_{i}, a_{p}]}, \;
  (\partial W)|_{[a_{i}, a_{p}]}\right).
$$ We then carry out the same construction for each $i = 0, \dots, p$
to obtain the maps $e_{0}, \dots, e_{p}$.  By employing again the
argument from Proposition \ref{proposition: disjunction condition iii}
using Proposition \ref{proposition: inductive disjunction}, we may arrange
for
$$e_{i}(\Lambda_{i}\times[a_{i}, a_{p}]\times\{0\}\times\partial_{1}D^{n+1}_{+})\bigcap e_{j}(\Lambda_{j}\times[a_{j}, a_{p}]\times\{0\}\times\partial_{1}D^{n+1}_{+}) \; = \; \emptyset \quad  \text{whenever $i \neq j$.}$$
Combining the maps $e_{0}, \dots, e_{p}$ yields the desired map 
$$e: \Lambda\times[a_{0}, a_{p}]\times\R^{n}\times\partial_{1}D^{n+1}_{+} \longrightarrow W|_{[a_{0}, a_{p}]}.$$
We may then extend this to a map
$$
e': \Lambda\times\R\times(a_{0}-\varepsilon_{0}, a_{p}+\varepsilon_{p})\times\R^{n}\times D^{n+1}_{+} \; \longrightarrow \; \R\times\R\times\R^{\infty}_{+}
$$
that satisfies the conditions of Definition \ref{defn: objects surgery data}. 
Since the ambient space if infinite dimensional, there is non obstruction to defining this extension.

Now, in addition to the map $e'$, the other piece of data required for
an element of $\widetilde{\mb{Y}}_{0}(x)$ is a $\theta$-structure
$\ell$ on $\Lambda\times K|_{(-6, -2)}$ that agrees with the
restriction of $\ell_{W}$ to the submanifold
$$
e'(\Lambda\times K_{|x|^{2}>2}|_{(-6, -2)}) \; \subset \; W|_{(-6, -2)},
$$ (see Definition \ref{defn: objects surgery data} condition (iii)).
By Proposition \ref{proposition: extension of theta-structure}, such
an extension always exists (the existence of this extension does not
depend on the choice of tangential structure $\theta$).  With this
$\theta$-structure defined over $(\Lambda\times K_{|x|^{2}>2})|_{(-6,
  -2)}$ and the map $e'$, we obtain the desired element of
$\widetilde{\mb{Y}}_{0}(x)$.  This concludes the proof of the
proposition.
\end{proof}
Propositions \ref{proposition: disjunction condition iii} and \ref{proposition: nonempty surgery data} together via Theorem \ref{theorem: flag complex equivalence}, imply that for any $p \in \Z_{\geq 0}$, the augmentation map 
$\widetilde{\mb{D}}^{\partial, \mathcal{H}}_{p, \bullet} \longrightarrow \mb{D}^{\partial, \mb{c}'}_{p}$ 
is a weak homotopy equivalence upon geometric realization.  Then,
by taking the geometric realization with respect to the
$p$-coordinate, we obtain a weak homotopy equivalence $
|\widetilde{\mb{D}}^{\partial, \mathcal{H}}_{\bullet, \bullet}|
\stackrel{\simeq} \longrightarrow |\mb{D}^{\partial,
  \mb{c}'}_{\bullet}|.$ 
  Theorem \ref{theorem: contractibility of space of surgery data} then follows from the weak homotopy
equivalence $|\widetilde{\mb{D}}^{\partial, \mathcal{H}}_{\bullet,
  \bullet}| \simeq |\mb{D}^{\partial, \mathcal{H}}_{\bullet,
  \bullet}|$ and commutativity of the diagram
$$
\xymatrix{
|\mb{D}^{\partial, \mathcal{H}}_{\bullet, \bullet}| \ar[dr] \ar[rr]^{\simeq} && |\widetilde{\mb{D}}^{\partial, \mathcal{H}}_{\bullet, \bullet}| \ar[dl]^{\simeq} \\
& |\mb{D}^{\partial, \mb{c}'}_{\bullet}|. &
}
$$
This concludes the proof of Theorem \ref{theorem: contractibility of space of surgery data}. 
With this, the proof of Theorem \ref{theorem: handlebody equivalence} is now complete.

\section{Group Completion} \label{section: group completion}
By combining the results of the previous sections we obtain the weak homotopy equivalence
\begin{equation} \label{equation: equivalence of categories final}
B\Cob^{\partial, \mathcal{H}}_{\theta} \; \simeq \; B\Cob^{\partial}_{\theta} \; \simeq \; \Omega^{\infty-1}\Sigma^{\infty}B_{+}
\end{equation}
in the case that $d = 2n+1 \geq 9$ and 
$\theta: B \longrightarrow BO(d)$ is such that $B$ is $(n-1)$-connected.  
Below we show how to derive
Theorems A* and A from this weak homotopy equivalence.
\subsection{Reduction to a monoid}
We begin with an observation.
\begin{lemma} \label{lemma: path connected objects}
The object space ${\mathrm{Ob}}\Cob^{\partial,
  \mathcal{H}}_{\theta}$ is path connected.
\end{lemma}
\begin{proof}
Recall that $\Ob\Cob^{\partial, \mathcal{H}}_{\theta} \subset \Ob\Cob^{\partial}_{\theta}$ is the subspace consisting of those pairs $(M, \ell)$ such that: 
the submanifold $M \subset \R^{\infty}_{+}$ is diffeomorphic to the disk $D^{2n}$ and $\partial M$ contains the standard $(2n-1)$-dimensional disk $D \subset \partial\R^{\infty}_{+}$. 
Furthermore, $\ell$ is required to agree with $\ell_{D}$ when restricted to $D$. 
This object space is homotopy equivalent to the homotopy quotient 
$$\xymatrix{
\BDiff_{\theta}(D^{2n}, D) = 
\textstyle{\Bun(TD^{2n}\oplus\epsilon^{1}, \theta^{*}\gamma^{2n+1}; \ell_{D})}//\Diff(D^{2n}, D)
}$$
where $\textstyle{\Bun(TD^{2n}\oplus\epsilon^{1}, \theta^{*}\gamma^{2n+1}; \ell_{D})}$ is the space of bundle maps $TD^{2n}\oplus\epsilon^{1} \longrightarrow \theta^{*}\gamma^{2n+1}$ that agree with $\ell_{D}$ when restricted to $D$, and $D$ here is considered to be a subspace of $\partial D^{2n}$.
Since $\textstyle{\Bun(TD^{2n}\oplus\epsilon^{1}, \theta^{*}\gamma^{2n+1}; \ell_{D})}$ is weakly contractible it follows that $\BDiff_{\theta}(D^{2n}, D)$ is weakly homotopy equivalent to the classifying space $\BDiff(D^{2n}, D)$, and thus is path connected. 
This completes the proof of the lemma. 
\end{proof}

We will need to work with an even ``smaller'' subcategory of $\Cob^{\partial, \mathcal{H}}_{\theta}$.
\begin{defn}
Fix once and for all an object $(P, \ell_{P}) \in \Ob\Cob^{\partial, \mathcal{H}}_{\theta}$.
The topological subcategory $\mathcal{M}_{\theta} \subset \Cob^{\partial, \mathcal{H}}_{\theta}$ is defined to be the endomorphism monoid on the object $(P, \ell_{P})$.
\end{defn}

The main theorem of this subsection is the following:
\begin{theorem} \label{theorem: reduction to a monoid}
The inclusion 
$B\mathcal{M}_{\theta} \hookrightarrow B\Cob^{\partial, \mathcal{H}}_{\theta}$ is a weak homotopy equivalence. 
\end{theorem}
The proof of the above theorem will require several steps and will be carried out over the course of the whole section.
Recall from the previous section the semi-simplicial space $\mb{D}^{\partial, \mathcal{H}}_{\bullet}$.
Let $\mb{D}^{\partial, \mathcal{M}}_{\bullet} \subset \mb{D}^{\partial, \mathcal{H}}_{\bullet}$ be the sub-semi-simplicial space whose $p$-simplices consist of those $(a, \varepsilon, (W, \ell))$ such that 
$$(W|_{a_{i}}, \ell|_{a_{i}}) \; = \; (P, \ell_{P}) \quad \text{for all $i = 0, \dots, p$,}$$
where $(P, \ell_{P}) \in \Ob\Cob^{\partial, \mathcal{H}}_{\theta, D}$ is the object used in the definition of the monoid $\mathcal{M}_{\theta}$.
By \cite[Proposition 2.14]{GRW 14}, there is a weak homotopy equivalence $B\mathcal{M}_{\theta} \simeq |\mb{D}^{\partial, \mathcal{M}}_{\bullet}|$, and thus to prove Theorem \ref{theorem: reduction to a monoid} it will suffice to prove that the inclusion $|\mb{D}^{\partial, \mathcal{M}}_{\bullet}| \hookrightarrow |\mb{D}^{\partial, \mathcal{H}}_{\bullet}|$ is a weak homotopy equivalence. 

In order to prove that this inclusion is a weak homotopy equivalence, we will need to work with slightly more rigid models for $\mb{D}^{\partial, \mathcal{H}}_{\bullet}$ and $\mb{D}^{\partial, \mathcal{M}}_{\bullet}$.
\begin{defn}
Let $\mb{D}^{\partial, \mathcal{H}, \perp}_{\bullet} \subset \mb{D}^{\partial, \mathcal{H}}_{\bullet}$ be the sub-semi-simplicial space whose $p$-simplices consist of those $(a, \varepsilon, (W, \ell))$ for which the $\theta$-manifold $(W, \ell)$ is \textit{cylindrical} over each interval $(a_{i}-\varepsilon_{i}, a_{i}+\varepsilon_{i})$. 
We then define
$\mb{D}^{\partial, \mathcal{M}, \perp}_{\bullet} \subset \mb{D}^{\partial, \mathcal{H}, \perp}_{\bullet}$
to be the intersection $\mb{D}^{\partial, \mathcal{H}, \perp}_{\bullet}\cap\mb{D}^{\partial, \mathcal{M}}_{\bullet}$. 
\end{defn}
The following proposition is proven in the same way as \cite[Theorem 3.9]{GRW 09}.
\begin{proposition}
The inclusion maps induce weak homotopy equivalences
$$|\mb{D}^{\partial, \mathcal{H}, \perp}_{\bullet}| \simeq |\mb{D}^{\partial, \mathcal{H}}_{\bullet}| \quad \text{and} \quad
|\mb{D}^{\partial, \mathcal{M}, \perp}_{\bullet}| \simeq |\mb{D}^{\partial, \mathcal{M}}_{\bullet}|.$$
\end{proposition}
By the above proposition, to prove Theorem \ref{theorem: reduction to a monoid} it will suffice to prove that the inclusion 
$$|\mb{D}^{\partial, \mathcal{M}, \perp}_{\bullet}| \hookrightarrow |\mb{D}^{\partial, \mathcal{H}, \perp}_{\bullet}|$$ 
is a weak homotopy equivalence. 
Our method of proving that this inclusion induces a weak homotopy equivalence will be formally similar to the proof of Theorem \ref{theorem: disk object equivalence}, though technically much easier. 
As in the proof of that theorem we will need to construct an augmented bi-semi-simplicial space $\mb{D}^{\partial, \mathcal{M}, \perp}_{\bullet, \bullet} \longrightarrow \mb{D}^{\partial, \mathcal{M}, \perp}_{\bullet, -1}$ with $\mb{D}^{\partial, \mathcal{M}, \perp}_{\bullet, -1} = \mb{D}^{\partial, \mathcal{H}, \perp}_{\bullet}$.
\begin{defn} \label{defn: path resolution}
For $x = (a, \varepsilon, (W, \ell)) \in \mb{D}^{\partial, \mathcal{H}, \perp}_{\bullet}$, let $\mb{Z}_{0}(x)$ denote the set of $(p+1)$-tuples $(\gamma_{0}, \dots, \gamma_{p})$, where each $\gamma_{i}$ is a path
$\gamma_{i}: [0, 1] \longrightarrow \bPsi_{\theta}^{\partial}((a_{i}-\varepsilon_{i}, a_{i}+\varepsilon_{i})\times\R^{\infty}_{+}),$
subject to the following conditions:
\begin{enumerate} \itemsep.2cm
\item[(a)] $\gamma_{i}(0) \; = \; (W|_{(a_{i}-\varepsilon_{i}, a_{i}+\varepsilon_{i})}, \; \ell|_{(a_{i}-\varepsilon_{i}, a_{i}+\varepsilon_{i})})$,
\item[(b)] $\gamma_{i}(1)|_{a_{i}} = (P, \; \ell_{P})$,
\item[(c)] for each $t \in [0, 1]$, the $\theta$-manifold $\gamma_{i}(t)$ agrees with $(W, \ell)$ outside of  $(a_{i}-\tfrac{2\varepsilon_{i}}{3}, a_{i}+\tfrac{2\varepsilon_{i}}{3})\times\R^{\infty}_{+}$.
\end{enumerate}
For $q \in \Z_{\geq 0}$, $\mb{Z}_{q}(x)$ is defined to be the $(q+1)$-fold product $[\mb{Z}_{0}(x)]^{\times(q+1)}$. 
We then define 
$$\mb{D}^{\partial, \mathcal{M}, \perp}_{p, q} \; = \; \{(x, y) \; | \; x \in \mb{D}^{\partial, \mathcal{H}, \perp}_{p} \quad \text{and} \quad y \in \mb{Z}_{q}(x) \; \},$$
topologized as a subspace of the product
$$
\mb{D}^{\partial, \mathcal{H}, \perp}_{p}\times\Maps\left([0, 1], \Ob\Cob^{\partial, \mathcal{H}}_{\theta, D}\right)^{\times(q+1)(p+1)}.
$$
This definition makes the assignment $[p, q] \mapsto \mb{D}^{\partial, \mathcal{M}, \perp}_{p, q}$ into a bi-semi-simplicial space. 
The forgetful maps
$\mb{D}^{\partial, \mathcal{M}, \perp}_{p, q} \longrightarrow \mb{D}^{\partial, \mathcal{H}, \perp}_{p},$ $(x, y) \mapsto x,$
yield the augmented bi-semi-simplicial space 
$$\mb{D}^{\partial, \mathcal{M}, \perp}_{\bullet, \bullet}  \longrightarrow \mb{D}^{\partial, \mathcal{M}, \perp}_{\bullet, -1}$$ 
with $\mb{D}^{\partial, \mathcal{M}, \perp}_{\bullet, -1} = \mb{D}^{\partial, \mathcal{H}, \perp}_{\bullet}$.
\end{defn}

The following proposition should be compared to Theorem \ref{theorem: contractibility of space of surgery data}. 
The key technical component of its proof relies on the fact that $\Ob\Cob^{\partial, \mathcal{H}}_{\theta}$ is path-connected. 
\begin{proposition} \label{proposition: resolution by path choices}
There are weak homotopy equivalences
$$
|\mb{D}^{\partial, \mathcal{M}, \perp}_{\bullet, 0}| \; \stackrel{\simeq} \longrightarrow \; |\mb{D}^{\partial, \mathcal{M}, \perp}_{\bullet, \bullet}| \; \stackrel{\simeq} \longrightarrow \; |\mb{D}^{\partial, \mathcal{M}, \perp}_{\bullet, -1}|
$$ where the first map is induced by inclusion of zero-simplices and
the second one is induced by the augmentation.
\end{proposition}
\begin{proof}
We will explicitly prove that the augmentation map $|\mb{D}^{\partial, \mathcal{M}, \perp}_{\bullet, \bullet}| \; \longrightarrow \; |\mb{D}^{\partial, \mathcal{M}, \perp}_{\bullet, -1}|$ is a weak homotopy equivalence.
With homotopy equivalence established, the proof that $|\mb{D}^{\partial, \mathcal{M}, \perp}_{\bullet, 0}| \;  \longrightarrow \; |\mb{D}^{\partial, \mathcal{M}, \perp}_{\bullet, \bullet}|$ is a weak homotopy equivalence follows by the same argument given in \cite[Page 327]{GRW 14}.
To prove the proposition it will suffice to prove that for each $p \in \Z_{\geq 0}$, the augmented topological flag complex 
$\mb{D}^{\partial, \mathcal{M}, \perp}_{p, \bullet}\longrightarrow \mb{D}^{\partial, \mathcal{M}, \perp}_{p, -1}$
satisfies conditions (i), (ii), and (iii) of Theorem \ref{theorem: flag complex equivalence}. 
Condition (i) follows from the fact that the augmentation map $\mb{D}^{\partial, \mathcal{M}, \perp}_{\bullet, 0} \longrightarrow \mb{D}^{\partial, \mathcal{M}, \perp}_{\bullet, -1}$ is a fibration. 
Indeed the augmentation map is given by sending a tuple of continuous paths to their starting points.
To establish condition (iii) we need to show that given any non-empty list $v_{1}, \dots, v_{k} \in \mb{Z}_{0}(x)$, there exists $v \in \mb{Z}_{0}(x)$ such that 
$(v, v_{i}) \in \mb{Z}_{0}(x)$ for all $i = 1, \dots, p$. 
Since $\mb{Z}_{q}(x) = (\mb{Z}_{0}(x))^{\times(q+1)}$, given such a list $v_{1}, \dots, v_{k} \in \mb{Z}_{0}(x)$, we simply set $v = v_{1}$. 
It is automatic that $(v, v_{i}) \in \mb{Z}_{0}(x)$ for all $i = 1, \dots, p$.
This establishes condition (iii). 
Condition (ii) is more involved and will be proven in the next lemma. 
\end{proof}

\begin{lemma}
Let $p \in \Z_{\geq 0}$. 
The set $\mb{Z}_{0}(x)$ is non-empty for all $x \in \mb{D}^{\partial, \mathcal{H}, \perp}_{p}$.
\end{lemma}
\begin{proof}
Let $x = (a, \varepsilon, (W, \ell)) \in \mb{D}^{\partial, \mathcal{H}, \perp}_{p}$. 
For each $i = 0, \dots, p$, we will construct a path $\gamma_{i}: [0, 1] \longrightarrow \bPsi_{\theta}^{\partial}((a_{i}-\varepsilon_{i}, a_{i}+\varepsilon_{i})\times\R^{\infty}_{+})$
satisfying conditions (a), (b), and (c) from Definition \ref{defn: path resolution}. 
Let $i = 0, \dots, p$. 
By Lemma \ref{lemma: path connected objects} there exists a path 
$$\alpha_{i}: [0, 1] \longrightarrow \Ob\Cob^{\partial, \mathcal{H}}_{\theta}$$ 
with $\alpha_{i}(0) = (W|_{a_{i}}, \ell|_{a_{i}})$ and $\alpha_{i}(1) = (P, \ell_{P})$. 
By Proposition \ref{proposition: salient features of smooth maps} we may assume that this path is smooth (see Definition \ref{defn: smooth maps}). 
Fix a smooth function $\rho: \R \longrightarrow [0, 1]$ with $\rho(t) = 1$ for $|t| \leq \tfrac{1}{3}$ and $\rho(t) = 0$ for $|t| \geq \tfrac{2}{3}$.
Let 
$$
\bar{\alpha}^{t}_{i}: (a_{i}-\varepsilon_{i}, a_{i}+\varepsilon_{i}) \longrightarrow \Ob\Cob^{\partial, \mathcal{H}}_{\theta}, \quad t \in [0, 1],
$$
be the one-parameter family of paths defined by the formula
$$\bar{\alpha}^{t}_{i}(s) = \alpha_{i}(t\cdot\rho(\tfrac{s-a_{i}}{\varepsilon_{i}})).$$
For $t \in [0, 1]$, consider the graph 
$$\Gamma(\bar{\alpha}^{t}_{i}) \; = \; \{\; (s, x) \in (a_{i}-\varepsilon_{i}, a_{i}+\varepsilon_{i})\times\R^{\infty}_{+} \; | \; x \in \bar{\alpha}^{t}_{i}(s)\; \}.$$
Since the path $\bar{\alpha}^{t}_{i}$ is smooth, it follows that $\Gamma(\bar{\alpha}^{t}_{i})$ is a $(2n+1)$-dimensional submanifold and that the projection $\Gamma(\bar{\alpha}^{t}_{i}) \longrightarrow (a_{i}-\varepsilon_{i}, a_{i}+\varepsilon_{i})$ is a submersion (see Subsection \ref{subsection: smooth families of manifolds}). 
Notice that by how the function $\rho$ was defined, the graph 
$$\Gamma(\bar{\alpha}^{t}_{i}) \subset (a_{i}-\varepsilon_{i}, a_{i}+\varepsilon_{i})\times\R^{\infty}_{+}$$ 
is cylindrical outside of the set $(a_{i}+\tfrac{2\varepsilon_{i}}{3}, a_{i}+\varepsilon_{i})\times\R^{\infty}_{+}$.
The vertical tangent bundle of this submersion comes equipped with a $\theta$-structure 
$$T^{v}\Gamma(\bar{\alpha}^{t}_{i})\oplus\epsilon^{1} \longrightarrow \theta^{*}\gamma^{2n+1}.$$
By precomposing this $\theta$-structure with a choice of bundle isomorphism 
\begin{equation} \label{equation: bundle isomorphism}
T\Gamma(\bar{\alpha}^{t}_{i}) \cong T^{v}\Gamma(\bar{\alpha}^{t}_{i})\oplus\epsilon^{1},
\end{equation}
we obtain a $\theta$-structure 
$$
\phi^{t}_{i}: T\Gamma(\bar{\alpha}^{t}_{i}) \longrightarrow \theta^{*}\gamma^{2n+1}.
$$
For each $t$, 
the pair 
$(\Gamma(\bar{\alpha}^{t}_{i}), \phi^{t}_{i})$ is an element of $\bPsi_{\theta}((a_{i}-\varepsilon_{i}, a_{i}+\varepsilon_{i})\times\R^{\infty}_{+})$. 
By choosing the bundle isomorphism (\ref{equation: bundle isomorphism}) appropriately, this pair is cylindrical (as a $\theta$-manifold) outside of the set $(a_{i}-\tfrac{2\varepsilon_{i}}{3}, a_{i}+\tfrac{2\varepsilon_{i}}{3})\times\R^{\infty}_{+}$. 
Let us now observe that by how the function $\rho$ was defined, $(\Gamma(\bar{\alpha}^{t}_{i}), \phi^{t}_{i})$ agrees with $(W, \ell)$ outside of $(a_{i}-\tfrac{2\varepsilon_{i}}{3}, a_{i}+\tfrac{2\varepsilon_{i}}{3})\times\R^{\infty}_{+}$ for all $t \in [0,1]$. 
We then set 
$$\gamma_{i}(t) \; = \; (\Gamma(\bar{\alpha}^{t}_{i}), \phi^{t}_{i}), \quad t \in [0,1].$$
By repeating the same construction for each $i = 0, \dots, p$, we obtain an element $(\gamma_{0}, \dots, \gamma_{p}) \in \mb{Z}_{0}(x)$. 
This concludes the proof of the proposition and furthermore concludes the proof of Proposition \ref{proposition: resolution by path choices}. 
\end{proof}
We now show how to use Proposition \ref{proposition: resolution by path choices} to finish the proof of Theorem \ref{theorem: reduction to a monoid}. 
Let $(a, \varepsilon, (W, \ell), \gamma) \in \mb{D}^{\partial, \mathcal{M}, \perp}_{\bullet, 0}$ be with $\gamma = (\gamma_{0}, \dots, \gamma_{p})$.  
For each $i \in 0, \dots, p$,
we define a one-parameter family 
$$
\mathcal{K}^{\gamma_{i}}_{t}(W, \ell) \; \in \; \bpsi^{\partial}_{\theta}(\infty, 1), \quad t \in [0, 1],
$$
by setting $\mathcal{K}^{\gamma_{i}}_{t}(W, \ell)$ equal to $(W, \ell)$ on the complement of $(a_{i}-\varepsilon_{i}, a_{i}+\varepsilon_{i})\times\R^{\infty}_{+}$ for all $t$. 
On the subspace $(a_{i}-\varepsilon_{i}, a_{i}+\varepsilon_{i})\times\R^{\infty}_{+}$, the family $\mathcal{K}^{\gamma_{i}}_{t}(W, \ell)$ is defined to be equal to $\gamma_{i}(t)$. 
By Definition \ref{defn: path resolution}, it follows that $\gamma_{i}(t)$ does agree with $(W, \ell)$ near the boundary of $(a_{i} - \varepsilon_{i}, a_{i} +\varepsilon_{i})$ and thus 
$\mathcal{K}^{\gamma_{i}}_{t}(W, \ell) \subset \R\times\R^{\infty}_{+}$ does yield a smooth family of elements of $\bpsi^{\partial}_{\theta}(\infty, 1)$.
By iterating the above definition we define the $(p+1)$-parameter family 
\begin{equation} \label{equation: iterated homotopy K}
\mathcal{K}^{\gamma}_{\bar{t}}(W, \ell) := \mathcal{K}^{\gamma_{p}}_{t_{p}}\circ \cdots \circ \mathcal{K}^{\gamma_{0}}_{t_{0}}(W, \ell), \quad \bar{t} = (t_{0}, \dots, t_{p}) \in [0,1]^{p+1}.
\end{equation}
From the above construction it can be verified directly that 
$$
(a, \; \tfrac{\varepsilon}{3}, \; \mathcal{K}^{\gamma}_{\bar{t}}(W, \ell)), \quad \bar{t} \in [0,1]^{\times(p+1)},
$$
defines a $(p+1)$-parameter family of elements of $\mb{D}^{\partial, \mathcal{H}, \perp}_{p}$. 
Furthermore it follows that 
$$(a, \; \tfrac{\varepsilon}{3}, \; \mathcal{K}^{\gamma}_{\bar{t}}(W, \ell)) \in \mb{D}^{\partial, \mathcal{M}, \perp}_{p} \quad \text{whenever $\bar{t} = (1, \dots, 1)$.}$$ 
For each $p \in \Z_{\geq 0}$ we define the map 
$$
\mathcal{F}_{p}: [0,1]^{p+1}\times\mb{D}^{\partial, \mathcal{M}, \perp}_{p, 0} \; \longrightarrow \; \mb{D}^{\partial, \mathcal{H}, \perp}_{p}, \quad (\bar{t}, \; (a, \varepsilon, (W, \ell), \gamma)) \; \mapsto \; (a, \tfrac{\varepsilon}{3}, \mathcal{K}^{\gamma}_{\bar{t}}(W, \ell)).
$$
Using the same construction from Proposition \ref{proposition: assembled homotopy} (\cite[Page 312-313]{GRW 14}) the maps $\mathcal{F}_{p}$ assemble to a homotopy 
\begin{equation}
\mathcal{F}: [0,1]\times|\mb{D}^{\partial, \mathcal{M}, \perp}_{\bullet, 0}| \; \longrightarrow \; |\mb{D}^{\partial, \mathcal{H}, \perp}_{\bullet}|
\end{equation}
that satisfies:
\begin{enumerate} \itemsep.2cm
\item[(i)] The map $\mathcal{F}(0, \underline{\hspace{.3cm}})$ agrees with the composition
$$\xymatrix{
|\mb{D}^{\partial, \mathcal{M}, \perp}_{\bullet, 0}| \ar[r]^{\simeq} & |\mb{D}^{\partial, \mathcal{M}, \perp}_{\bullet, \bullet}| \ar[r]^{\simeq \ \ \ \ \ \ \ \ \ } & |\mb{D}^{\partial, \mathcal{M}, \perp}_{\bullet, -1}| \; = \; |\mb{D}^{\partial, \mathcal{H}, \perp}_{\bullet}|,
}
$$
which is a weak homotopy equivalence by Proposition \ref{proposition: resolution by path choices}. 
\item[(ii)] The map $\mathcal{F}(1, \underline{\hspace{.3cm}})$ factors through the inclusion 
$
|\mb{D}^{\partial, \mathcal{M}, \perp}_{\bullet}| \; \hookrightarrow \; |\mb{D}^{\partial, \mathcal{H}, \perp}_{\bullet}|.
$
\end{enumerate}

With the above construction in place we are now ready to finish the proof of Theorem \ref{theorem: reduction to a monoid}. 
\begin{proof}[Proof of Theorem \ref{theorem: reduction to a monoid}]
The homotopy $\mathcal{F}$ constructed above yields the homotopy commutative diagram 
$$
\xymatrix{
|\mb{D}^{\partial, \mathcal{M}, \perp}_{\bullet, 0}| \ar[d]^{\mathcal{F}(0, \underline{\hspace{.3cm}})}_{\simeq} \ar[drr]^{\mathcal{F}(1, \underline{\hspace{.3cm}})} && \\
|\mb{D}^{\partial, \mathcal{H}, \perp}_{\bullet}|  && |\mb{D}^{\partial, \mathcal{M}, \perp}_{\bullet}|. \ar@{_{(}->}[ll]
}
$$
Since the left-vertical map is a weak homotopy equivalence and the diagram is homotopy commutative, it follows that the diagonal map induces an injection on homotopy groups. 
Now, this diagonal map $\mathcal{F}(1, \underline{\hspace{.3cm}})$ has a section, namely the embedding $|\mb{D}^{\partial, \mathcal{M}, \perp}_{\bullet}| \hookrightarrow |\mb{D}^{\partial, \mathcal{H}, \perp}_{\bullet, 0}|$ induced by sending a $p$-simplex $(a, \varepsilon, (W, \ell))$ to the tuple $(a, \varepsilon, (W, \ell), \gamma^{c})$, where each $\gamma^{c}_{i}$ is the constant path defined by $\gamma^{c}_{i}(t) = (P, \ell_{P})$.
The existence of this section implies that $\mathcal{F}(1, \underline{\hspace{.3cm}})$ induces a surjection on homotopy groups, and thus is a weak homotopy equivalence. 
Homotopy commutativity of the diagram then implies that the inclusion 
$
|\mb{D}^{\partial, \mathcal{M}, \perp}_{\bullet}| \hookrightarrow |\mb{D}^{\partial, \mathcal{H}, \perp}_{\bullet}|
$
is a weak homotopy equivalence. 
Combining this with the weak homotopy equivalences $|\mb{D}^{\partial, \mathcal{M}, \perp}_{\bullet}| \simeq B\mathcal{M}_{\theta}$ and $|\mb{D}^{\partial, \mathcal{H}, \perp}_{\bullet}| \simeq B\Cob^{\partial, \mathcal{H}}_{\theta}$ concludes the proof of Theorem \ref{theorem: reduction to a monoid}. 
\end{proof}

Combining Theorem \ref{theorem: reduction to a monoid} with all previous results proven about the cobordism category $\Cob^{\partial}_{\theta}$ yields the following corollary. 
\begin{corollary} \label{corollary: equivalence to infinite loopspace}
Let $n \geq 4$ and let $\theta: B \longrightarrow BO(2n+1)$ be such that $B$ is $(n-1)$-connected. 
Then there is a weak homotopy equivalence $B\mathcal{M}_{\theta} \simeq \Omega^{\infty-1}\Sigma^{\infty}B_{+}$. 
\end{corollary}

\subsection{Handlebodies and the structure of $\mathcal{M}_{\theta}$} \label{subsection: handlebodies and structure of the monoid}
The results of the previous section identify the homotopy type of the
classifying space $B\mathcal{M}_{\theta}$.  We now aim to give a more
detailed description of the homotopy type of the monoid
$\mathcal{M}_{\theta}$ itself.  We must first recall some basic
results from \cite{W 67} on high dimensional handlebodies.
\begin{defn} \label{defn: handlebody}
For $m, k \in \N$, let $\mathcal{H}(m, k)$ denote the set of all  $m$-dimensional manifolds that can be constructed by attaching index-$k$ handles to the $m$-dimensional disk. 
More precisely, $M$ is an element of $\mathcal{H}(m, k)$ if there exists a smooth embedding 
$
\displaystyle
\varphi: \coprod_{i} \partial D_{i}^{k}\times D_{i}^{m-k} \longrightarrow \partial D^{m}
$
such that there is a diffeomorphism
$$
M \; \cong \; D^{m}\cup_{\varphi}\bigg(\coprod_{i}D_{i}^{k}\times D_{i}^{m-k}\bigg).
$$
The elements of $\mathcal{H}(m, k)$ are called \textit{index $k$
  handlebodies of dimension $m$}.
\end{defn}
Our main applications concern the set of handlebodies
$\mathcal{H}(2n+1, n)$, that is the index-$n$ handlebodies of
dimension $2n+1$.  The following proposition gives us a simple means
for recognizing elements of the set $\mathcal{H}(2n+1, n)$.
\begin{proposition} \label{proposition: recognition of handlebodies}
Let $n \geq 3$ and let $M$ be a compact manifold of dimension $2n+1$. 
Suppose that the boundary $\partial M$ is $(n-1)$-connected and that the pair $(M, \partial M)$ is $n$-connected. 
Then the manifold $M$ is diffeomorphic to an element of $\mathcal{H}(2n+1, n)$.
\end{proposition}
\begin{proof}
Let $M'$ denote the manifold obtained from $M$ by deleting an open
disk from the interior of $M$.  Using \textit{Lefschetz-Duality}
together with the homology long exact sequence associated to the pair
$(M', \partial M)$, it follows that $H_{i}(M', \partial M) = 0$
whenever $i \neq n+1$ and that $H_{n+1}(M', \partial M)$ is
free-abelian.  The boundary of $M'$ has the decomposition $\partial M'
= \partial_{0}M'\sqcup\partial_{1}M'$, where $\partial_{1}M' =
\partial M$ and $\partial_{0}M' \cong S^{2n}$ corresponds to the
boundary of the deleted disk.  Applying \textit{Lefschetz-Duality} to
the pairs $(M', \partial_{0}M')$ and $(M', \partial_{1}M')$ it follows
that $H_{i}(M', \partial_{0}M') = 0$ when $i \neq n$ and that
$H_{n}(M', \partial_{0}M')$ is free-abelian.  Now, choose a Morse
function $f: M' \longrightarrow [0, 1]$ with regular levels $f^{-1}(i)
= \partial_{i}M'$ for $i = 0, 1$.  Since $H_{i}(M', \partial_{0}M') =
0$ whenever $i \neq n$, we may apply the critical point cancelation
techniques of \cite{M 65} so as to arrange for $f$ to only have
critical points of index $n$ (in order to apply the cancelation
theorems from \cite{M 65} it is necessary that $n \geq 3$ as in the
hypothesis).  With such a Morse function with only index $n$-critical
points obtained, the desired handlebody structure on $M$ is given by
the \textit{stable-trajectories} of the Morse function.  This
concludes the proof of Proposition \ref{proposition: recognition of handlebodies}.
\end{proof}
\begin{corollary} \label{corollary: handlebody structure on W}
Let $(t, (W, \ell)) \in \mathcal{M}_{\theta}$.  Then the underlying
manifold $W|_{[0, t]}$ is an element of $\mathcal{H}(2n+1, n)$.
Conversely, let $(t, (W, \ell)) \in \Mor\Cob_{\theta}^{\partial}$
be an endomorphism of the object $(P, \ell_{P})$  
and suppose that $W|_{[0, t]} \in
\mathcal{H}(2n+1, n)$.  Then the morphism $(t, (W, \ell))$ is an
element of $\mathcal{M}_{\theta}$.
\end{corollary}
\begin{proof}
Suppose that $(t, (W, \ell)) \in \Mor\Cob^{\partial, \mathcal{H}}_{\theta}$.  
By definition we have that $(W|_{[0, t]}; (\partial W)|_{[0, t]}, W|_{t})$ is $n$-connected.
The manifolds $W|_{t}$ and $W|_{0}$ are are both diffeomorphic to the disk $D^{2n}$ and thus it follows that the pair
  $$
 \left(W|_{[0, t]}, \; W|_{0}\cup(\partial W)|_{[0, t]}\cup W|_{t}\right) \;  = \; \left(W|_{[0, t]}, \; \partial(W|_{[0, t]})\right)
  $$ 
  is $n$-connected.  
  Similarly, by definition the pair $((\partial
  W)|_{[0, t]}, \; (\partial W)|_{t})$ is $(n-1)$-connected but since
  $(\partial W)|_{t}$ and $(\partial W)|_{0}$ are diffeomorphic to the
  sphere $S^{2n}$, it follows that the manifold
$$W|_{0}\cup(\partial W)|_{[0, t]}\cup W|_{t} = \partial(W|_{[0, t]})$$ 
is $(n-1)$-connected. 
Combining these observations it follows from Proposition \ref{proposition: recognition of handlebodies} that $W|_{[0, t]}$ is diffeomorphic to an index-$n$ handlebody of dimension $2n+1$.
The other statement in the corollary is proven similarly. 
This concludes the proof.
\end{proof}
We now use Corollary \ref{corollary: handlebody structure on W} to
give a description of the homotopy type of the topological monoid
$\mathcal{M}_{\theta}$.  
Let us fix some notation.  
For $t \in \R$, denote by $P_{t}$ the submanifold $\{t\}\times P \subset \{t\}\times \R^{\infty}_{+}$ and denote by
$D_{t}$ the disk $\{t\}\times D \subset
\{t\}\times\partial\R^{\infty}_{+}$ used in the definition of
$\Cob^{\partial, \mathcal{H}}_{\theta}$.  
Let $L_{t}$ denote the
product $[0, t]\times D \subset [0, t]\times\R^{\infty}_{+}$.  We have
$L_{t}\cap(P_{0}\sqcup P_{t}) = D_{0}\sqcup D_{t}$.  Consider the
union $P_{0}\cup L_{t}\cup P_{t}$.  Upon smoothing corners, it follows
that $P_{0}\cup L_{t}\cup P_{t}$ is diffeomorphic to the
$2n$-dimensional disk $D^{2n}$.  
Now, let $(t, (W, \ell)) \in
\mathcal{M}_{\theta}$.  The boundary of the underlying manifold
$W|_{[0, t]}$ has the decomposition $\partial(W|_{[0, t]}) =
P_{0}\cup(\partial W)|_{[0, t]}\cup P_{1}$.  The boundary
$\partial(W|_{[0, t]})$ then contains the union $P_{0}\cup L_{t}\cup
P_{t}$.  
From this observation we obtain the following proposition. 
\begin{proposition} \label{proposition: structure of the monoid}
There is a weak homotopy equivalence 
$$
\mathcal{M}_{\theta} \; \simeq \; \coprod_{V}\textstyle{\BDiff_{\theta}}(V, D^{2n})
$$
where the disjoint union varies over diffeomorphism classes of elements $V \in \mathcal{H}(2n+1, n)$. 
Furthermore, the monoid $\mathcal{M}_{\theta}$ is homotopy commutative. 
\end{proposition}
\begin{proof}
By Remark \ref{remark: topological description of morphism spaces} it follows that we have the weak homotopy equivalence 
$$
\mathcal{M}_{\theta} \simeq \coprod_{W}\textstyle{\BDiff_{\theta}}(W, P_{0}\cup L_{t}\cup P_{1})
$$ where $W$ ranges over all relative cobordisms $(W, \partial _{1})$
from $(P_{0}, \partial P_{0})$ to $(P_{1}, \partial P_{1})$, equipped
with an embedding $[0,1]\times D \hookrightarrow \partial_{1} W$, such
that $(W; \partial_{1}W, P_{1})$ and $(\partial_{1}W, \partial P_{1})$
are $n$ and $(n-1)$-connected respectively.  With this connectivity
condition, it follows from Corollary \ref{corollary: handlebody
  structure on W} that $W \in \mathcal{H}(2n+1, n)$ after smoothing
corners.  By choosing a diffeomorphism $P_{0}\cup L_{t}\cup P_{1}
\cong D^{2n}$, we obtain the weak homotopy equivalence asserted in the
statement of the proposition.

The second statement of the proposition is obtained by the same
approach taken in \cite{GRW 09} (see specifically the argument in
\cite[Page 46]{GRW 09}).  Indeed, via this weak homotopy equivalence
$\mathcal{M}_{\theta} \; \simeq \;
\coprod_{V}\textstyle{\BDiff_{\theta}}(V, D^{2n})$, the monoid
structure on the right-hand side can be interpreted as being induced
by boundary-connected sum, and this operation is homotopy commutative.
This concludes the proof of the proposition.
\end{proof}

\subsection{Group completion}
We will need to apply the group completion theorem from \cite{MS 75} to the topological monoid $\mathcal{M}_{\theta}$. 
This will allow us to relate the homology of $\mathcal{M}_{\theta}$ to that of the infinite loopspace $\Omega^{\infty}\Sigma^{\infty}B_{+}$. 
Let us recall this group completion result. 
Let $\mathcal{M}$ be any topological monoid that is homotopy commutative. 
Consider the graded homology group $H_{*}(\mathcal{M})$ equipped with ring structure given by the \textit{Pontryagin product} (the product induced by the multiplication in $\mathcal{M}$). 
The monoid of path-components $\pi_{0}(\mathcal{M}) \subset H_{*}(\mathcal{M})$ is a multiplicative subset.
There is a natural map 
$
T: \mathcal{M} \longrightarrow \Omega B\mathcal{M}
$
whose induced homomorphism on homology sends $\pi_{0}\mathcal{M}$ to the group of units in the ring $H_{*}(\Omega B\mathcal{M})$. 
Thus the map $T$ induces a homomorphism
\begin{equation} \label{equation: pi-0 localization map}
H_{*}(\mathcal{M})[\pi_{0}\mathcal{M}^{-1}] \; \longrightarrow H_{*}(\Omega B\mathcal{M}).
\end{equation}
In \cite{MS 75} it is proven that this map (\ref{equation: pi-0 localization map}) is in fact an isomorphism. 
Now, in the case that the monoid $\pi_{0}\mathcal{M}$ is finitely generated (or more generally if the localization $H_{*}(\mathcal{M})[\pi_{0}\mathcal{M}^{-1}]$ can be constructed by inverting a finite subset of elements), the ring $H_{*}(\mathcal{M})[\pi_{0}\mathcal{M}^{-1}]$ can be realized as the homology of a certain space $\mathcal{M}_{\infty}$ defined below. 
\begin{defn} \label{defn: M-infty}
Let $\mathcal{M}$ be a homotopy commutative topological monoid. 
Suppose that the monoid of path-components $\pi_{0}\mathcal{M}$ is finitely generated and let
$m_{1}, \dots, m_{k} \in \mathcal{M}$ be a list of elements such that $\{[m_{1}], \dots, [m_{k}]\} \subset \pi_{0}\mathcal{M}$ is a generating set. 
Let $m$ denote the product $m_{1}\cdots m_{k}$. 
We define $\mathcal{M}_{\infty}$ to be the homotopy colimit taken over the direct system,
$
\xymatrix{
\mathcal{M} \ar[r]^{\cdot m} & \mathcal{M} \ar[r]^{\cdot m} & \mathcal{M} \ar[r]^{\cdot m} & \cdots.
}
$
Note that homotopy commutativity implies that the homotopy type of $\mathcal{M}_{\infty}$ does not depend on the choice of generators $m_{1}, \dots, m_{k}$. 
\end{defn}
It is easy to verify the isomorphism
$H_{*}(\mathcal{M})[\pi_{0}\mathcal{M}^{-1}] \cong H_{*}(\mathcal{M}_{\infty}).$
Furthermore the natural map $T: \mathcal{M} \longrightarrow \Omega B\mathcal{M}$ induces a map 
\begin{equation} \label{equation: induced map to loopspace}
\mathcal{M}_{\infty} \longrightarrow  \Omega B\mathcal{M}.
\end{equation}
The main theorem from \cite{MS 75} is the following:
\begin{theorem} \label{theorem: group completion theorem}
Let $\mathcal{M}$ be a homotopy commutative topological monoid such that the monoid of path components $\pi_{0}(\mathcal{M})$ is homotopy commutative. 
Then the map $\mathcal{M}_{\infty} \longrightarrow  \Omega B\mathcal{M}$ from (\ref{equation: induced map to loopspace}) induces an isomorphism
$
H_{*}(\mathcal{M}_{\infty}) \; \cong \; H_{*}(\Omega B\mathcal{M}).
$
\end{theorem}
We wish to apply this theorem to our monoid $\mathcal{M}_{\theta}$. 
Corollary \ref{corollary: equivalence to infinite loopspace} implies that there is a weak homotopy equivalence $\Omega B\mathcal{M}_{\theta} \simeq \Omega^{\infty}\Sigma^{\infty}B_{+}$ and Theorem \ref{theorem: group completion theorem} implies that there is an isomorphism 
\begin{equation}
H_{*}((\mathcal{M}_{\theta})_{\infty}) \; \cong \; H_{*}(\Omega^{\infty}\Sigma^{\infty}B_{+}).
\end{equation}
We will prove
\textcolor{black}{Theorem A} by identifying the space
$(\mathcal{M}_{\theta})_{\infty}$ with
$\displaystyle{\colim_{g\to\infty}}\textstyle{\BDiff_{\theta}}(V^{2n+1}_{g}, D^{2n})$ when
the tangential structure $\theta$ is appropriately chosen.

\subsection{Proof of Theorem A*} \label{subsection: proof of theorem A}
In this subsection we restrict our attention to the case where the tangential structure $\theta: B \longrightarrow BO(2n+1)$ is chosen such that the space $B$ is $n$-connected.  
This is precisely the condition imposed on $\theta$ in the statement of Theorem A*.

\begin{proposition} \label{proposition: structure of monoid n connected theta}
Suppose that $\theta: B \longrightarrow BO(2n+1)$ is such that $B$ is $n$-connected. 
Then there is a homotopy equivalence, 
$\mathcal{M}_{\theta} \; \simeq \; \coprod_{g=0}^{\infty}\textstyle{\BDiff_{\theta}}(V^{2n+1}_{g}, D^{2n}).$
\end{proposition}
\begin{proof}
By Proposition \ref{proposition: structure of the monoid} we have the weak homotopy equivalence 
$$
\mathcal{M}_{\theta} \; \simeq \; \coprod_{V}\textstyle{\BDiff_{\theta}}(V, D^{2n})
$$
with disjoint union ranging over the diffeomorphism classes of elements of $\mathcal{H}(n, 2n+1)$. 
To prove the proposition we will show that for $V \in \mathcal{H}(n, 2n+1)$, the space $\textstyle{\BDiff_{\theta}}(V, D^{2n})$ is non-empty if and only if $V$ is diffeomorphic to $V^{2n+1}_{g}$ for some $g \in \Z_{\geq 0}$. 
To do this we will apply \ Wall' s \ 
  classification \ theorem \ for \ high-dimensional handlebodies, \ 
  \textcolor{black}{see} \cite[Theorem 2]{W 67}.  For $2n+1 \geq 4$,
  let $W \in \mathcal{H}(n, 2n+1)$.  
It follows from this theorem that the
  diffeomorphism type of $W$ is determined completely by two invariants:  
  \begin{enumerate} \itemsep.2cm
  \item[(a)] the $\Z$-rank of the free-abelian group $\pi_{n}(W)$;
  \item[(b)] a homomorphism
$\alpha_{n}: \pi_{n}(W) \longrightarrow \pi_{n-1}(SO_{n+1})$
defined by sending an element $x \in \pi_{n}(W)$ to the class in $ \pi_{n-1}(SO_{n+1}) \cong \pi_{n-1}(SO)$ that classifies the normal bundle of an embedding $S^{n} \hookrightarrow W$ that represents $x$. 
\end{enumerate}
Suppose now that $W$ admits a $\theta$-structure. 
  Since $B$ is $n$-connected it follows that $W$ is parallelizable over its $n$-skeleton (and hence it actually is parallelizable since $W$ deformation retracts to its $n$-skeleton).
It follows that the homomorphism $\alpha_{n}$ must be identically zero. 
By \cite[Theorem 2]{W 67} the diffeomorphism class of $W$ then is determined entirely by the $\Z$-rank of $\pi_{n}(W)$. 
For each $g \in \Z_{\geq 0}$ we have $\pi_{n}(V^{2n+1}_{g}) \cong \Z^{\oplus g}$, and thus it follows that the manifolds $V^{2n+1}_{g}$ exhaust all diffeomorphism classes of elements in $\mathcal{H}(n, 2n+1)$ for which the invariant $\alpha_{n}$ vanishes. 
It follows that $W \cong V^{2n+1}_{g}$ where $g$ is the $\Z$-rank of $\pi_{n}(W)$.
This concludes the proof of the corollary.
\end{proof}

The next proposition will also be important for describing the structure of the monoid of path components $\pi_{0}\mathcal{M}_{\theta}$ and its group completion. 
\begin{proposition} \label{proposition: path connectivity of theta structures}
Let $\theta: B \longrightarrow BO(2n+1)$ be such that $B$ is $n$-connected. 
Then for all $g \in \Z_{\geq 0}$ the space $\BDiff_{\theta}(V^{2n+1}_{g}, D^{2n})$ is path connected. 
\end{proposition}
\begin{proof}
To verify this claim, it will suffice to show that for all $g \in
\Z_{\geq 0}$ the action of $\Diff(V^{2n+1}_{g}, D^{2n})$ on
$\pi_{0}\Bun(TV^{2n+1}_{g}, \theta^{*}\gamma^{2n+1}; \ell_{D})$ is
transitive.  Since the space $B$ is $n$-connected and the
$V^{2n+1}_{g}$ is homotopy equivalent to an $n$-dimensional
$CW$-complex, namely $(S^{n})^{\vee g}$, it follows any
$\theta$-structure on $V^{2n+1}_{g}$ is induced from a framing of the
tangent bundle $TV^{2n+1}_{g}$.  To show that $\Diff(V^{2n+1}_{g},
D^{2n})$ acts transitively on $\pi_{0}\Bun(TV^{2n+1}_{g},
\theta^{*}\gamma^{2n+1}; \ell_{D})$ it will then suffice to show that
$\Diff(V^{2n+1}_{g}, D^{2n})$ transitively on the set of homotopy
classes of framings of $TV^{2n+1}_{g}$.  In other words, given any two
framings $\phi_{0}, \phi_{1}: TV^{2n+1}_{g} \stackrel{\cong}
\longrightarrow \epsilon^{2n+1}$, there exists $f \in
\Diff(V^{2n+1}_{g}, D^{2n})$ such that $\phi_{0}\circ Df$ is homotopic
(though bundle isomorphisms) to $\phi_{1}$.  This fact follows from
Wall's description of the mapping class groups of the manifolds
$V^{2n+1}_{g}$ in \cite{W 63}. 
Indeed, the
  invariants $\beta_i$ from \cite{W 63} determine a framing of 
  the tangent bundle $TV^{2n+1}_{g}$, and \cite[Lemma 13]{W 63} in
  fact implies that there exists a diffeomorphism $f$ homotopic to the
  identity such that $\phi_{0}\circ Df$ and $\phi_{1}$ are indeed 
  homotopic though bundle isomorphisms. 
  This completes the proof of
the proposition.
\end{proof}
The previous two propositions let us identify the structure of the
monoid of path components $\pi_{0}\mathcal{M}_{\theta}$.  We have the
following corollary.
\begin{corollary} \label{corollary: identification of monoid}
Let $\theta: B \longrightarrow BO(2n+1)$ be such that $B$ is $n$-connected.
Then there is an isomorphism of monoids $\pi_{0}\mathcal{M}_{\theta} \cong \Z_{\geq 0}$,
where $\Z_{\geq 0}$ is considered a monoid under addition. 
\end{corollary}
\begin{proof}
There is map $ \pi_{0}\mathcal{M}_{\theta} \longrightarrow \Z_{\geq 0}
$ defined by sending $[(W, \ell)] \in \pi_{0}\mathcal{M}_{\theta}$ to
$\dim_{\Z}(H_{n}(W; \Z))$.  It is easily checked that this is a
surjective and additive map.  To see that it is an isomorphism of
monoids, observe that Propositions \ref{proposition: structure of
  monoid n connected theta} and \ref{proposition: path connectivity of
  theta structures} imply that the path component of any element $(W,
\ell) \in \mathcal{M}_{\theta}$ is determined entirely by the rank
$\dim_{\Z}(H_{n}(W; \Z))$.  This concludes the proof of the corollary.
\end{proof}
With the above results established we are now in a position to prove
Theorem A*.  By the above corollary it follows that the monoid
$\pi_{0}\mathcal{M}_{\theta}$ is generated by a singe element, namely
the path component corresponding to $\BDiff_{\theta}(V^{2n+1}_{1},
D^{2n})$ under the weak homotopy equivalence of Proposition
\ref{proposition: structure of monoid n connected theta}.  Let $x \in
\mathcal{M}_{\theta}$ represent such a generator and let
$\displaystyle{\colim_{g\to\infty}}\textstyle{\BDiff_{\theta}}(V^{2n+1}_{g},
D^{2n})$ denote the direct limit taken over the direct system
$$
\xymatrix{
\BDiff_{\theta}(V^{2n+1}_{0}, D^{2n}) \ar[r] & \BDiff_{\theta}(V^{2n+1}_{1}, D^{2n}) \ar[r] & \cdots \ar[r] & \BDiff_{\theta}(V^{2n+1}_{g}, D^{2n}) \ar[r] & \cdots 
}
$$
with each map induced by multiplication by the generator $x$. 
The homotopy type of this direct limit does not depend on choice of generator $x$.
By Corollary \ref{corollary: identification of monoid}, it follows that there is a homotopy equivalence 
\begin{equation} \label{equation: group completion limit}
(\mathcal{M}_{\theta})_{\infty} \; \simeq \; \left(\displaystyle{\colim_{g\to\infty}}\textstyle{\BDiff_{\theta}}(V^{2n+1}_{g}, D^{2n})\right)\times\Z,
\end{equation}
where $(\mathcal{M}_{\theta})_{\infty}$ is the direct limit from Definition \ref{defn: M-infty} applied to the monoid $\mathcal{M}_{\theta}$.
Restricting the homology equivalence $(\mathcal{M}_{\theta})_{\infty}  \simeq_{H_{*}} \Omega B\mathcal{M}_{\theta}$ to a path component and combining with (\ref{equation: group completion limit}) yields the homology isomorphism
$$
\displaystyle{\colim_{g\to\infty}}H_{*}(\textstyle{\BDiff_{\theta}}(V^{2n+1}_{g}, D^{2n}); \Z)  \; \cong \; H_{*}(\Omega_{0}B\mathcal{M}_{\theta}; \Z).
$$
Combining this with the weak homotopy equivalence $\Omega_{0}B\mathcal{M}_{\theta} \simeq Q_{0}B_{+}$ (which holds in the case that $n \geq 4$) yields Theorem A*.

By restricting to a particular choice of tangential structure $\theta$ we may obtain the proof of Theorem A stated in the introduction. 
Let $\theta^{n}: BO(2n+1)\langle n \rangle \longrightarrow BO(2n+1)$ denote the $n$-connected cover. 
Theorem A will follow from the proposition below. 
\begin{proposition} \label{proposition: theta-n structures}
For all $g \in \Z_{\geq g}$, the forgetful map 
$$
\textstyle{\BDiff_{\theta^{n}}}(V^{2n+1}_{g}, D^{2n}) \; \longrightarrow \; \BDiff(V^{2n+1}_{g}, D^{2n})
$$
is a weak homotopy equivalence. 
\end{proposition}
\begin{proof}
To prove the proposition it will suffice to prove that the space $\Bun_{\theta^{n}}(TV^{2n+1}_{g}, (\theta^{n})^{*}\gamma^{2n+1}; \ell_{D})$ is weakly contractible. 
This space can be identified with the space of lifts
$$
\xymatrix{
D^{2n} \ar[rr] \ar@{^{(}->}[d]  && BO(2n+1)\langle n\rangle \ar[d]^{\theta^{n}} \\
V^{2n+1}_{g} \ar[rr] \ar@{-->}[urr] && BO(2n+1)
}
$$
where the bottom-horizontal map if the map classifying the tangent bundle $TV^{2n+1}_{g}$ and the top-horizontal map is the map underlying the fixed $\theta^{n}$-structure $\ell_{D}: TD\oplus\epsilon^{1} \longrightarrow (\theta^{n})^{*}\gamma^{2n+1}$. 
Let $F(\theta^{n})$ denote the homotopy fibre of the map $\theta^{n}$. 
Since $\theta^{n}$ is the $n$-connected cover, it follows that $\pi_{i}(F(\theta^{n})) = 0$ for all $i \geq n$. 
Weak contractibility of this space of lifts then follows from the fact that $V^{2n+1}_{g}$ is homotopy equivalent to $(S^{n})^{\vee g}$. 
This concludes the proof of the proposition.
\end{proof}

Proposition \ref{proposition: theta-n structures} yields the weak homotopy equivalence 
$$
\colim_{g\to\infty}\textstyle{\BDiff_{\theta^{n}}}(V^{2n+1}_{g}, D^{2n}) \; \simeq \; \displaystyle{\colim_{g\to\infty}}\BDiff(V^{2n+1}_{g}, D^{2n}).
$$
Combining this weak homotopy equivalence with Theorem A* then yields the homology isomorphism 
$$
\colim_{g\to\infty}H_{*}(\BDiff(V^{2n+1}_{g}, D^{2n}); \Z) \; \cong \; H_{*}(Q_{0}BO(2n+1)\langle n \rangle_{+}; \Z)
$$
and concludes the proof of Theorem A.

\appendix

\section{Surgery on Morphisms} \label{section: surgery on morphisms}
In this section we prove that the inclusion $B\Cob^{\partial, \mb{c}}_{\theta} \hookrightarrow B\Cob^{\partial, \mb{b}}_{\theta}$ is a weak homotopy equivalence in the case that $2n+1 \geq 5$.
This constitutes one of the weak homotopy equivalences stated in Theorem \ref{theorem: easy weak equivalences}. 
We will prove this weak homotopy equivalence in stages by filtering $\Cob^{\partial, \mb{b}}_{\theta}$ by a sequence of intermediate subcategories contained in $\Cob^{\partial, \mb{c}}_{\theta}$. 
We begin with a definition. 
\begin{defn} \label{defn: intermediate morphism category}
Let $k \in \Z_{\geq-1}$. 
The topological subcategory $\Cob^{\partial, \mb{b}, k}_{\theta} \subset \Cob^{\partial, \mb{b}}_{\theta}$ has the same space of objects. 
Its morphisms are those $(t, (W, \ell))$ such that the triad 
$(W|_{[0, t]}; \; W|_{t}, \; (\partial W)|_{[0, t]})$
is $k$-connected. .  
\end{defn}
For each $k \in \Z_{\geq-1}$, the corresponding semi-simplicial spaces $\mb{D}^{\partial, \mb{b}, k}_{\bullet}$ and $\mb{X}_{\bullet}^{\partial, \mb{b}, k}$ are defined analogously to Definitions \ref{defn: flexible model 1} and \ref{defn: more flexible model} so that there are weak homotopy equivalences
$$B\Cob^{\partial, \mb{b}, k}_{\theta} \simeq |\mb{D}^{\partial, \mb{b}, k}_{\bullet}| \simeq |\mb{X}_{\bullet}^{\partial, \mb{b}, k}|.$$
By definition we have $\Cob^{\partial, \mb{b}, n}_{\theta} = \Cob^{\partial, \mb{c}}_{\theta}$.
We will prove the following theorem. 
\begin{theorem} \label{theorem: highly connected morphisms}
Let $2n+1 \geq 5$.
Then the inclusion $\Cob^{\partial, \mb{b}, k}_{\theta} \hookrightarrow \Cob^{\partial, \mb{b}, k-1}_{\theta}$ induces a weak homotopy equivalence
$
B\Cob^{\partial, \mb{b}, k}_{\theta} \; \simeq \; B\Cob^{\partial, \mb{b}, k-1}_{\theta}
$
whenever $k \leq n$. 
\end{theorem}
The proof of this weak homotopy equivalence is
similar to \cite[Theorem 3.1]{GRW 14}; we use the
same construction but take it one step further and achieve a slightly
stronger connectivity condition on the morphisms in our cobordism
category.  In particular, we use a bi-semi-simplicial
  space very similar to the one from \cite[Section 3.1]{GRW 14} and
the same surgery move that was constructed in \cite[Section 3.2]{GRW 14}.

\subsection{Surgery data} \label{subsection: morphism surgery data}
We are now ready for the main construction of this section used in the proof of Theorem \ref{theorem: highly connected morphisms}. 
We must first fix some notation. 
We denote $d := 2n+1$, which is the dimension of the morphisms in $\Cob_{\theta}^{\partial}$.
Let 
$V \subset \bar{V} \subset \R^{k+1}\times\R^{d - k}$ 
be the subspaces 
$$
V = (-2, 0)\times\R^{k}\times\R^{d-k}, \quad \quad \bar{V} = [-2, 0]\times\R^{k}\times\R^{d-k},
$$
and let 
$h: \bar{V} \longrightarrow [-2, 0] \subset \R$ 
denote the projection onto the first coordinate.
Let 
$$\partial_{-}D^{k+1}\subset \partial D^{k+1} \subset \R^{k+1}$$ 
denote the lower hemisphere of the boundary of the disk, i.e.\ 
$$
\partial_{-}D^{k+1} = \{(x_{1}, \dots, x_{k+1}) \in \R^{k+1} \; | \; \textstyle{\sum}x_{i}^{2} = 1, \quad x_{1} \leq 0 \; \}.
$$
Let $\triangle$ denote the \textit{simplex category}. 
For $[p] = \{0, \dots, p\} \in \Ob(\triangle)$, we use the notation 
$[p]^{\vee} = \triangle([p], [1]),$
 which is the set of monotone non-decreasing maps from $[p]$ to $[1]$. 
The elements of $[p]^{\vee}$ are in bijection with the set $\{0, \dots, p+1\}$, using the convention that $\varphi: [p] \rightarrow [1]$ corresponds to the number $i$ if $\varphi^{-1}(1) = \{i, i+1, \dots, p\}$.
Next definition should be compared to \cite[Definition 3.2]{GRW 14}.
\begin{defn} \label{defn: Surgery on morphisms data}
Fix once and for all an infinite set $\Omega$. 
Let 
$x = (a, \varepsilon, (W, \ell_{W})) \in \mb{D}^{\partial, \mb{b}, k-1}_{p}$
be a $p$-simplex. 
We define $\mb{Z}_{0}(x)$ to be the set of triples $(\Lambda, \delta, e)$, where $\Lambda \subset \Omega$ is a finite set, $\delta: \Lambda \longrightarrow [p]^{\vee}$ is a function, and 
$
e: \Lambda\times\bar{V} \hookrightarrow \R\times[0, 1)\times(-1, 1)^{\infty-1}
$ is an embedding subject to the following
  conditions:
\begin{enumerate} \itemsep.3cm
\item[\textcolor{black}{(i)}] the image $e(\Lambda\times\bar{V})$ is contained in the interior of $\R\times[0, 1)\times(-1, 1)^{\infty-1}$;
\item[\textcolor{black}{(ii)}] 
on each subset $(x_{0}\circ e|_{\{\lambda\}\times\bar{V}})^{-1}(a_{j}-\varepsilon_{j}, a_{j} + \varepsilon_{j}) \subset \{\lambda \}\times\bar{V}$, the function 
$$x_{0}\circ e: \Lambda\times\bar{V} \longrightarrow \R$$ coincides
with the function $h$ up to an affine translation, where recall
$x_{0}: \R\times\R^{\infty}_{+} \longrightarrow \R$ is the projection
onto the first coordinate;
\item[\textcolor{black}{(iii)}] the embedding $e$ sends $\Lambda\times
  h^{-1}(0)$ into $x_{0}^{-1}(a_{p}+\varepsilon_{p}, \infty)$ and sends $\Lambda\times h^{-1}(-2)$ into $x_{0}^{-1}(-\infty,
  a_{0}-\varepsilon_{0})$;
\item[\textcolor{black}{(iv)}] for $i > 0$, $e$ sends $\delta^{-1}(i)\times h^{-1}(-\tfrac{3}{2})$ into $x_{0}^{-1}(a_{i-1}+\varepsilon_{i-1}, \infty)$; 
\item[(v)] $e^{-1}(W\setminus\partial W) = \Lambda\times\partial_{-}D^{k+1}\times\R^{d-k}$.
\end{enumerate}
For each $i \in \{0, \dots, p\}$ let us denote 
$D_{i} = e(\delta^{-1}(i)\times\partial_{-}D^{k+1}\times\{0\}).$
 The above conditions imply that $D_{i} \subset W\setminus\partial W$. 
Below is our final condition:
\begin{enumerate}
\item[(vi)] For each $i \in \{1, \dots, p\}$, we require that
the triad
$$
\left(W|_{[a_{i-1}, a_{i}]}; \; (\partial W)|_{[a_{i-1}, a_{i}]}, \; W|_{a_{i}}\cup D_{i}|_{[a_{i-1}, a_{i}]}\right)
$$
is $k$-connected. 
\end{enumerate}
We will refer to $\mb{Z}_{0}(x)$ as the \textit{space of surgery data} associated to the element $x$. 
The way in which $\mb{Z}_{0}(x)$ is topologized will be described below in Definition \ref{defn: semi-simplicial space of morphism surgery data}. 
\end{defn}

\begin{remark}
In the above definition, the pairs $\left((\partial W)|_{[a_{i-1}, a_{i}]}, \; (\partial W)|_{a_{i}}\right)$ are $(n-1)$-connected for all $i = 0, \dots, p$. 
Thus if $k \leq n$,  then the triad
$$
\left(W|_{[a_{i-1}, a_{i}]}; \; (\partial W)|_{[a_{i-1}, a_{i}]}, \; W|_{a_{i}}\cup D_{i}|_{[a_{i-1}, a_{i}]}\right)
$$ 
is $k$-connected whenever the pair $\left(W|_{[a_{i-1}, a_{i}]}, \; W|_{a_{i}}\cup D_{i}|_{[a_{i-1}, a_{i}]}\right)$ is $k$-connected. 
This can be seen by using the long exact sequence on homotopy groups associated to the triad.
However, there are of course 
cases where $\left(W|_{[a_{i-1}, a_{i}]};  \; (\partial W)|_{[a_{i-1}, a_{i}]}, \; W|_{a_{i}}\cup D_{i}|_{[a_{i-1}, a_{i}]} \right)$ is $k$-connected but the pair $\left(W|_{[a_{i-1}, a_{i}]}, \; W|_{a_{i}}\cup D_{i}|_{[a_{i-1}, a_{i}]}\right)$ fails to be $k$-connected.
\end{remark}

\begin{remark} \label{remark: main remark on morphism surgery data}
We also observe that if the element $x = (a, \varepsilon, (W, \ell_{W}))$ is contained in the subspace $\mb{D}^{\partial, \mb{b}, k}_{p} \subset \mb{D}^{\partial, \mb{b}, k-1}_{p}$, then the \textit{empty surgery data}, determined by setting $\Lambda = \emptyset$, yields a valid element of $\mb{Z}_{0}(x)$. 
This observation will prove important in the construction of a certain map in Section \ref{subsection: formal proof of theorem}.
\end{remark}

We now enhance the space $\mb{Z}_{0}(x)$ to a semi-simplicial
space $\mb{Z}_{\bullet}(x)$ such that the space of $0$-simplices has
$\mb{Z}_{0}(x)$ as its underlying set.
\begin{defn} \label{defn: semi-simplicial space of morphism surgery data}
Let $x \in \mb{D}^{\partial, \mb{b}, k-1}_{p}$.
The set of $q$-simplices $\mb{Z}_{q}(x)$ is defined to be the set of triples $(\Lambda, \delta, e)$, where $\Lambda \subset \Omega$ is a finite set, $\delta: \Lambda \longrightarrow [p]^{\vee}\times[q]$ is a function, and 
$$
e: \Lambda\times\bar{V} \longrightarrow \R\times[0, 1)\times(-1, 1)^{\infty-1}
$$
is an embedding, subject to the requirement that for each $j \in [q]$, the restriction of the embedding $e$ to $\delta^{-1}([p]^{\vee}\times\{j\})\times\bar{V}$ defines an element of $\mb{Z}_{0}(x)$. 
For all $q$, the space $\mb{Z}_{q}(x)$ is topologized as a subspace  of 
$$
\bigg(\coprod_{\Lambda \subset \Omega}C^{\infty}(\Lambda\times\bar{V}, \R^{\infty+1})\bigg)^{q+1}.
$$
We shall write $\Lambda_{i, j} = \delta^{-1}(i, j)$ and $e_{i, j} = e|_{\Lambda_{i, j}\times\bar{V}}$ for each $i \in [p]^{\vee}$ and $j \in [q]$.
\end{defn}
Using the semi-simplicial spaces constructed in the above definition, we 
define a bi-semi-simplicial space $\mb{D}^{\partial, \mb{b}, k}_{\bullet, \bullet}$ by setting 
\begin{equation} \label{equation: resolution by handle subtractions}
\mb{D}^{\partial, \mb{b}, k}_{p, q} := \{(x, y) \; : \; x \in \mb{D}^{\partial, \mb{b}, k-1}_{p}, \;  \; y \in \mb{Z}_{q}(x) \}.
\end{equation}
The assignment $([p], [q]) \mapsto \mb{D}^{\partial, \mb{b}, k}_{p, q}$ is functorial in $[p]$ by composing each $\delta_{j}: \Lambda_{j} \rightarrow [p]^{\vee}$ with the induced map $[p']^{\vee} \rightarrow [p]^{\vee}$. 
It is functorial in $[q]$ in the same way that $\mb{Z}_{q}(x)$ is.

For each pair $(p, q)$ there is a forgetful map
$\mb{D}^{\partial, \mb{b}, k}_{p, q} \longrightarrow \mb{D}^{\partial, \mb{b}, k-1}_{p},$ $(x, y) \mapsto x.$
By setting $\mb{D}^{\partial, \mb{b}, k}_{\bullet, -1} = \mb{D}^{\partial, \mb{b}, k-1}_{\bullet}$,
these maps piece together to define an augmented bi-semi-simplicial space
$\mb{D}^{\partial, \mb{b}, k}_{\bullet, \bullet} \; \longrightarrow \; \mb{D}^{\partial, \mb{b}, k}_{\bullet, -1}.$
The following theorem is is proven in the same way as \cite[Theorem 3.4]{GRW 14}. 
\begin{theorem} \label{theorem: contractible surgery data for morphisms}
Let $k \leq n$.
Then the map 
$|\mb{D}^{\partial, \mb{b}, k}_{\bullet, \bullet}| \; \longrightarrow \; |\mb{D}^{\partial, \mb{b}, k}_{\bullet, -1}| = |\mb{D}^{\partial, \mb{b}, k}_{\bullet}|$ is a weak homotopy equivalence.
\end{theorem}

\subsection{The surgery move}
We now implement the surgery move construction from \cite[Section 3.2]{GRW 14}.
Recall the subspaces 
$
V \subset \bar{V} \subset \R^{k+1}\times\R^{d-k}
$
used in the previous section.
The following proposition is a restatement of \cite[Proposition 3.6]{GRW 14}.
\begin{proposition} \label{proposition: morphism surgery move}
Suppose that $d = 2n + 1 \geq 3$ and let $k \leq n$. 
Let $\ell$ be a $\theta$-structure on $\Int(\partial_{-}D^{k+1})\times\R^{d-k} \subset V$. 
There exists a one parameter family  
$$(\mathcal{P}_{t}, \ell_{t}) \in \bPsi_{\theta}(V), \quad t \in [0, 1],$$ 
that satisfies the following conditions:
\begin{enumerate} \itemsep.2cm
\item[\textcolor{black}{(i)}] $(\mathcal{P}_{0}, \ell_{0}) =  (\Int(\partial_{-}D^{k+1})\times\R^{d-k}, \; \ell).$
\item[\textcolor{black}{(ii)}] The restriction of the height function $h: V \longrightarrow (-2, 0)$ to $\mathcal{P}_{t} \subset V$, has isolated critical values.
\item[\textcolor{black}{(iii)}] Independently of $t$, we have 
$$
\mathcal{P}_{t}\setminus(\R^{k+1}\times B_{3}^{d-k}(0)) \; = \; \Int(\partial_{-}D^{k+1})\times(\R^{d-k}\setminus B_{3}^{d-k}(0)).
$$
We will denote by $\mathcal{P}^{\partial}_{t}$ this closed subset of $\mathcal{P}_{t}$.
\item[\textcolor{black}{(iv)}] For all $t$ and each pair of regular values $-2 < a < b < 0$ of the height function, the pair 
$$
(\mathcal{P}_{t}|_{[a, b]}, \; \mathcal{P}_{t}|_{b}\cup\mathcal{P}_{t}^{\partial}|_{[a,b]})
$$
is $k$-connected. 
\item[\textcolor{black}{(v)}] For each pair of regular values $-2 < a < b < 0$ of the height function, the pair 
$$
(\mathcal{P}_{1}|_{[a, b]}, \; \mathcal{P}_{1}|_{b})
$$
is $k$-connected. 
\end{enumerate}
Furthermore, the family $(\mathcal{P}_{t}, \ell_{t})$ varies continuously in $\ell$. 
\end{proposition}

The one-parameter family from the above proposition will be used as follows.
Let 
$$(a, \varepsilon, (W, \ell_{W}), e) \in \mb{D}^{\partial, \mb{b}, k}_{p, 0}.$$
We construct a one-parameter family 
\begin{equation}
\mathcal{K}^{t}_{e}(W, \ell_{W}) \in \bPsi^{\partial}_{\theta}((a_{0}-\varepsilon_{0}, a_{p}+\varepsilon_{p})\times\R^{\infty}_{+})
\end{equation}
by letting it equal $W|_{(a_{0}-\varepsilon_{0}, a_{p}+\varepsilon_{p})}$ outside the images of the embeddings $e_{i, 0}|_{\Lambda_{i, j}\times V}$, and on each of the subsets 
$$
e_{i, 0}(\{\lambda\}\times V)\cap\left((a_{0}-\varepsilon_{0}, a_{p}+\varepsilon_{p})\times\R^{\infty}_{+}\right),
$$ 
we set it equal to the $\theta$-manifold
$$\left(e_{i, 0}(\{\lambda\}\times\mathcal{P}_{t}), \; \ell_{W}\circ De_{i, 0}\right),$$ 
where 
$De_{i, 0}$ denotes the differential of the embedding $e_{i, 0}$.  
\begin{remark}
Notice that since $e^{-1}(\partial W) = \emptyset$, the above construction does not touch the boundary of $W$ and so $\partial \mathcal{K}^{t}_{e}(W, \ell_{W}) \; = \; (\partial W, \ell|_{\partial W})$ for all $t \in [0, 1]$.
\end{remark}

The main result that we will need regarding the family $\mathcal{K}^{t}_{e}(W, \ell_{W})$ is the lemma below.
\begin{lemma} \label{lemma: surgery implementation lemma}
Let $d = 2n + 1 \geq 5$,  let $k \leq n$, and let $(a, \varepsilon, (W, \ell), e) \in \mb{D}^{\partial, \mb{b}, k}_{p, 0}$. 
Then for all $t \in [0, 1]$, the tuple $(a, \varepsilon, \mathcal{K}^{t}_{e}(W, \ell_{W}))$ is an element of $\mb{X}^{\partial, \mb{b}, k-1}_{p}$.
If $t = 1$ or $(a, \varepsilon, (W, \ell_{W})) \in \mb{D}^{\partial, \mb{b}, k}_{p}$, then $(a, \varepsilon, \mathcal{K}^{t}_{e}(W, \ell_{W}))$ lies in the subspace $\mb{X}^{\partial, \mb{b}, k}_{p} \subset \mb{X}^{\partial, \mb{b}, k-1}_{p}$.
\end{lemma}
In the case that $k < n$, the above lemma is proven in the same way as \cite[Lemma 3.7]{GRW 14}.
For when $k = n$, the proof will require a little something extra. 
To prove Lemma \ref{lemma: surgery implementation lemma} in the case that $k = n$, 
we will need to use the following basic propositions below.

For the statement of the next proposition, 
let $W$ be a $(2n+1)$-dimensional manifold and let 
$M \subset \partial W$ be a $2n$-dimensional submanifold. 
Let 
$\varphi: (D^{n}, \partial D^{n}) \hookrightarrow (W, M)$ 
be an embedding and set $(W', M')$ equal to the pair defined by
$
(W', M') = (W\setminus\varphi(D^{n}), \; M\setminus\varphi(\partial D^{n})).
$
\begin{proposition} \label{proposition: handle subtraction connectivity}
Let $(W, M)$ and $(W', M')$ be as above and let $n \geq 2$. 
Suppose that the pair $(W, M)$ is $1$-connected.
Then the map induced by inclusion,
$
\pi_{i}(W', M') \longrightarrow \pi_{i}(W, M),
$
is an isomorphism for all $i \leq n$.
Furthermore, if $N \subset \partial W$ is another submanifold with $\partial N \subset N\cap M$ and $N\cap\varphi(D^{n}) = \emptyset$, then the map 
$
\pi_{i}(W'; M', N) \longrightarrow \pi_{i}(W; M, N),
$
is an isomorphism for all $i \leq n$ as well. 
\end{proposition}
\begin{proof}
In the case that $i < n$, the fact that $\pi_{i}(W', M') \longrightarrow \pi_{i}(W, M)$ is isomorphism follows by a general position argument. 
Indeed, the co-dimension of $\varphi(D^{n}) \subset W$ is $n+1$, to follows that any map $(D^{i}, \partial D^{i}) \longrightarrow (W, M)$ (or even homotopy between such maps) can be deformed by a small homotopy so as to have its image in $(W', M')$. 
A similar general position argument implies that $\pi_{n}(W', M') \longrightarrow \pi_{n}(W, M)$ is surjective. 
We now focus on proving that $\pi_{n}(W', M') \longrightarrow \pi_{n}(W, M)$ is injective. 
The proof of this will require use of the \textit{Half Whiney trick} (see Theorem \ref{theorem: half whitney
  trick}). 
Let 
$$
f: (D^{n}\times[0, 1], \; \partial D^{n}\times[0,1]) \; \longrightarrow \; (W\times[0, 1], M\times[0,1])
$$
be a map with 
$$
f(D^{n}\times\{\nu\}) \subset W'\times\{\nu\} \quad \text{and} \quad  f(\partial D^{n}\times\{\nu\}) \; \subset \; M'\times\{\nu\} \quad \text{for $\nu = 0, 1$.}
$$  
The proposition will be proven if we can deform $f$ through a
homotopy, relative to $D^{n}\times\{0, 1\}$, to a map $f'$ with image
in $(W'\times[0, 1], M'\times[0,1])$.  This will imply that any two
maps
$$
(D^{n}, \partial D^{n}) \longrightarrow (W', M')
$$
that are homotopic in $(W, M)$ are actually homotopic in $(W', M')$. 

By general position, we may assume that to begin with the restriction
$f|_{D^{n}\times\{0, 1\}}$ is an embedding.  Using the fact that $(W,
M)$ is $1$-connected, we may apply Theorem \ref{theorem: hudsen
  embedding theorem} (or \cite[Theorem 1]{H 69}) to obtain a homotopy
$$f_{t}: (D^{n}\times[0, 1], \; \partial D^{n}\times[0,1]) \;
\longrightarrow \; (W\times[0, 1], M\times[0,1]), \quad t \in [0, 1],$$
with $f_{0} = f$, relative to $D^{n}\times\{0, 1\}$, such that $f_{1}$
is an embedding.  We then may apply Theorem \ref{theorem: half whitney
  trick} to obtain an isotopy of
$f_{1}$, relative to $D^{n}\times\{0, 1\}$, to a new embedding $f'$
such that
$$f'(D^{n}\times[0, 1])\bigcap(\varphi(D^{n})\times[0, 1]) = \emptyset,$$ 
and thus $f'(D^{n}\times[0,1]) \subset W'\times[0,1]$. 
This proves that $\pi_{n}(W', M') \longrightarrow \pi_{n}(W, M)$ is injective and hence is an isomorphism.

We now prove that $\pi_{i}(W'; M', N) \longrightarrow
  \pi_{i}(W; M, N)$ is an isomorphism for all $i \leq n$.  
  As before, the proof that this map is an isomorphism for $i < n$ and surjective for $i = n$ follows by a general position argument. 
  We focus on proving that $\pi_{n}(W'; M', N) \longrightarrow \pi_{n}(W; M, N)$ is injective. 
Consider the map of exact sequences
$$
\xymatrix{
 \pi_{n}(M', M'\cap N) \ar[r] \ar[d] & \pi_{n}(W', M') \ar[r] \ar[d]^{\cong} &  \pi_{n}(W; M, N) \ar[r] \ar[d] & \pi_{n-1}(M', M'\cap N) \ar[r] \ar[d]  &\cdots \\
\pi_{n}(M, M\cap N) \ar[r] &  \pi_{n}(W, M) \ar[r] & \pi_{n}(W; M\cup D, N) \ar[r] & \pi_{n-1}(M, M\cap N) \ar[r] & \cdots
}
$$ 
Left-most vertical map is surjective by general position, as any
map $(D^{n}, \partial D^{n}) \rightarrow (M, N\cap M)$ can be deformed
to have image disjoint from $\varphi(\partial D^{n}) \subset M\setminus N$.  
By a similar general position argument it follows that
the right-most vertical map is an isomorphism.  The second vertical
map is an isomorphism by what was proven in the previous paragraph.  
It follows from the
\textit{four lemma} (or just a simple diagram chase) that the third
vertical map $\pi_{n}(W'; M', N) \longrightarrow \pi_{n}(W; M, N)$ is
injective and hence an isomorphism.  This concludes the proof of the
proposition.  
\end{proof}

With the above preliminary result established,
we now can give the proof of Lemma \ref{lemma: surgery implementation lemma}.

\begin{proof}[Proof of Lemma \ref{lemma: surgery implementation lemma}]
We prove the lemma explicitly for the case that $k = n$. 
The other cases are similar. 
Let $(a, \varepsilon, (W, \ell), e) \in \mb{D}^{\partial, \mb{b}, n}_{p, 0}$.
We let 
$W_{t}$ denote $\mathcal{K}^{t}_{e}(W, \ell_{W})$. 
Notice that by construction $\partial W_{t} = \partial W$ for all $t \in [0, 1]$. 
Consider the height function $x_{0}: W_{t} \longrightarrow \R$.
We begin by showing that if 
$(a, \varepsilon, (W, \ell))$ is contained in $\mb{D}^{\partial, \mb{b}, n}_{p},$
then for any two regular values 
$$a < b \in \bigcup^{p}_{i=0}(a_{i}-\varepsilon_{i}, a_{i}+\varepsilon_{i})$$
of the height function $x_{0}$, 
the triad
$\left(W_{t}|_{[a, b]}; \; (\partial W)|_{[a, b]}, \; W_{t}|_{b}\right)$
is $n$-connected, and hence the triple
$(a, \varepsilon, \mathcal{K}^{t}_{e}(W, \ell))$ is contained in $\mb{X}^{\partial, \mb{b}, n}_{p}$ for any $t \in [0, 1]$.

Assume that $(a, \varepsilon, (W, \ell)) \in \mb{D}^{\partial, \mb{b}, n}_{p}$.
The manifold $W_{t}|_{[a, b]}$ is obtained from $W|_{[a, b]}$ by cutting out embedded images of cobordisms $\mathcal{P}_{0}|_{[a_{\lambda}, b_{\lambda}]}$ indexed by $\lambda \in \Lambda = \coprod^{p+1}_{i=0}\Lambda_{i, 0}$ and gluing in $\mathcal{P}_{t}|_{[a_{\lambda}, b_{\lambda}]}$, where $a_{\lambda} < b_{\lambda}$ are regular values of the height function on $\mathcal{P}_{0}$ and $\mathcal{P}_{1}$. 
Let $X$ denote the complement of $e_{i, 0}(\Int(\partial_{-}D^{n+1}\times B^{d-n}_{3}(0)))$ in the manifold $W|_{[a, b]}$. 
Notice that $X\cap(\partial W)|_{[a, b]} = (\partial W)|_{[a, b]}$. 
There are homotopy pushout squares 
\begin{equation} \label{equation: pushout square 1}
\xymatrix{ \left(X|_{b}, \; (\partial W)|_{b}\right)
  \ar[rr] \ar[d] && \left(W_{t}|_{b}, \; (\partial
    W)|_{b}\right) \ar[d] \\ 
    \left(X, \; (\partial W)|_{[a,b]}
    \right) \ar[rr] && \left(W_{t}|_{b}\cup X, \; (\partial W)|_{[a,b]}\right), }
\end{equation}
and 
\begin{equation} \label{equation: pushout square 2}
\xymatrix{
\coprod_{\lambda \in \Lambda}\left((\mathcal{P}_{t}|_{b_{\lambda}}\cup \mathcal{P}^{\partial}_{t}|_{[a_{\lambda}, b_{\lambda}]}), \; \emptyset \right)
\ar[d] \ar[rr] && \left(W_{t}|_{b}\cup X, \; (\partial W)|_{[a,b]}\right) \ar[d] \\
\coprod_{\lambda \in \Lambda}\left(\mathcal{P}_{t}|_{[a_{\lambda}, b_{\lambda}]}, \; \emptyset \right)
  \ar[rr] && \left(W_{t}|_{[a, b]}, \; (\partial W)|_{[a,b]}\right).
}
\end{equation}
The pair $(X, X|_{b})$ was obtained from the pair $(W|_{[a, b]},
W_{b})$ by cutting out embedded disks $(D^{n}, \partial D^{n})$.
Since $\dim(W) = 2n+1 \geq 5$, by Proposition \ref{proposition: handle
  subtraction connectivity} there is an isomorphism
\begin{equation} \label{equation: relative homotopy isomorphism}
\pi_{i}\left(X; \; X|_{b}, \; (\partial W)|_{[a,b]}\right) \cong \pi_{i}\left(W|_{[a, b]}; \; W|_{b}, \; (\partial W)|_{[a,b]}\right)
\end{equation}
for all $i \leq n$.  Since the triad $\left(W|_{[a, b]};\; W|_{b},\;
(\partial W)|_{[a,b]}\right)$ is $n$-connected by assumption, it
follows form these isomorphisms that $\left(X;\; X|_{b},\; (\partial
W)|_{[a, b]}\right)$ is $n$-connected as well and thus the
left-vertical map of (\ref{equation: pushout square 1}) is
$n$-connected (induces isomorphism on $\pi_{i}$ for $i < n$ and
epimorphism on $\pi_{n}$).  Since (\ref{equation: pushout square 1})
is a homotopy pushout square, it follows that the right-hand vertical
map of (\ref{equation: pushout square 1}) is $n$-connected as well.
The left-hand vertical map of (\ref{equation: pushout square 2}) is
$n$-connected by property \textcolor{black}{(iv)} of Proposition
\ref{proposition: morphism surgery move}.  Since the square is a
homotopy pushout, it follows that the right-hand vertical map of
(\ref{equation: pushout square 2}) is $n$-connected as well.  It
follows that the composition of the right-hand vertical maps of
(\ref{equation: pushout square 1}) and (\ref{equation: pushout square
  2}),
$$
\left(W_{t}|_{b}, \; (\partial W)|_{b}\right) \; \longrightarrow \; \left(W_{t}|_{b}\cup X, \; (\partial W)|_{[a,b]}\right) \; \longrightarrow \; \left(W_{t}|_{[a, b]}, \; (\partial W)|_{[a,b]}\right),
$$ is $n$-connected for all $t \in [0, 1]$ and hence the triad
$\left(W_{t}|_{[a, b]}; \; W_{t}|_{b}, \; (\partial
W)|_{[a,b]}\right)$ is $n$-connected for all $t \in [0,1]$.  This
proves that $(a, \varepsilon, \mathcal{K}^{t}_{e}(W, \ell)) \in
\mb{X}^{\partial, \mb{b}, n}_{p}$ whenever $(a, \varepsilon, (W,
\ell)) \in \mb{D}^{\partial, \mb{b}, n}_{p}$.

The isomorphism (\ref{equation: relative homotopy isomorphism}) did
not depend on whether or not $(a, \varepsilon, (W, \ell))$ was an
element of the subspace $\mb{D}^{\partial, \mb{b}, n}_{p} \subset
\mb{D}^{\partial, \mb{b}, n-1}_{p}.$ Therefore, it is automatic that
$(X, X|_{b})$ is $(n-1)$-connected for any $(a, \varepsilon, (W,
\ell), e) \in \mb{D}^{\partial, \mb{b}, n}_{p, 0}$.  It then follows
by the same argument made above that the pair $(W_{t}|_{[a, b]},
W_{t}|_{b})$ is $(n-1)$-connected.  It follows that $(a, \varepsilon,
\mathcal{K}^{t}_{e}(W, \ell)) \in \mb{X}^{\partial, \mb{b}, n-1}_{p}$
for any choice of $(a, \varepsilon, (W, \ell), e) \in
\mb{D}^{\partial, \mb{b}, n}_{p, 0}$.

When $t = 1$ and the element $(a, \varepsilon, (W, \ell)) \in \mb{D}^{\partial, \mb{b}, n-1}_{p}$ is arbitrary, the proof requires an additional argument.
The proof breaks down into three cases. 

\textit{Case 1.}  Suppose that $a < b \in (a_{i}-\varepsilon_{i},
a_{i}+\varepsilon_{i})$.  Then the triad $(W|_{[a, b]}; (\partial
W)|_{[a, b]}, W|_{b})$ is $\infty$-connected.  It follows then that
$(W_{1}|_{[a, b]}; (\partial W)|_{[a, b]}, W_{1}|_{b})$ is
$n$-connected by the same argument employed in the paragraph above.

\textit{Case 2.}  Suppose that $a \in (a_{i-1} - \varepsilon_{i-1},
a_{i-1} + \varepsilon_{i-1})$ and $b \in (a_{i} - \varepsilon_{i},
a_{i} + \varepsilon_{i})$.  Let $\widetilde{W}_{t}$ be the family of
manifolds obtained by performing the surgery move that corresponds to
the set $\Lambda_{i, 0} \subset \Lambda$.  We will show that the triad
$(\widetilde{W}_{1}|_{[a, b]}; \; (\partial W)|_{[a, b]}, \;
\widetilde{W}_{1}|_{b})$ is $n$-connected.  Once this is
established, doing the remaining surgeries does not change this
property by the argument given in the first part of this proof.

By Definition \ref{defn: Surgery on morphisms data} (condition (vii)),
the triad $\left(W_{0}|_{[a, b]}; \; (\partial W)|_{[a, b]}, \;
W_{0}|_{b}\cup D_{i, 0}|_{[a, b]}\right)$ is $n$-connected, where
$$
D_{i, 0} \; = \; e_{i, 0}(\Lambda_{i, 0}\times\partial_{-}D^{n+1}\times\{0\}) \; \subset \; W \; = \; W_{0}.
$$
We denote 
$$
\widetilde{D}_{i, 0} = e_{i, 0}(\Lambda_{i,
  0}\times\partial_{-}D^{n+1}\times\{v\}) \subset W = W_{0}
$$
for some $v \in B^{n+1}_{4}(0)$.  It follows that the triad
$$ (W_{0}|_{[a, b]}; \; (\partial W)|_{[a, b]}, \;
W_{0}|_{b}\cup\widetilde{D}_{i, 0}|_{[a, b]})$$
is also $n$-connected.  As before, we let $X \subset W$ denote the
complement of $e_{i, 0}(\Int(\partial_{-}D^{n+1}\times
B^{d-n}_{3}(0)))$ in the manifold $W|_{[a, b]}$.  By the same argument
used to prove Proposition \ref{proposition: handle subtraction
  connectivity}, it follows that
$$ \pi_{i}(X;\; (\partial W)|_{[a, b]}, \; X|_{b}\cup\widetilde{D}_{i,
  0}|_{[a, b]}) \; \longrightarrow \; \pi_{i}(W_{0}|_{[a, b]}; \;
(\partial W)|_{[a, b]}, \; W_{0}|_{b}\cup\widetilde{D}_{i, 0}|_{[a,
    b]}) $$ is an isomorphism for all $i \leq n$, and thus the
triad $(X;\; (\partial W)|_{[a, b]}, \; X|_{b}\cup\widetilde{D}_{i,
  0}|_{[a, b]})$ is $n$-connected as well.  The same gluing argument
employed in the first part of the proof implies that
\begin{equation} \label{equation: tilde triad}
(\widetilde{W}_{t}|_{[a, b]}; \; (\partial W)|_{[a, b]}, \; \widetilde{W}_{t}|_{b}\cup\widetilde{D}_{i, 0}|_{[a, b]})
\end{equation}
is $n$-connected of all $t \in [0, 1]$. 

The subset $\widetilde{D}_{i, 0} \subset W$ is contained in $e_{i,
  0}(\Lambda_{i, 0}\times\mathcal{P}^{\partial}_{0})$, so we can
regard $\widetilde{D}_{i, 0}$ as a subset of $\widetilde{W}_{t}$ for
all $t \in [0, 1]$.  By construction, the inclusion
\begin{equation} \label{equation: inclusion of disk}
  (\widetilde{D}_{i, 0}|_{[a, b]}, \; \widetilde{D}_{i, 0}|_{b}) \;
  \hookrightarrow \; (\widetilde{W}_{1}|_{[a, b]}, \; \widetilde{W}_{1}|_{b})
\end{equation}
factors through the pair $\left(e_{i, 0}(\Lambda_{i,
  0}\times\mathcal{P}_{1})|_{[a, b]}, \; e_{i, 0}(\Lambda_{i,
  0}\times\mathcal{P}_{1})|_{b}\right),$ and thus Proposition
\ref{proposition: morphism surgery move} (condition
\textcolor{black}{(v)}) implies that (\ref{equation: inclusion of
  disk}) is homotopic (through maps of pairs) to a map with image in
$\widetilde{W}_{1}|_{b}$.  Thus, the inclusion
$$
(\widetilde{W}_{1}|_{[a, b]}; \; (\partial W)|_{[a, b]}, \; \widetilde{W}_{1}|_{b}) \; \hookrightarrow \; (\widetilde{W}_{1}|_{[a, b]}; \; (\partial W)|_{[a, b]}, \; \widetilde{W}_{1}|_{b}\cup\widetilde{D}_{i, 0}|_{[a, b]})
$$
is a homotopy equivalence of triads. 
Since the triad on the right is $n$-connected (as established in (\ref{equation: tilde triad})), it follows that 
$(\widetilde{W}_{1}|_{[a, b]}; \; (\partial W)|_{[a, b]}, \; \widetilde{W}_{1}|_{b})$ 
is $n$-connected as well. 
This concludes the proof of the lemma in this case when $a \in (a_{i-1}-\varepsilon_{i-1}, a_{i-1}+\varepsilon_{i-1})$ and $b \in (a_{i} - \varepsilon_{i}, a_{i}+\varepsilon_{i})$.  

\textit{Case 3.} Suppose that $a < b \in \bigcup^{p}_{i = 0}(a_{i}-\varepsilon_{i}, a_{i}+\varepsilon_{i})$. 
By choosing regular values in each intermediate interval $(a_{j}-\varepsilon_{j}, a_{j}+\varepsilon_{j})$, we may break up $(W|_{[a, b]}, (\partial W)|_{[a, b]})$ into a composition of relative cobordisms, for each of which we can apply the argument from case 2. 
In this way, we may express $(W_{1}|_{[a, b]}, (\partial W_{1})|_{[a, b]})$ as a composition of relative cobordisms, all of which satisfy the required connectivity conditions. 
This concludes the proof of the lemma. 
\end{proof}

\subsection{Proof of Theorem \ref{theorem: highly connected morphisms}} \label{subsection: formal proof of theorem}
Consider the composition 
\begin{equation} \label{equation: augmentation composition}
|\mb{D}^{\partial, \mb{b}, k}_{\bullet, \bullet}| \stackrel{\simeq} \longrightarrow |\mb{D}^{\partial, \mb{b}, k-1}_{\bullet}| \stackrel{\simeq} \longrightarrow |\mb{X}^{\partial, \mb{b}, k-1}_{\bullet}|,
\end{equation}
where the first map is the one induced by the augmentation (which is a homotopy equivalence by Theorem \ref{theorem: contractible surgery data for morphisms}) and the second is a homotopy equivalence by Proposition \ref{proposition: X-D equivalence}. 
Observe that if $x = (a, \varepsilon, (W, \ell_{W})) \in \mb{D}^{\partial, \mb{b}, k-1}_{p}$ is contained in the subspace $\mb{D}^{\partial, \mb{b}, k}_{p} \subset \mb{D}^{\partial, \mb{b}, k-1}_{p}$, then the \textit{empty surgery data} determined by setting $\Lambda = \emptyset$ yields a valid element of $\mb{Z}_{0}(x)$. 
Thus, we obtain a simplicial map 
$\mb{D}^{\partial, \mb{b}, k}_{\bullet} \; \longrightarrow \; \mb{D}^{\partial, \mb{b}, k}_{\bullet, 0}$, $x \mapsto (x, \emptyset),$
and an embedding  
 \begin{equation} \label{equation: empty surgery inclusion}
 i_{\emptyset}: |\mb{D}^{\partial, \mb{b}, k}_{\bullet}| \; \hookrightarrow \; |\mb{D}^{\partial, \mb{b}, k}_{\bullet, 0}| \;  \hookrightarrow \;  |\mb{D}^{\partial, \mb{b}, k}_{\bullet, \bullet}|.
 \end{equation}
where the second embedding in the composition is given by the inclusion of $0$-simplices.
The following proposition is proven by implementing the same construction from
\cite[p. 301]{GRW 14}, using Proposition \ref{proposition: morphism
  surgery move} and Lemma \ref{lemma: surgery implementation lemma}.
\begin{proposition} \label{proposition: surgery move homotopy}
There exists a homotopy 
$
\mathcal{F}: [0, 1]\times|\mb{D}^{\partial, \mb{b}, k}_{\bullet, \bullet}| \longrightarrow |\mb{X}^{\partial, \mb{b}, k-1}_{\bullet}|
$
with the following properties:
\begin{enumerate} \itemsep.2cm
\item[(i)] the map $\mathcal{F}(0, \underline{\hspace{.2cm}}): |\mb{D}^{\partial, \mb{b}, k}_{\bullet, \bullet}| \longrightarrow |\mb{X}^{\partial, \mb{b}, k-1}_{\bullet}|$ is the weak homotopy equivalence from (\ref{equation: augmentation composition}); 
\item[(ii)] for all $t \in [0,1]$ the maps
\begin{itemize} \itemsep.2cm
\item $\mathcal{F}(1, \underline{\hspace{.3cm}}): |\mb{D}^{\partial, \mb{b}, k}_{\bullet, \bullet}| \longrightarrow |\mb{X}^{\partial, \mb{b}, k-1}_{\bullet}|$,
\item $\mathcal{F}(t, \underline{\hspace{.2cm}})\circ i_{\emptyset}: |\mb{D}^{\partial, \mb{b}, k}_{\bullet}| \longrightarrow |\mb{X}^{\partial, \mb{b}, k-1}_{\bullet}|$,
\end{itemize}
\end{enumerate}
both factor through the inclusion $|\mb{X}^{\partial, \mb{b}, k}_{\bullet}| \hookrightarrow |\mb{X}^{\partial, \mb{b}, k-1}_{\bullet}|$.
\end{proposition}
Consider the diagram,
 \begin{equation} \label{equation: homotopy commutative diagram}
 \xymatrix{
 |\mb{D}^{\partial, \mb{b}, k}_{\bullet}| \ar[rrr]^{\simeq} \ar[d]^{i_{\emptyset}} &&& |\mb{X}^{\partial, \mb{b}, k}_{\bullet}| \ar@{^{(}->}[d] \\
 |\mb{D}^{\partial, \mb{b}, k}_{\bullet, \bullet}| \ar[urrr]^{\mathcal{F}(1, \underline{\hspace{.3cm}})} \ar[rrr]^{\mathcal{F}(0, \underline{\hspace{.3cm}})}_{\simeq} &&& |\mb{X}^{\partial, \mb{b}, k-1}_{\bullet}|.
 }
 \end{equation}
 By Proposition \ref{proposition: surgery move homotopy},
 the outer square and the top triangle commute exactly, while the bottom triangle commutes up to the homotopy $\mathcal{F}$. 
 Passing to homotopy groups, it follows that the vertical maps are also weak homotopy equivalences. 
 Theorem \ref{theorem: highly connected morphisms} then follows by applying the weak homotopy equivalences $B\Cob^{\partial, \mb{b}, k-1}_{\theta} \simeq |\mb{X}^{\partial, \mb{b}, k-1}_{\bullet}|$ and $B\Cob^{\partial, \mb{b}, k}_{\theta} \simeq |\mb{X}^{\partial, \mb{b}, k}_{\bullet}|$.

\section{Surgery on the Boundary}    \label{section: surgery on the boundary}
In this section we prove that the inclusions 
$B\Cob^{\partial, \mb{b}}_{\theta} \hookrightarrow B\Cob^{\partial, D}_{\theta} \hookrightarrow B\Cob_{\theta}^{\partial}$ are weak homotopy equivalences in the when $\theta: B \longrightarrow BO(2n+1)$ is such that $B$ is $(n-1)$-connected. 
This is done in three stages, with each stage carried out in its own subsection.  
The key result that lets us do this is the lemma stated below. 
\begin{lemma} \label{lemma: restriction to boundary is a fibration}
Let $U \subset \R^{m}_{+}$ be an open subset. 
Then the restriction map 
$$\widehat{\partial}: \bPsi^{\partial}_{\theta}(U) \; \longrightarrow \; \bPsi_{\theta_{d-1}}(\partial U), \quad (M, \ell) \mapsto (\partial M, \ell|_{\partial M}),$$
is a Serre-fibration.
\end{lemma}
\begin{proof}
Let $U \subset \R^{m}_{+}$ be an open subset and let 
$U_{\infty} = U\cup_{\partial U}(\partial U\times(-\infty, 0]) \subset \R^{m}$
be the subspace obtained by attaching the cylinder $\partial U\times(-\infty, 0]$ along the boundary. 
In \cite{RW 11} Randal-Williams proved that the sheaf $\bPsi_{\theta}(\underline{\hspace{.3cm}})$ is \textit{micro-flexible} and that it is \textit{$\Diff(\underline{\hspace{.3cm}})$-equivariant}, and thus it satisfies the \textit{parametric h-principle} on open manifolds. 
Let $\bPsi_{\theta}(\partial U\times(-\infty, 0])$ be the limiting space $\displaystyle{\lim_{V}}\bPsi_{\theta}(V)$, with limit taken over all open subsets $V \subset \R^{m}$ containing the closed subset $\partial U\times(-\infty, 0]$.
It follows as a consequence of the parametric $h$-principle that the restriction map 
$$
r: \bPsi_{\theta}(U_{\infty}) \; \longrightarrow \; \bPsi_{\theta}(\partial U\times(-\infty, 0]), \quad (M, \ell) \mapsto \left(M|_{\partial U\times(-\infty, 0]}, \; \ell|_{\partial U\times(-\infty, 0]}\right),
$$
is a Serre-fibration.
Consider now the commutative square 
\begin{equation} \label{equation: restriction pull-back}
\xymatrix{
\bPsi^{\partial}_{\theta}(U)  \ar[rr]^{T_{U}} \ar[d]^{\widehat{\partial}} && \bPsi_{\theta}(U_{\infty}) \ar[d]^{r} \\ 
\bPsi_{\theta_{d-1}}(\partial U) \ar[rr] && \bPsi_{\theta}(\partial U\times(-\infty, 0]),
}
\end{equation}
where the bottom horizontal map is given by sending an element $(M, \ell_{M}) \in \bPsi_{\theta_{d-1}}(\partial U)$ to the product $(M\times\R, \ell_{M\times\R}) \in  \bPsi_{\theta}(\partial U\times\R)$ and then restricting to its germ at $\partial U\times(-\infty, 0]$.
It is easy to verify that this square (\ref{equation: restriction pull-back}) is a pull-back square. 
Since the map $r$ is a Serre-fibration it follows then that the boundary restriction $\widehat{\partial}$ is a Serre-fibration, since the condition of being a Serre-fibration is hereditary under pull-backs.
This concludes the proof of the lemma. 
\end{proof}

\subsection{The equivalence $B\Cob_{\theta}^{\partial, D} \simeq B\Cob^{\partial}_{\theta}$}
In this subsection we prove the following proposition,
the proof of which is a direct consequence of Lemma \ref{lemma: restriction to boundary is a fibration}. 
It holds for any choice of tangential structure $\theta: B \longrightarrow BO(d)$. 
\begin{proposition}
The inclusion $B\Cob_{\theta}^{\partial, D} \hookrightarrow B\Cob_{\theta}^{\partial}$ is a weak homotopy equivalence. 
\end{proposition}
\begin{proof}
Let $\bpsi^{\partial, D}_{\theta}(\infty, 1) \subset \bpsi^{\partial}_{\theta}(\infty, 1)$ be the subspace that consists of those $(W, \ell)$ such that 
$$
\partial W\cap\left(\R\times(-1, 0]\times(-1, 1)^{\infty-2}\right) = \R\times D,
$$
and such that the restriction of $\ell$ to $\R\times D$ agrees with $\ell_{\R\times D}$. 
By following the same argument in \cite[Theorem 3.10]{GRW 09} verbatim, we obtain the weak homotopy equivalence $B\Cob_{\theta}^{\partial, D} \simeq \bpsi^{\partial, D}_{\theta}(\infty, 1)$.
Now consider the space $\bpsi_{\theta_{d-1}}(\infty, 1)$. 
We define $\bpsi^{D}_{\theta_{d-1}}(\infty, 1) \subset \bpsi_{\theta_{d-1}}(\infty, 1)$ similarly.
The commutative square,
$$
\xymatrix{
\bpsi^{\partial, D}_{\theta}(\infty, 1) \ar@{^{(}->}[rr] \ar[d]^{\widehat{\partial}} && \bpsi^{\partial}_{\theta}(\infty, 1) \ar[d]^{\widehat{\partial}} \\
\bpsi^{D}_{\theta_{d-1}}(\infty, 1) \ar@{^{(}->}[rr]^{\simeq} && \bpsi_{\theta_{d-1}}(\infty, 1)
}
$$
is a pull-back square, where the vertical maps are boundary restriction maps. 
By Lemma \ref{lemma: restriction to boundary is a fibration} the right-hand vertical boundary map is a Serre-fibration and thus the above diagram is \textit{homotopy cartesian}. 
By \cite[Proposition 2.16]{GRW 14} the bottom-horizontal map is a weak homotopy equivalence.  
It follows that the top-horizontal map is a weak homotopy equivalence as well. 
Combining this with the weak homotopy equivalences
$B\Cob_{\theta}^{\partial, D} \simeq \bpsi^{\partial, D}_{\theta}(\infty, 1)$ and $B\Cob_{\theta}^{\partial} \simeq \bpsi^{\partial}_{\theta}(\infty, 1)$
completes the proof of the proposition. 
\end{proof}

\subsection{Surgery on the boundary morphisms}
We begin by defining a sequence of new intermediate subcategories of
$\Cob^{\partial}_{\theta}$ that are contained in $\Cob^{\partial, \mb{b}}_{\theta}$.
\begin{defn}
For $k \in \Z_{\geq -1}$, $\Cob_{\theta}^{\partial, k} \subset
\Cob_{\theta}^{\partial}$ is defined to be the topological
subcategory with the same space of objects.  It has as its morphisms
those tuples $(t, (W, \ell))$ such that the pair $\left((\partial W)|_{[0, t]}, (\partial W)|_{t}\right)$ is $k$-connected.
\end{defn}
To go with the categories $\Cob_{\theta}^{\partial, k}$, there are
semi-simplicial spaces $\mb{D}^{\partial, k}_{\bullet}$ and $\mb{X}^{\partial, k}_{\bullet}$.  
  The object of this section is to prove the following theorem.
\begin{theorem} \label{theorem: surgery on boundary morphisms}
Let $k \leq n-1$.  Then the inclusion $\Cob^{\partial, k}_{\theta}
\hookrightarrow \Cob^{\partial, k-1}_{\theta}$ induces the weak
homotopy equivalence $B\Cob^{\partial, k}_{\theta} \simeq
B\Cob^{\partial, k-1}_{\theta}$.
\end{theorem}
The proof of the above theorem is almost the same as the proof of
\cite[Theorem 3.1]{GRW 14}.  It uses the same constructions and
requires adjustments in just a few places.  We provide an outline.  

We begin by fixing some notation. 
Set $d := 2n+1$. 
Let 
$$V \subset \bar{V} \subset \R^{k+1}\times\R^{d-k+1}_{+}$$ 
be the subspaces 
$$
V \; = \; (-2, 0)\times\R^{k}\times\R^{d-k+1}_{+} \quad \text{and} \quad \bar{V} \; = \; [-2, 0]\times\R^{k}\times\R^{d-k+1}_{+}.
$$ We then denote by $\partial_{1}\bar{V} \subset \bar{V}$ the subset
of the boundary given by
$$
\partial_{1}\bar{V} = [-2, 0]\times\R^{k}\times\partial\R^{d-k+1}_{+}.
$$
Let 
$h: \bar{V} \longrightarrow [-2, 0]$ 
denote the projection onto the first coordinate, which we call the \textit{height function}. 
As before $x_{0}: \R\times\R^{\infty}_{+} \longrightarrow \R$ is the projection onto the first coordinate. 
We let 
$$\partial_{-}D^{k+1} \subset \partial D^{k+1} \subset \R^{k+1}$$
denote the lower hemisphere. 

As in Section \ref{subsection: morphism surgery data} we use the
notation $[p]^{\vee} = \triangle([p], [1])$ for $[p] \in \triangle$.
The elements of $[p]^{\vee}$ are in bijection with $\{0, \dots,
p+1\}$, using the convention that $\varphi: [p] \longrightarrow [1]$
corresponds to the number $i$ with $\varphi^{-1}(1) = \{i, i+1, \dots,
p\}$.  
The definition below should be compared to
  \cite[Definition 3.3]{GRW 14}.
\begin{defn}
Fix once and for all an infinite set $\Omega$.
Let $x = (a, \varepsilon, (W, \ell_{W})) \in \mb{D}^{\partial,
  k-1}_{p}$.  We define $\mb{Z}_{0}(x)$ to be the set of triples
$(\Lambda, \delta, e)$, where $\Lambda \subset \Omega$ is a finite
subset, $\delta: \Lambda \longrightarrow [p]^{\vee}$ is a function,
and
$
e: \Lambda\times\bar{V} \hookrightarrow \R\times[0, 1)\times(-1, 1)^{\infty-1}  
$
is an embedding satisfying the following conditions:
\begin{enumerate} \itemsep.2cm
\item[(i)] \textcolor{black}{the embedding} $e$ satisfies
  $e^{-1}(\R\times\{0\}\times(-1, 1)^{\infty-1}) =
  \partial_{1}\bar{V}$;
\item[(ii)] on every subset $(x_{0}\circ e|_{\{\lambda\}\times\bar{V}})^{-1}(a_{k}-\varepsilon_{k}, a_{k}+\varepsilon_{k}) \; \subset \; \{\lambda\}\times\bar{V}$, the height function $x_{0}\circ e$ coincides with the height function $h$ up to an affine translation; 
\item[(iii)] the embedding $e$ sends $\Lambda\times h^{-1}(0)$ into $x^{-1}_{0}(a_{p}+\varepsilon_{p}, \infty)$; 
\item[(iv)] for $i > 0$, $e$ sends $\delta^{-1}(i)\times h^{-1}(-\tfrac{3}{2})$ into $x_{0}^{-1}(a_{i-1}+\varepsilon_{i-1}, \infty)$; 
\item[(v)] the embedding $e$ sends $\Lambda\times
  h^{-1}(-2)$ into $x_{0}^{-1}(-\infty, a_{0}-\varepsilon_{0})$;
\item[(vi)] the embedding $e$ satisfies 
$$\begin{aligned}
e^{-1}(W) \; &= \; \Lambda\times\partial_{-}D^{k+1}\times\R^{d-k+1}_{+},\\
e^{-1}(\partial W) \; &= \; \Lambda\times\partial_{-}D^{k+1}\times\partial\R^{d-k+1}_{+};
\end{aligned}$$
\item[(vii)] writing $D_{i} = e(\delta^{-1}(i)\times\partial_{-}D^{k+1}\times\{0\})$ for $i \in [p]^{\vee}$, the pair 
$$
\left((\partial W)|_{[a_{i-1}, a_{i}]}, \; (\partial W)|_{a_{i}}\cup D_{i}|_{[a_{i-1}, a_{i}]}\right)
$$
is $k$-connected for each $i \in \{0, \dots, p\}$. 
\end{enumerate}
Now let $q \in \Z_{\geq 0}$.  We define $\mb{Z}_{q}(x)$ to be the set
of triples $(\Lambda, \delta, e)$ where $\Lambda \subset \Omega$ is a
finite set, $\delta: \Lambda \longrightarrow [p]^{\vee}\times[q]$ is a
function, and
$
e: \Lambda\times\bar{V} \longrightarrow \R\times[0, 1)\times(-1, 1)^{\infty-1}
$
  is an embedding, subject to the condition that for each $j \in [q]$, the restriction of $e$ to
  $\delta^{-1}([p]^{\vee}\times\{j\})\times\bar{V}$ 
  defines an element of $\mb{Z}_{0}(x)$. 
We shall write 
$$\Lambda_{i, j} = \delta^{-1}(i ,j) \quad \text{and} \quad e_{i, j} = e|_{\Lambda_{i, j}\times\bar{V}}$$ 
for $i \in [p]^{\vee}$ and $j \in [q]$. 
The bi-semi-simplicial space $\mb{D}^{\partial, k}_{\bullet, \bullet}$ is defined by setting 
$$
\mb{D}^{\partial, k}_{p, q} \; = \; \{(x, y)\; | \; x \in \mb{D}^{\partial, k-1}_{p}, \; \; y \in \mb{Z}_{q}(x) \; \}
$$
and it is topologized as a subspace of the product 
$$
\mb{D}^{\partial, k-1}_{p}\times\bigg(\coprod_{\Lambda \subset \Omega}C^{\infty}(\Lambda\times\bar{V}, \R\times\R^{\infty}_{+})\bigg)^{(p+2)(q+1)}.
$$
The projection maps 
$\mb{D}^{\partial, k}_{p, q} \longrightarrow \mb{D}^{\partial, k-1}_{p}$, $(x, y) \mapsto x,$
yield an augmented bi-semi-simplicial space
$\mb{D}^{\partial, k}_{\bullet, \bullet} \longrightarrow \mb{D}^{\partial, k}_{\bullet, -1},$
with $\mb{D}^{\partial, k}_{\bullet, -1} :=  \mb{D}^{\partial, k-1}_{\bullet}$.
\end{defn}
The following lemma is proven in the same way as Theorem \ref{theorem: contractible surgery data for morphisms} and \cite[Theorem 3.4]{GRW 14}.
\begin{lemma}
Suppose that $k \leq n-1$. 
Then the augmentation map induces a weak homotopy equivalence 
$|\mb{D}^{\partial, k}_{\bullet, \bullet}| \simeq |\mb{D}^{\partial, k}_{\bullet, -1}| = |\mb{D}^{\partial, k-1}_{\bullet}|.$
\end{lemma}
The next ingredient in the proof of Theorem \ref{theorem: surgery on boundary morphisms} is the construction of a new surgery move, analogous to the one from Proposition \ref{proposition: morphism surgery move}.
We will need the following proposition.
\textcolor{black}{This should be compared to \cite[Proposition 3.6]{GRW 14}.}
\begin{proposition} \label{proposition: boundary morphism surgery move}
Let $d = 2n+1 \geq 3$ and $k \leq n-1$. 
There exists a one-parameter family 
$\mathcal{P}_{t} \in \bPsi^{\partial}_{d}(V)$ defined for $t \in [0, 1]$
with the following properties:
\begin{enumerate} \itemsep.2cm
\item[(i)] $\mathcal{P}_{0} = \Int(\partial_{-}D^{k+1})\times\R^{d-k+1}_{+}$, where recall that $\partial_{-}D^{k+1} \subset \partial D^{k+1} \subset \R^{k+1}$ is the lower hemisphere. 
\item[(ii)] The associated family of boundaries $\partial\mathcal{P}_{t} \in \bPsi_{d-1}(\partial V)$ satisfies the following conditions:
\begin{enumerate} \itemsep.2cm
\item[(a)] The height function, i.e.\ the restriction of $h|_{\partial V}: \partial V \longrightarrow (-2, 0)$ to $\partial\mathcal{P}_{t} \subset V$, has isolated critical values. 
\item[(b)] Let $B^{d-k}_{3}(0) \subset \partial\R^{d-k+1}_{+} = \R^{d-k}$ denote the ball of radius $3$. 
Independently of $t \in [0, 1]$ we have, 
$
\partial\mathcal{P}_{t}\setminus(\R^{k+1}\times B^{d-k}_{3}(0)) \; = \; \Int(\partial_{-}D^{k+1})\times(\R^{d-k}\setminus B^{d-k}_{3}(0)).
$
For ease of notation we denote $\widehat{\mathcal{P}}_{t} := \partial\mathcal{P}_{t}\setminus(\R^{k+1}\times B^{d-k}_{3}(0))$.
\item[(c)] For all $t \in [0, 1]$ and each pair of regular values $-2 < a < b < 0$ of the height function, the pair 
$
\left((\partial\mathcal{P}_{t})|_{[a, b]}, \; (\partial\mathcal{P}_{t})|_{b}\cup\widehat{\mathcal{P}}_{t}|_{[a, b]}\right)
$
is $k$-connected. 
\item[(d)] For each pair of regular values $-2 < a < b < 0$ of the height function, the pair 
$$ ((\partial\mathcal{P}_{1})|_{[a, b]}, \;
  (\partial\mathcal{P}_{1})|_{b})
$$
is $k$-connected. 
\end{enumerate}
\end{enumerate}
Furthermore, if $\mathcal{P}_{0}$ is equipped with a $\theta$-structure $\ell$, the above family can be upgraded to a one parameter family $(\mathcal{P}_{t}, \ell_{t}) \in \bPsi^{\partial}_{\theta}(V)$ with $\ell_{0} = \ell,$ such that the path of boundaries 
$$(\partial\mathcal{P}_{t}, \partial\ell_{t}) \in \bPsi_{\theta_{d-1}}(\partial V)$$ 
is constant as a $\theta$-manifold near $\widehat{\mathcal{P}}_{t}$, where $\partial\ell_{t}$ denotes the restriction of $\ell_{t}$ to the boundary $\partial\mathcal{P}_{t}$. 
 \end{proposition} 
\begin{proof}
We begin setting $\mathcal{P}_{0} = \Int(\partial_{-}D^{k+1})\times\R^{d-k+1}_{+}$. 
We need to extend $\mathcal{P}_{0}$ to a one parameter family $\mathcal{P}_{t}$ with $t \in [0, 1]$. 
We have $\partial\mathcal{P}_{0} = \Int(\partial_{-}D^{k+1})\times\R^{d-k} \subset \bPsi_{d-1}(\partial V)$.
By \cite[Proposition 3.6]{GRW 14}, the boundary $\partial\mathcal{P}_{0}$ can be extended to a one-parameter family $\partial\mathcal{P}_{t}$ that satisfies condition (ii) parts (a), (b), (c), and (d) from the statement of this proposition.
We have a path $\partial\mathcal{P}_{t} \in \bPsi_{d-1}(\partial V)$  such that $\partial\mathcal{P}_{0}$ admits an extension to an element $\mathcal{P}_{0} \in \bPsi^{\partial}_{d}(V)$.
 This data yields the diagram 
 $$
 \xymatrix{
 \{0\} \ar@{^{(}->}[d] \ar[rr]^{\mathcal{P}_{0}} && \bPsi^{\partial}_{d}(V) \ar[d] \\
 [0, 1] \ar@{-->}[urr] \ar[rr]^{\partial\mathcal{P}_{(\underline{\hspace{.2cm}})}} && \bPsi_{d-1}(\partial V).
 }
 $$
Using Lemma \ref{lemma: restriction to boundary is a fibration}, the right-hand vertical map is a Serre-fibration and thus the dotted arrow indeed exists. 
Thus we obtain an extension of the family $\partial\mathcal{P}_{t}$ from \cite[Proposition 3.6]{GRW 14} to a family $\mathcal{P}_{t} \in \bPsi^{\partial}_{d}(V)$. 
The argument in the case where a tangential structure is present is exactly the same since the key result Lemma \ref{lemma: restriction to boundary is a fibration} applies for arbitrary choice of $\theta$. 
 This concludes the proof of the proposition.
\end{proof}
To finish the proof of Theorem \ref{theorem: surgery on boundary morphisms}, we proceed in the same way as in Section \ref{subsection: formal proof of theorem}.
As in the proof of Proposition \ref{proposition: surgery move homotopy}, we use the surgery move from Proposition \ref{proposition: boundary morphism surgery move} to construct a homotopy 
\begin{equation} \label{equation: surgery homotopy}
\mathcal{F}: |\mb{D}^{\partial, k}_{\bullet, \bullet}|\times[0, 1] \longrightarrow |\mb{X}^{\partial, k-1}_{\bullet}|
\end{equation}
such that the map $\mathcal{F}(\underline{\hspace{.3cm}}, 0)$ equals the composition
$
|\mb{D}^{\partial, k}_{\bullet, \bullet}| \stackrel{\simeq} \longrightarrow |\mb{D}^{\partial, k-1}_{\bullet}| \stackrel{\simeq} \longrightarrow |\mb{X}^{\partial, k-1}_{\bullet}|
$
and $\mathcal{F}(\underline{\hspace{.3cm}}, 1)$ factors through the inclusion $|\mb{D}^{\partial, k}_{\bullet}| \hookrightarrow |\mb{D}^{\partial, k-1}_{\bullet}|$.
With such a homotopy constructed, we obtain the diagram 
$$
 \xymatrix{
 |\mb{D}^{\partial, k}_{\bullet}| \ar[rrr]^{\simeq} \ar@{^{(}->}[d]^{i_{\emptyset}} &&& |\mb{X}^{\partial, k}_{\bullet}| \ar@{^{(}->}[d] \\
 |\mb{D}^{\partial, k}_{\bullet, \bullet}| \ar[urrr]^{\mathcal{F}(1, \underline{\hspace{.3cm}})} \ar[rrr]^{\mathcal{F}(0, \underline{\hspace{.3cm}})}_{\simeq} &&& |\mb{X}^{\partial, k-1}_{\bullet}|,
 }
 $$
 analogous to (\ref{equation: homotopy commutative diagram}), with the top-triangle commutative and bottom-triangle homotopy commutative. 
 The left-vertical map is the embedding induced by the empty surgery data (this map is defined analogously to (\ref{equation: empty surgery inclusion})).
 The proof of Theorem \ref{theorem: surgery on boundary morphisms} then follows by the same argument in the proof of Theorem \ref{theorem: highly connected morphisms} at the end of Section \ref{section: surgery on morphisms}.
 Homotopy-commutativity of these triangles implies that the diagonal map $\mathcal{F}(\underline{\hspace{.3cm}}, 1):  |\mb{D}^{\partial, k}_{\bullet, \bullet}| \longrightarrow |\mb{X}^{\partial, k}_{\bullet}|$ is a weak homotopy equivalence. 
 This completes the proof of Theorem \ref{theorem: surgery on boundary morphisms}.
 
 \subsection{Surgery on boundary objects} 
We begin by filtering $\Cob^{\partial, n-1}_{\theta}$ by a sequence of subcategories that are contained in 
$\Cob_{\theta}^{\partial, \mb{b}}$. 
\begin{defn} \label{defn: filtration by highly connected boundary objects}
Let $k, l \in \Z_{\geq -1}$. 
We define $\Cob^{\partial, k, l}_{\theta} \subset \Cob^{\partial, k}_{\theta}$ to be the full-subcategory on those objects $(M, \ell)$ such that the boundary $\partial M$ is $l$-connected. 
\end{defn}
The main theorem of this subsection is the following.
\begin{theorem} \label{theorem: highly connected boundary objects}
Let $k \leq n-1$ and let
$l \leq \min\{k, n-2\}$. 
Let $\theta: B \longrightarrow BO(2n+1)$ be a tangential structure such that the space $B$ is $l$-connected. 
Then the inclusion $\Cob^{\partial, k, l}_{\theta} \hookrightarrow \Cob^{\partial, k, l-1}_{\theta}$ induces the weak homotopy equivalence
$B\Cob^{\partial, k, l}_{\theta} \simeq B\Cob^{\partial, k, l-1}_{\theta}$.
\end{theorem}

Notice that we have $\Cob^{\partial, n-1, n-2}_{\theta} = \Cob^{\partial, \mb{b}}_{\theta}$. 
Combining Theorem \ref{theorem: highly connected boundary objects} together with Theorem \ref{theorem: surgery on boundary morphisms} then yields the weak homotopy equivalence $B\Cob^{\partial, \mb{b}}_{\theta} \simeq B\Cob^{\partial}_{\theta}$ whenever the space $B$ is $(n-2)$-connected.
This is one of the weak homotopy equivalences from the statement of Theorem \ref{theorem: easy weak equivalences}. 
The proof of Theorem \ref{theorem: highly connected boundary objects} is almost identical to that of \cite[Theorem 4.1]{GRW 14}. 
Just a few adjustments are needed. 
The first ingredient that we need is a new surgery move.
We set $d := 2n+1$. 
The family of manifolds in the next proposition will be contained in the space $(-6, -2)\times\R^{d-l-1}_{+}\times\R^{l+1}$. 
We denote by 
$$\textcolor{black}{h: (-6, -2)\times\R^{d-l-1}_{+}\times\R^{l+1}
  \longrightarrow (-6, -2)}$$ the projection onto the first
coordinate.  We refer to the function $h$ and its
  restriction to various submanifolds as the \textit{height function}.
\begin{proposition} \label{proposition: surgery move for boundary objects}
There exists a one-parameter family 
$$\mathcal{P}_{t} \in \bPsi_{d}^{\partial}((-6, -2)\times\R^{d-l-1}_{+}\times\R^{l+1})$$
satisfying the following conditions:
\begin{enumerate}\itemsep.3cm
\item[(i)] $\mathcal{P}_{0} = (-6, -2)\times\R^{d-l-1}_{+}\times S^{l}$. 
\item[(ii)] The associated family of boundaries $\partial\mathcal{P}_{t} \in \bPsi_{d-1}((-6, -2)\times\R^{d-l-2}\times\R^{l+1})$ satisfies the following further conditions:
\begin{enumerate} \itemsep.2cm
\item[(a)] $\partial\mathcal{P}_{t} \subset (-6, -2)\times\R^{d-l-2}\times D^{l+1}$ for all $t \in [0, 1]$; 
\item[(b)] for all $t \in [0, 1]$, the manifold
  $\partial\mathcal{P}_{t}$ agrees with $\partial\mathcal{P}_{0} =
  (-6, -2)\times\R^{d-l-2}\times S^{l}$ outside of the region $(-5,
  -2)\times B^{d-l-2}_{2}(0)\times D^{l+1}$;
\item[(c)] for all $t$ and each pair of regular values $a < b$ of the
  height function $$h: \partial\mathcal{P}_{t} \longrightarrow (-6,
  -2),$$ the pair $((\partial\mathcal{P}_{t})|_{[a, b]}, \;
  (\partial\mathcal{P}_{t})|_{b})$ is $(d - l - 3)$-connected;
\item[(d)] the only critical value of the height function $h: \partial\mathcal{P}_{1} \longrightarrow (-6, -2)$ is $-4$, and for $a \in (-4, -2)$, the manifold $(\partial\mathcal{P}_{1})|_{a}$ is obtained from $(\partial\mathcal{P}_{0})|_{a} = \{a\}\times\R^{d-l-2}\times S^{l}$ by $l$-surgery along the standard embedding. 
\end{enumerate}
 \end{enumerate}
Furthermore, if the initial manifold $\mathcal{P}_{0} = (-6, -2)\times\R^{d-l-1}_{+}\times S^{l}$ is equipped with a $\theta$-structure $\ell$, then the family can be upgraded to a one-parameter family 
$$(\mathcal{P}_{t}, \ell_{t}) \in \bPsi^{\partial}_{\theta}((-6, -2)\times\R^{d-l-1}_{+}\times\R^{l+1})$$ 
with $\ell_{0} = \ell$, such that the family $(\partial\mathcal{P}_{t}, \partial\ell_{t})$ is constant outside of the set $(-5, -2)\times B^{d-l-2}_{2}(0)\times D^{l+1}$. 
\end{proposition}
\begin{proof}
This proposition is proven in same way as Proposition \ref{proposition: boundary morphism surgery move} by applying Lemma \ref{lemma: restriction to boundary is a fibration} and the one-parameter family from \cite[Proposition 4.2]{GRW 14}.
\end{proof}
The next ingredient in the proof of Theorem \ref{theorem: highly connected boundary objects} is the construction of an augmented bi-semi-simplicial space $\mb{D}^{\partial, k, l}_{\bullet, \bullet} \longrightarrow \mb{D}^{\partial, k, l}_{\bullet, -1}$
with $ \mb{D}^{\partial, k, l}_{\bullet, -1} \; = \; \mb{D}^{\partial, k, l-1}_{\bullet}$, analogous to \cite[Definition 4.3]{GRW 14}.
For the definition that follows, we choose once and for all, smoothly in the data $(a_{i}, \varepsilon_{i}, a_{p}, \varepsilon_{p})$, increasing diffeomorphisms 
\begin{equation} \label{equation: reparametrisation of height}
\varphi = \varphi(a_{i}, \varepsilon_{i}, a_{p}, \varepsilon_{p}): (-6, -2) \stackrel{\cong} \longrightarrow (a_{i}-\varepsilon_{i}, a_{p}-\varepsilon_{p}). 
\end{equation}
\begin{defn}
Fix once and for all an infinite set $\Omega$.
Let $x = (a, \varepsilon, (W, \ell_{W})) \in \mb{D}^{\partial, k, l-1}_{p}$. 
With $a = (a_{0}, \dots, a_{p})$, for each $i$ we write $M_{i} = W|_{a_{i}}$.
Define a set $\mb{Y}_{q}(x)$ to consist of tuples $(\Lambda, \delta, e, \ell)$, where:
\begin{itemize} \itemsep.2cm
\item $\Lambda \subset \Omega$ is a finite subset;
\item $\delta: \Lambda \longrightarrow [p]\times[q]$ is a function;
\item  
$
e: \Lambda\times(-6, -2)\times\R^{d-l-1}_{+}\times S^{l} \longrightarrow \R\times [0, 1)\times(-1, 1)^{\infty-1}
$
is an embedding that satisfies 
$$
e^{-1}(\R\times\{0\}\times(-1, 1)^{\infty-1}) =  \Lambda\times(-6, -2)\times\R^{d-l-2}\times S^{l}.
$$
\item $\ell$ is a $\theta$-structure on $\Lambda\times(-6, -2)\times\R^{d-l-1}_{+}\times S^{l}$.
\end{itemize}
We write $\Lambda_{i, j} = \delta^{-1}(i, j)$ and 
$$
e_{i, j}: \Lambda_{i, j}\times(a_{i} - \varepsilon_{i}, a_{i}+\varepsilon_{i})\times\R^{d-l-1}_{+}\times S^{l} \longrightarrow \R\times[0, 1)\times(-1, 1)^{\infty-1}
$$ for the embedding obtained by restricting $e$ and reparametrising
  using (\ref{equation: reparametrisation of height}).
  \textcolor{black}{We} let $\ell_{i, j}$ denote the restriction of
  $\ell$ to $\Lambda_{i, j}\times(-6, -2)\times\R^{d-l-1}_{+}\times
  S^{l}$.  This data is required to satisfy the following conditions:
\begin{enumerate} \itemsep.2cm
\item[(i)] The embedding $e$ satisfies
$$\begin{aligned}
e^{-1}(W) &= \Lambda\times(-6, -2)\times\R^{d-l-1}_{+}\times S^{l}, \\
e^{-1}(\partial W) &= \Lambda\times(-6, -2)\times\R^{d-l-2}\times S^{l}.
\end{aligned}$$
\item[(ii)] Let $(i, j) \in [p]\times[q]$. 
For $t \in (a_{i}-\varepsilon_{i}, a_{i}+\varepsilon_{i})$
we have 
$$(x_{0}\circ e_{i, j})^{-1}(t) = \Lambda_{i, j}\times\{t\}\times\R^{d-l-1}_{+}\times S^{l}.$$
In particular, this implies that 
$$\begin{aligned}
e_{i, j}(\Lambda_{i, j}\times\{a_{i}\}\times \R^{d-l-1}_{+}\times S^{l}) \; &\subset \; M_{i} = W|_{a_{i}},\\
e_{i, j}(\Lambda_{i, j}\times\{a_{i}\}\times\R^{d-l-2}\times S^{l}) \; &\subset \; \partial M_{i} = (\partial W)|_{a_{i}}.
\end{aligned}$$
\item[(iii)] The restriction of $\ell_{W}$ to the image $e(\Lambda\times(-6, -2)\times\R^{d-l-1}_{+}\times S^{l})$ agrees with $\ell$.
\end{enumerate}
For $(i, j) \in [p]\times[q]$, let $\partial M'_{i, j}$ denote the manifold obtained by doing surgery on $\partial M_{i,j}$ with respect to the embedding
$$
 \Lambda_{i, j}\times\{a_{i}\}\times\R^{d-l-2}\times S^{l} \longrightarrow \partial M_{i}
$$
given by restricting $e_{i, j}$ (see condition (ii)). 
We insist that
\begin{enumerate}
\item[(iv)] for each $(i, j) \in [p]\times[q]$ the resulting manifold $\partial M'_{i, j}$ is $l$-connected.  
\end{enumerate}
For each $x$, the assignment $q \mapsto \mb{Y}_{q}(x)$ defines a semi-simplicial space $\mb{Y}_{\bullet}(x)$.
The bi-semi-simplicial space $\mb{D}^{\partial, k, l}_{\bullet, \bullet}$ is defined by setting
$
\mb{D}^{\partial, k, l}_{p, q} \; = \; \{(x, y) \; | \; x \in \mb{D}^{\partial, k, l-1}_{p}, \quad  y \in \mb{Y}_{q}(x) \; \}.
$ 
The projections 
$\mb{D}^{\partial, k, l}_{p, q} \longrightarrow \mb{D}^{\partial, k, l-1}_{p}$, $(x, y) \mapsto x$
yield an augmented bi-semi-simplicial space
$\mb{D}^{\partial, k, l}_{\bullet, \bullet} \longrightarrow \mb{D}^{\partial, k, l}_{\bullet, -1}$ with $\mb{D}^{\partial, k, l}_{\bullet, -1} := \mb{D}^{\partial, k, l-1}_{\bullet}$.
\end{defn}
\textcolor{black}{The} following lemma which is proven in the same way
as \cite[Theorem 4.5]{GRW 14}.
\begin{lemma} \label{lemma: boundary objects augmentation}
Under the assumptions of Theorem \ref{theorem: highly connected boundary objects}, the maps 
$$
|\mb{D}^{\partial, k, l}_{\bullet, 0}| \longrightarrow |\mb{D}^{\partial, k, l}_{\bullet, \bullet}| \longrightarrow |\mb{D}^{\partial, k, l}_{\bullet, -1}| = |\mb{D}^{\partial, k, l-1}_{\bullet}|
$$ are weak homotopy equivalences, where the first map is induced by
the inclusion of zero-simplices and the second is induced by the
augmentation map.
\end{lemma}
With the surgery move defined and Lemma \ref{lemma: boundary objects
  augmentation} established, the proof of Theorem \ref{theorem: highly
  connected boundary objects} proceeds in the same way as the proof of
Theorem \ref{theorem: disk object equivalence}.  We use the surgery
move from Proposition \ref{proposition: surgery move for boundary
  objects} to define a homotopy
$$
\mathcal{F}: |\mb{D}^{\partial, k, l}_{\bullet, 0}|\times [0, 1] \longrightarrow |\mb{X}^{\partial, k, l-1}_{\bullet}|
$$
such that:
\begin{itemize} \itemsep.2cm
\item $\mathcal{F}(\underline{\hspace{.2cm}}, 0)$ is equal to the composition 
$
|\mb{D}^{\partial, k, l}_{\bullet, 0}| \stackrel{\simeq} \longrightarrow |\mb{D}^{\partial, k, l}_{\bullet, \bullet}| \stackrel{\simeq} \longrightarrow |\mb{D}^{\partial, k, l-1}_{\bullet}| \stackrel{\simeq} \longrightarrow |\mb{X}^{\partial, k, l-1}_{\bullet}|\textcolor{black}{;}
$
\item $\mathcal{F}(\underline{\hspace{.2cm}}, 1)$ factors through the inclusion $|\mb{X}^{\partial, k, l}_{\bullet}| \hookrightarrow |\mb{X}^{\partial, k, l-1}_{\bullet}|$.
\end{itemize}
This homotopy is constructed in the same way as in Section \ref{subsection: disk equivalence} (see also \cite[Section 4.1]{GRW 14}).
The empty surgery data yields an embedding 
$
|\mb{D}^{\partial, k, l}_{\bullet}| \hookrightarrow |\mb{D}^{\partial, k, l}_{\bullet, 0}|.
$
This embedding together with the homotopy $\mathcal{F}$ 
yields the diagram 
$$
\xymatrix{
|\mb{D}^{\partial, k, l}_{\bullet}| \ar@{^{(}->}[rr] \ar[d]_{\simeq} && |\mb{D}^{\partial, k, l}_{\bullet, 0}| \ar[d]^{\mathcal{F}(\underline{\hspace{.2cm}}, 0)}_{\simeq} \ar[dll]_{\mathcal{F}(\underline{\hspace{.2cm}}, 1)} \\
|\mb{X}^{\partial, k, l}_{\bullet}|  \ar@{^{(}->}[rr] &&  |\mb{X}^{\partial, k, l-1}_{\bullet}|
}
$$
with the upper-triangle commutative and the lower-triangle homotopy commutative. 
Homotopy commutativity implies that the bottom inclusion $|\mb{X}^{\partial, k, l}_{\bullet}|  \hookrightarrow |\mb{X}^{\partial, k, l-1}_{\bullet}|$ is a weak homotopy equivalence. 
This completes the proof of Theorem \ref{theorem: highly connected boundary objects}.

\section{Embeddings and Disjunction} \label{section: A generalization of the Whitney trick}
\subsection{The half Whitney-trick} \label{subsection: the half Whitney trick}
For what follows, let $(M; \partial_{0}M, \partial_{1}M)$ be a $2n$-dimensional manifold triad and let
$(P; \partial_{0}P, \partial_{1}P)$ and $(Q; \partial_{0}Q, \partial_{1}Q)$ be manifold triads of dimension $n$.
We consider embeddings 
\begin{equation} \label{equation: embeddings n-case}
f: (P, \partial_{0}P) \longrightarrow (M, \partial_{0}M) \quad \text{and} \quad g: (Q, \partial_{0}Q) \longrightarrow (M, \partial_{0}M).
\end{equation}
Note, that for the above embeddings we impose no conditions on the restrictions $f|_{\partial_{1}P}$ and $g|_{\partial_{1}Q}$. 
In particular we do not require that $f$ and $g$ send $\partial_{1}P$ and $\partial_{1}Q$ to $\partial_{1}M$. 
The theorem below is essentially the same as the \textit{half Whitney trick} in \cite[p. 345]{GRW 14}, or the \textit{piping argument} of \cite[p. 40]{W 70}. 

\begin{theorem} \label{theorem: half whitney trick}
Let $P$, $Q$, $M$, $f$ and $g$ be as above and let $P$ be compact. 
Suppose that $n \geq 3$ and that the following conditions are satisfied:
\begin{enumerate} \itemsep.2cm
\item[(i)] the pairs $(P, \partial_{0}P)$ and $(Q, \partial_{0}Q)$ are both $0$-connected;
\item[(ii)] the pair $(M, \partial_{0}M)$ is $1$-connected;
\item[(iii)] $f(\partial_{1}P)\cap g(\partial_{1}Q) = \emptyset$. 
\end{enumerate}
Then there exists an isotopy of embeddings $f_{t}: (P, \partial_{0}P) \longrightarrow (M, \partial_{0}M)$ with $t \in [0, 1]$ and $f_{0} = f$, such that: 
\begin{itemize} \itemsep.2cm
\item $f_{t}|_{\partial_{1}P} = f|_{\partial_{1}P}$ for all $t \in [0,1]$; 
\item $f_{1}(P)\cap g(Q) = \emptyset$. 
\end{itemize}
\end{theorem}
\begin{remark}
We emphasize that the diffeotopy in the above lemma will in general
not be constant when restricted to the boundary of $M$.  If we were to
insist that $\Psi_{t}|_{\partial M} = \Id_{\partial M}$ for all $t \in
[0, 1]$ then the lemma would be false.  Indeed, the obstruction to the
existence of such a relative diffeotopy (relative meaning constant on
$\partial M$) is precisely the oriented algebraic intersection number
of $P$ and $Q$.  If we allow the diffeotopy to move the boundary, then
this intersection number ceases to be an obstruction to disjunction.
\end{remark}
\begin{proof}[Proof of Theorem \ref{theorem: half whitney trick}]
Through out the proof we will denote $P' := f(P)$ and $Q' := g(Q)$.
By putting $P'$ into general position we may assume that the intersection $P'\cap Q'$ is a set of points, all contained in the interior of $M$. 
Furthermore since $P$ is by assumption compact, it follows that the set $P'\cap Q'$ is actually finite.
It will suffice to prove the lemma in the case when the intersection contains a single point, say $P'\cap Q' = \{x\}$ for $x \in \Int(M)$. 
The general case follows by induction. 
We will apply the \textit{half Whitney trick} from see \cite[p. 345]{GRW 14} but with different notation. 
Recall the upper half-disk $D^{2}_{+} \subset \R^{2}_{+}$. 
The face $\partial_{1}D^{2}_{+}$ decomposes as the union $\partial_{1}D^{2}_{+} = \partial^{+}_{1}D^{2}_{+}\cup \partial^{-}_{1}D^{2}_{+}$, where 
$$
\partial^{+}_{1}D^{2}_{+} = \{\bar{x} \in \partial_{1}D^{2}_{+} \; | \; x_{2} \leq 0\}, \quad \partial^{+}_{1}D^{2}_{+} = \{\bar{x} \in \partial_{1}D^{2}_{+} \; | \; x_{2} \geq 0\}.
$$
Let $U \subset \R^{2}_{+}$ be an open neighborhood of $D^{2}_{+} \subset \R^{2}_{+}$.
Notice that we have $\partial_{0}D^{2}_{+} \subset \partial U$. 
We claim that there exists an embedding, which we call a \textit{half Whitney disk},
$\varphi: (U, \partial U) \longrightarrow (M, \partial M)$
that satisfies: 
$$\varphi^{-1}(P') = \partial^{+}_{1}D^{2}_{+}, \quad \varphi^{-1}(Q') = \partial^{-}_{1}D^{2}_{+}, \quad \varphi^{-1}(x) = \partial^{+}_{1}D^{2}_{+}\cap\partial^{-}_{1}D^{2}_{+}, \quad
\text{and} \quad \varphi^{-1}(\partial M) = \partial U.$$
The existence of such an embedding follows from the connectivity conditions put forth in the statement of the lemma: $(P, \partial P)$ and $(Q, \partial Q)$ are $0$-connected and $(M, \partial M)$ is $1$-connected
(it is also necessary that $\dim(U) < \dim(M)/2$ and that $\dim(U)$ be less than the codimensions of $P'$ and $Q'$ in $M$).
Using $\varphi$, we construct an embedding 
$$\tilde{\varphi}: (\R^{n-1}\times U\times\R^{n-1}, \; \; \R^{n-1}\times\partial U\times\R^{n-1}) \longrightarrow (M, \partial M)$$
whose restriction to $(\{0\}\times U\times\{0\}, \; \{0\}\times\partial U\times\{0\})$ is equal to $\varphi$ and satisfies:
\begin{itemize} \itemsep.2cm
\item $\tilde{\varphi}^{-1}(P') = \R^{n-1}\times\partial^{+}_{1}D^{2}_{+}\times\{0\}$;
\item $\tilde{\varphi}^{-1}(Q') = \{0\}\times\partial^{-}_{1}D^{2}_{+}\times\R^{n-1}$.
\end{itemize}
The extended embedding $\tilde{\varphi}$ is constructed in the same way as in \cite[Lemma 6.7]{M 65}
(the argument is easier in this case as we we are canceling intersection points against the boundary rather against each other, and so no framing obstructions arise). 

We now proceed to use the embedding $\tilde{\varphi}$ to define the
isotopy that pushes the submanifold $P'$ off of $Q'$.  This is done in
the same way as in the proof of \cite[Theorem 6.6 (page 74)]{M 65}.
The isotopy is constructed as follows.
First, choose a compactly supported vector field on $U$ that is equal
to $\partial/\partial x_{1}$ on the half-disk $D^{2}_{+} \subset U$.
Then extend this vector field to a compactly supported vector field,
denoted by $\mb{v}$, on $\R^{n-1}\times U\times\R^{n-1}$ using bump
functions.  Let $\mb{w}$ denote the vector field on $M$ obtained by
pushing forward $\mb{v}$ by the embedding $\tilde{\varphi}$ and then
extending over the complement
$M\setminus\tilde{\varphi}(\R^{n-1}\times U\times\R^{n-1})$
by the
zero vector field.  The flow associated to the vector field $\mb{w}$
yields an ambient isotopy of $M$ that pushes $P'$ off of $Q'$.  Since
the vector field $\mb{w}$ is by definition zero on
$M\setminus\tilde{\varphi}(\R^{n-1}\times U\times\R^{n-1})$, and hence
is zero on $\partial_{1}P'$, it follows that this ambient isotopy is
constantly the identity on $\partial_{1}P'$.  This concludes the proof
of the theorem.
\end{proof}

\subsection{A higher dimensional half Whitney-trick} \label{subsection: high dimensional half whitney trick}
In this section we will prove a higher-dimensional analogue of Theorem \ref{theorem: half whitney trick} from the previous section.
We now let $(M; \partial_{0}M, \partial_{1}M)$ be a $(2n+1)$-dimensional manifold triad
and let
$(P; \partial_{0}P, \partial_{1}P)$ and $(Q; \partial_{0}Q, \partial_{1}Q)$ be manifold triads of dimension $n+1$.
We consider embeddings 
\begin{equation}
f: (P, \partial_{0}P) \longrightarrow (M, \partial_{0}M) \quad \text{and} \quad g: (Q, \partial_{0}Q) \longrightarrow (M, \partial_{0}M),
\end{equation}
where as before we do not require that $f$ and $g$ send $\partial_{1}P$ and $\partial_{1}Q$ to $\partial_{1}M$. 
The main theorem of this section is the following.
\begin{theorem} \label{theorem: higher half whitney trick}
Let $P$, $Q$, $M$, $f$ and $g$ be as above and let $P$ be compact. 
Suppose that $n \geq 4$ and that the following conditions are satisfied:
\begin{enumerate} \itemsep.2cm
\item[(i)] the pairs $(P, \partial_{0}P)$ and $(Q, \partial_{0}Q)$ are both $1$-connected;
\item[(ii)] the pair $(M, \partial_{0}M)$ is $2$-connected;
\item[(iii)] $f(\partial_{1}P)\cap g(\partial_{1}Q) = \emptyset$. 
\end{enumerate}
Then there exists an isotopy of embeddings $f_{t}: (P, \partial_{0}P) \longrightarrow (M, \partial_{0}M)$ with $t \in [0, 1]$ and $f_{0} = f$, such that: 
\begin{itemize} \itemsep.2cm
\item $f_{t}|_{\partial_{1}P} = f|_{\partial_{1}P}$ for all $t \in [0,1]$; 
\item $f_{1}(P)\cap g(Q) = \emptyset$. 
\end{itemize}
\end{theorem}
Generically the intersection $f(P)\cap g(Q)$ is a $1$-dimensional manifold with boundary given by $f(\partial_{0}P)\cap g(\partial_{0}Q) \subset \partial_{0}M$.  
Furthermore, since $P$ is compact by assumption, it follows that $f(P)\cap g(Q)$ is a compact $1$-dimensional manifold and thus is a disjoint union of a finite number of compact intervals and circles. 
We will prove Theorem \ref{theorem: higher half whitney trick} in two steps. 
The first will be to prove the existence of the separating isotopy $f_{t}$ in the special case that $f(P)\cap g(Q)$ has no closed components, and hence is a finite disjoint union of compact intervals. 
The second step will be to show that for any such embeddings $f$ and $g$ as above, there exists an isotopy $f_{t}: (P, \partial_{0}P) \longrightarrow (M, \partial_{0}M)$ with $f_{0} = f$, such that $f_{1}(P)\cap g(Q)$ is the disjoint union of a  collection of compact intervals.
The following lemma is the first step in this process. 
\begin{lemma} \label{lemma: whitney trick no closed components}
Let conditions (i), (ii), and (iii) of the statement of Theorem \ref{theorem: higher half whitney trick} be satisfied. 
Suppose further that the intersection $f(P)\cap g(Q)$ has no closed components and thus is diffeomorphic to a disjoint union of compact intervals. 
Then there exists an isotopy of embeddings 
$$f_{t}: (P, \partial_{0}P) \longrightarrow (M, \partial_{0}M)$$ 
with $f_{0} = f$ such that:
\begin{itemize} \itemsep.2cm
\item $f_{t}|_{\partial_{1}P} = f|_{\partial_{1}P}$ for all $t \in [0,1]$;
\item $f_{1}(P)\cap g(Q) = \emptyset$. 
\end{itemize}
\end{lemma}

\begin{proof}
Let us denote the by $P'$ and $Q'$ the images $f(P)$ and $g(Q)$ respectively. 
Similarly, we denote $\partial_{i}P' = f(\partial_{i}P)$ and $\partial_{i}Q' = g(\partial_{i}Q)$. 
We will prove the lemma explicitly in the special case that $P'\cap Q'$ is diffeomorphic to the single compact interval $[0, 1]$. 
The general case follows by induction. 
Our construction is a simple generalization of the \textit{half Whitney trick} from Theorem \ref{theorem: half whitney trick}. 
Recall the upper half-disk $D^{3}_{+} \subset \R^{3}_{+}$. 
The face $\partial_{1}D^{3}_{+}$ decomposes as the union 
$$\partial_{1}D^{3}_{+} = \partial^{+}_{1}D^{3}_{+}\cup \partial^{-}_{1}D^{3}_{+},$$ 
where 
$$
\partial^{-}_{1}D^{3}_{+} = \{\bar{x} \in \partial_{1}D^{3}_{+} \; | \; x_{2} \leq 0\}, \quad \partial^{+}_{1}D^{3}_{+} = \{\bar{x} \in \partial_{1}D^{3}_{+} \; | \; x_{2} \geq 0\}.
$$
Let $U \subset \R^{3}_{+}$ be an open neighborhood of $D^{3}_{+} \subset \R^{3}_{+}$.
Notice that we have $\partial_{0}D^{3}_{+} \subset \partial U$. 
\begin{claim} \label{claim: existence of whitney disk}
There exists an embedding
$\varphi: (U, \partial U) \longrightarrow (M, \partial_{0}M)$
that satisfies: 
\begin{enumerate} \itemsep.2cm
\item[(i)] $\varphi^{-1}(P'\cap Q') = \partial^{+}_{1}D^{3}_{+}\cap\partial^{-}_{1}D^{3}_{+}$;
\item[(ii)] $\varphi^{-1}(P') = \partial^{+}_{1}D^{3}_{+}$;
\item[(iii)] $\varphi^{-1}(Q') = \partial^{-}_{1}D^{3}_{+}$;
\item[(iv)] $\varphi^{-1}(\partial_{0}M) = \partial U$.
\end{enumerate}
\end{claim}
We the above claim follows from the connectivity conditions in the statement of Theorem \ref{theorem: higher half whitney trick} together with the fact that $n \geq 4$. 
We will postpone the proof of this claim until after the proof of the current lemma. 
Assuming this claim the remainder of the proof proceeds in essentially the same way as the proof of Theorem \ref{theorem: half whitney trick}. 
Indeed, let $\varphi: (U, \partial U) \longrightarrow (M, \partial_{0}M)$ be such an embedding as in Claim \ref{claim: existence of whitney disk}. 
We now need to extend $\varphi$ to an embedding 
$$\tilde{\varphi}: (\R^{n-2}\times U\times\R^{n-2}, \; \; \R^{n-2}\times\partial U\times\R^{n-2}) \longrightarrow (M, \partial_{0}M)$$
that satisfies:
$$
\tilde{\varphi}^{-1}(P') = \R^{n-2}\times\partial^{+}_{1}D^{3}_{+}\times\{0\}, \quad \quad \tilde{\varphi}^{-1}(Q') = \{0\}\times\partial^{-}_{1}D^{3}_{+}\times\R^{n-2}.
$$
This extension is constructed in the same way as in \cite[Lemma 6.7]{M 65}.
As in the proof of Lemma Theorem \ref{theorem: half whitney trick}, since the extension is over a half-disk, there is no obstruction to the construction of such an extension. 
With this new embedding $\tilde{\varphi}$ constructed we may define the isotopy that pushes the submanifold $P'$ off of $Q'$ in the same way as in the proof of Theorem \ref{theorem: half whitney trick}. 
This completes the proof of the lemma assuming Claim \ref{claim: existence of whitney disk}.
\end{proof}

\begin{proof}[Proof of Claim \ref{claim: existence of whitney disk}]
We will construct the embedding $\varphi: (U, \partial U) \longrightarrow (M, \partial_{0}M)$ in stages.
By assumption, the intersection $P'\cap Q'$ is diffeomorphic to the interval $[0, 1]$ with boundary $\partial(P'\cap Q') = \partial_{0}P'\cap\partial_{0}Q'$ contained in $\partial M$. 
The inclusion maps 
$$\left(P'\cap Q',\; \partial(P'\cap Q')\right) \; \hookrightarrow \; (P', \partial P') \quad \text{and} \quad \left(P'\cap Q',\; \partial(P'\cap Q')\right) \; \hookrightarrow \; (Q', \partial Q')$$ 
determine classes in $\pi_{1}(P', \partial_{0}P')$ and $\pi_{1}(Q', \partial_{0}Q')$. 
Since $\pi_{1}(P', \partial_{0}P') = \pi_{1}(Q', \partial_{0}Q) = 0$, it follows that there exists embeddings
$$
j_{+}: (D^{2}_{+}, \partial_{0}D^{2}_{+}) \longrightarrow (P', \partial P'), \quad j_{-}: (D^{2}_{+}, \partial_{0}D^{2}_{+}) \longrightarrow (Q', \partial Q'),
$$
such that 
$$
j_{+}^{-1}(P'\cap Q') \; = \; \partial_{1}D^{2}_{+} \; = \; j_{-}^{-1}(P'\cap Q').
$$
(By the condition $n \geq 4$ we may assume that these maps are actually embeddings because $\dim(P') = \dim(Q') \geq 5$ and thus $\dim(D^{2}_{+})$ is strictly less than half the dimension of $P'$ and $Q'$.) 
Combining these embeddings and identifying $D^{2}_{+} \cong \partial^{\pm}_{1}D^{3}_{+}$, we obtain an embedding
$$
j: \partial_{1}D^{3}_{+} \; \longrightarrow \; M
$$
with
$$
j^{-1}(P') = \partial^{+}_{1}D^{3}_{+}, \quad \quad j^{-1}(Q') = \partial^{-}_{1}D^{3}_{+}, \quad \quad j^{-1}(\partial_{0}M) = \partial_{0, 1}D^{3}_{+}.
$$
Since $j(\partial_{0, 1}D^{3}_{+}) \subset \partial_{0}M$, the embedding $j$ determines a class in $\pi_{2}(M, \partial_{0}M)$. 
Since $\pi_{2}(M, \partial_{0}M) = 0$, the embedding $j$ extends to a map 
$$
\widehat{j}: (D^{3}_{+}, \partial_{0}D^{3}_{+}) \longrightarrow (M, \partial_{0}M)
$$
which we may assume is an embedding since $3 < 9/2 \leq \dim(M)/2$. 
By putting this embedding into general position we may further assume that the interior of $\widehat{j}(D^{3}_{+})$ is disjoint from $P'$ and $Q'$ (and $\partial_{1}M$). 
The desired embedding $\varphi: (U, \partial U) \longrightarrow (M, \partial_{0}M)$ is then constructed by extending $\widehat{j}$ using a ``small'' non-vanishing, normal vector-field along $\widehat{j}(\partial_{1}D^{3}_{+})$.
This completes the proof of Claim \ref{claim: existence of whitney disk}. 
\end{proof}

We now go on to the second component of the proof of Theorem \ref{theorem: higher half whitney trick}. 
\begin{lemma}
Let $P$, $Q$, $M$, $f$ and $g$ be as in the statement of Theorem \ref{theorem: higher half whitney trick}. 
Then there exists an isotopy $f_{t}: (P, \partial_{0}P) \longrightarrow (M, \partial_{0}M)$ with $f_{0} = f$ such that:
\begin{itemize} \itemsep.2cm
\item $f_{t}|_{\partial_{1}P} = f_{\partial_{1}P}$ for all $t \in [0,1]$;
\item $f_{1}(P)\cap g(Q)$ has no closed components, i.e. it is a disjoint union of compact intervals.
\end{itemize}
\end{lemma}
\begin{proof}
We will again let $P'$ and $Q'$ denote the images $f(P)$ and $g(Q)$. 
We will prove the lemma in the case that $P'\cap Q'$ is diffeomorphic to the circle $S^{1}$. 
The general case follows by induction. 
So, suppose that $P'\cap Q'$ is diffeomorphic to $S^{1}$.
We will show how to deform $f$, via an isotopy fixed on $\partial_{1}P$, to a new embedding $\hat{f}$ so that $\hat{f}(P)\cap Q$ is diffeomorphic to a single compact interval. 

Recall the upper half-disk $D^{2}_{+} \subset \R^{2}_{+}$. 
The face $\partial_{1}D^{2}_{+}$ decomposes as the union 
$$\partial_{1}D^{2}_{+} = \partial^{+}_{1}D^{2}_{+}\cup \partial^{-}_{1}D^{2}_{+},$$
where 
$$
\partial^{+}_{1}D^{2}_{+} = \{\bar{x} \in \partial_{1}D^{2}_{+} \; | \; x_{2} \leq 0\}, \quad \partial^{+}_{1}D^{2}_{+} = \{\bar{x} \in \partial_{1}D^{2}_{+} \; | \; x_{2} \geq 0\}.
$$
Let $U \subset \R^{2}_{+}$ be an open neighborhood of $D^{2}_{+} \subset \R^{2}_{+}$.
Notice that we have $\partial_{0}D^{2}_{+} \subset \partial U$. 
We construct an embedding 
$$
\varphi: (U, \partial U) \; \longrightarrow \; (M, \partial_{0}M)
$$
that satisfies the following properties:
\begin{itemize} \itemsep.2cm
\item $\left(\varphi^{-1}(P'), \; \varphi^{-1}(\partial_{0}P')\right) \; = \; \left(\partial^{-}_{1}D^{2}_{+}, \; \partial^{-}_{0,1}D^{2}_{+}\right)$;
\item $\left(\varphi^{-1}(Q'), \; \varphi^{-1}(\partial_{0}Q')\right) \; = \; \left(\partial^{+}_{1}D^{2}_{+}, \; \partial^{+}_{0,1}D^{2}_{+}\right)$;
\item $\varphi^{-1}(P'\cap Q') = \partial^{+}_{1}D^{2}_{+}\cap\partial^{-}_{1}D^{2}_{+}$;
\item $\varphi^{-1}(\partial_{0}M) = \partial_{0}D^{2}_{+}$;
\item $\varphi^{-1}(\partial_{1}P'\cup\partial_{1}Q') = \emptyset$.
\end{itemize}
The existence of this embedding follows easily from the fact that both $(P, \partial_{0}P)$ and $(Q, \partial_{0}Q)$ are $1$-connected and that $(M, \partial_{0}M)$ is $2$-connected. 
We now need to extend the embedding $\varphi$ to an embedding 
$$
\widehat{\varphi}: \R^{n-1}\times\R\times\R^{n-1}\times U \; \longrightarrow \; M
$$
that satisfies the following properties: 
\begin{enumerate} \itemsep.2cm
\item[(i)] $\widehat{\varphi}|_{\{0\}\times U} = \varphi$;
\item[(ii)] $\widehat{\varphi}^{-1}(\partial_{0}M) = \R^{n-1}\times\R\times\R^{n-1}\times\partial U$;
\item[(iii)] $\widehat{\varphi}^{-1}(P') = \R^{n-1}\times\R\times\{0\}\times\partial_{1}^{-}D^{2}_{+}$;
\item[(iv)] $\widehat{\varphi}^{-1}(Q') = \{0\}\times\R\times\R^{n-1}\times\partial_{1}^{+}D^{2}_{+}.$
\end{enumerate}

We now proceed to use the embedding $\tilde{\varphi}$ to define the isotopy that deforms the intersection $P'\cap Q'$ into a compact interval.  
The isotopy that we use is constructed in the same way as the one used in proof of Theorem \ref{theorem: half whitney trick} or \cite[Theorem 6.6 (page 74)]{M 65}. 
Recall that we choose a compactly supported vector field on $U$ that is equal to $\partial/\partial x_{1}$ on the half-disk $D^{2}_{+} \subset U$. 
Extend this vector field to a compactly supported vector field, denoted by $\mb{v}$, on $\R^{n-1}\times\R\times\R^{n-1}\times U$ using bump functions. 
Let $\mb{w}$ denote the vector field on $M$ obtained by pushing forward $\mb{v}$ by the embedding $\tilde{\varphi}$ and then extending over the complement of $\tilde{\varphi}(\R^{n-1}\times\R\times\R^{n-1}\times U)$ by the zero vector field. 
It can be easily verified that this isotopy alters the intersection of $P'$ with $Q'$ in the desired way. 
Since the vector field $\mb{w}$ is by definition zero on $M\setminus\tilde{\varphi}(\R^{n-1}\times\R\times\R^{n-1}\times U)$, and hence is zero on $\partial_{1}P'$, it follows that this ambient isotopy is constantly the identity on $\partial_{1}P'$. 
This concludes the proof of the theorem. 
\end{proof}

\subsection{Embeddings of bounded manifolds} \label{subsection: embeddings of manifolds with boundary}
In addition to removing the intersections between embeddings, 
we will also need a technique for the producing of such embeddings. 
The following result is a restatement of \cite[Theorem 1]{H 69}.
\begin{theorem}[Hudson '72] \label{theorem: hudsen embedding theorem}
Let $(M; \partial_{0}M, \partial_{1}M)$ and $(P; \partial_{0}P, \partial_{1}P)$ be manifold triads of dimension $m$ and $p$ respectively. 
Let 
\begin{equation} \label{equation: map to make embedding}
f: (P; \partial_{0}P, \partial_{1}P) \longrightarrow (M; \partial_{0}M, \partial_{1}M)
\end{equation}
be a maps of triads such that the restriction $f|_{\partial_{1}P}: \partial_{1}P \longrightarrow \partial_{1}M$ is an embedding with $f^{-1}(\partial_{1}M) = \partial_{1}P$.
Suppose that the following conditions are satisfied:
\begin{enumerate}  \itemsep.2cm
\item[(i)] $m - p \geq 3$;
\item[(ii)] $(P, \partial_{0}P)$ is $(2p - m)$-connected;
\item[(iii)] $(M, \partial_{0}M)$ is $(2p - m +1)$-connected. 
\end{enumerate}
Then there exists a homotopy $f_{t}: (P; \partial_{0}P, \partial_{1}P) \longrightarrow (M; \partial_{0}M, \partial_{1}M)$ with $t \in [0, 1]$ and $f_{0} = f$ such that:
\begin{itemize} \itemsep.2cm
\item $f_{t}|_{\partial_{1}P} = f|_{\partial_{1}P}$ for all $t \in [0,1]$;
\item $f_{1}$ is an embedding with $f^{-1}(\partial_{i}M) = \partial_{i}P$ for $i = 0,1$.
\end{itemize}
\end{theorem}
We will apply the above theorem in the case when $m = 2n+1$ and $p = n+1$ for $n \geq 3$.
The main place in this paper where we use this theorem is in the proof of Lemma \ref{proposition: nonempty surgery data}.

\end{document}